\documentclass{article}
\usepackage{mathrsfs}
\usepackage{a4wide}
\usepackage{amsmath,amsfonts,amssymb,amsthm,color}

\usepackage[active]{srcltx}

\allowdisplaybreaks

\newcommand\numberthis{\addtocounter{equation}{1}\tag{\theequation}}

\def \beq{\begin{equation}}
\def \eeq{\end{equation}}

\def\V{\mathcal{V}}

\newcommand{\R}{{\mathbb R}}
\newcommand{\T}{\textbf{T}}

\newcommand{\X}{{\mathcal{X}}}
\newcommand{\z}{\text{\tt{z}}}
\def\bb1{{1\hspace*{-2.4pt}\rm{l}}}
\newcommand {\curl} {{\rm curl}\,}
\newcommand{\lek}{\widetilde{\lambda}^{\epsilon,\kappa}_{\gamma^*}}

\newtheorem{theorem}{Theorem}[section]
\newtheorem{definition}[theorem]{Definition}
\newtheorem{proposition}[theorem]{Proposition}
\newtheorem{corollary}[theorem]{Corollary}
\newtheorem{lemma}[theorem]{Lemma}

\theoremstyle{definition}
\newtheorem{remark}[theorem]{Remark}

\newtheorem{hypothesis}[theorem]{Hypothesis}

\numberwithin{equation}{section}

\begin{document}

\noindent 
\begin{center}
\textbf{\large Low lying spectral gaps induced by slowly varying magnetic fields}\\ 
\vspace{0,3cm}
\end{center}

\begin{center}
{April 5, 2017}
\end{center}

\vspace{0.2cm}

\begin{center}
{\bf
Horia D. Cornean\footnote{Department of Mathematical Sciences,
  Aalborg University, Fredrik Bajers Vej 7G, 9220 Aalborg, Denmark}, \ 
  Bernard Helffer\footnote{Laboratoire
de Math\'ematiques Jean Leray, Universit\'{e} de Nantes, France and Laboratoire de Math\'ematiques d'Orsay, Univ.
Paris-Sud, CNRS, Universit\'e Paris-Saclay, 91405 Orsay, France.}  and Radu
Purice\footnote{Institute
of Mathematics Simion Stoilow of the Romanian Academy, P.O.  Box
1-764, Bucharest, RO-70700, Romania.}}
\end{center}

\begin{abstract}
We consider a periodic Schr\"odinger operator in two dimensions perturbed by a weak magnetic field whose intensity slowly varies around a positive mean. We show in great generality that the bottom of the spectrum of the corresponding magnetic Schr\"odinger operator develops spectral islands  separated by gaps, reminding of a Landau-level structure. 

First, we construct an effective {Hofstadter-like} magnetic matrix which accurately describes the low lying spectrum of the full operator. The construction of this effective magnetic matrix does not require a gap in the spectrum of the non-magnetic operator, only that the first and the second Bloch eigenvalues do not cross but their ranges might overlap.  The crossing case is more difficult and will be considered elsewhere. 

Second, we perform a detailed spectral analysis of the effective matrix using a
gauge-covariant magnetic pseudo-differential calculus adapted to slowly varying
magnetic fields. {As an application, we prove in the overlapping case the
appearance of spectral islands separated by gaps.}
\end{abstract}


\section{Introduction}

In this paper we analyze  the gap structure appearing at the bottom of the spectrum of a
two dimensional periodic Hamiltonian which is  
perturbed by a magnetic field that is neither
constant, nor vanishing at infinity, but which is supposed to have {\it 'weak
variation'}, in a sense made precise in Eq.\eqref{Bek} below. Our main
purpose is to show the appearance of a structure of narrow spectral islands separated by open spectral gaps. We shall
also investigate how the size of these spectral objects varies with
the smallness and the weak variation of the magnetic field.

We therefore contribute to the mathematical understanding of the so called 
Peierls substitution at weak magnetic fields \cite{Pe}; this problem has been mathematically analyzed by various authors
(Buslaev \cite{Bus}, Bellissard \cite{B, B2}, Nenciu \cite{Ne-LMP,Ne-RMP}, Helffer-Sj\"ostrand \cite{HS, Sj}, Panati-Spohn-Teufel \cite{PST}, Freund-Teufel \cite{FrTe}, De Nittis-Lein \cite{dNL}, 
Cornean-Iftimie-Purice \cite{IP14,CIP}) in order to investigate the validity domain of various models developed by physicists like Kohn \cite{Koh} and Luttinger \cite{Lu}. An exhaustive discussion on the physical background of this problem can be found in \cite{Ne-RMP}.

\subsection{Preliminaries}

On the configuration space $\X:=\mathbb{R}^2$ we consider a lattice
$\Gamma\subset\X$ generated by two linearly independent vectors
$\{e_1,e_2\}\subset\X$, and we also consider a smooth, $\Gamma$-periodic potential
$V:\X\rightarrow\mathbb{R}\,$. The dual lattice of $\Gamma$ is defined as 
$$
\Gamma_*\,:=\,\left\{\gamma^*\in\X^*{=\R^2}\,\mid\,\langle\gamma^*,\gamma\rangle
/(2\pi)\in \mathbb{Z}\,
,\ \forall\gamma\in\Gamma\right\}.
$$
Let us fix an \textit{elementary cell}:
$$
E\,:=\,\Big\{y=\sum\limits_{j=1}^2t_je_j\in\mathbb{R}^2\,\mid\,
-1/2\leq t_j<1/2\,,\
\forall j\in\{1,2\}\Big\}\,.
$$
We consider the quotient group $\X/\Gamma$ that is canonically isomorphic to
the $2$-dimensional torus $\mathbb{T}$.
The dual basis $\{e^*_1,e^*_2\}\subset\X^*$ is defined by
$\langle e^*_j,e_k\rangle =(2\pi)\delta_{jk}\,$, and we have
${\Gamma_*=\oplus_{j=1}^2\mathbb{Z}e^*_j}\,
$.
 We define
$\mathbb{T}_*:=\X^*/\Gamma_*$ and
$E_*$ by 
$$
E_*\,:=\,\Big\{\theta =\sum\limits_{j=1}^2t_je_j^*\in\mathbb{R}^2\,\mid\,
-1/2\leq t_j<1/2\,,\
\forall j\in\{1,2\}\Big\}\,.
$$

Consider the differential operator $-\Delta+V$, which is essentially self-adjoint on the 
Schwartz set $\mathscr{S}(\X)$. Denote by $H^0$ its self-adjoint extension in $\mathcal{H}:=L^2(\X)$.
The map
\beq\label{Floquet-1}
\big(\mathscr{V}_\Gamma\varphi\big)(\theta,x)\ :=\
|E|^{-1/2}\underset{\gamma\in\Gamma}{\sum}e^{-i<\theta,x-\gamma>}\varphi(x-\gamma)\,,
\quad\forall \,{(x,\theta)}\in\X\times E_*,\,\forall \, \varphi\in\mathscr{S}(\X)\,,
\eeq
(where $|E|$ is the Lebesgue measure of the elementary cell $E$) induces a unitary operator
$\mathscr{V}_\Gamma$ from $L^2(\X)$ onto $ L^2\big(E_*;L^2(\mathbb{T})\big)$. 
Its inverse $\mathscr{V}_\Gamma^{-1}$ is given by:
\begin{equation}\label{inv}
(\mathscr{V}_\Gamma^{-1}  \psi) (x) =|E_*|^{-\frac 12}  \int_{E_*} e^{i < \theta,x> } \psi(\theta,x) \, d \theta\,.
\end{equation}

 We know from the Bloch-Floquet theory
(see for example \cite{RS-4})  the following:\\

\noindent 1. We have a fibered structure: 
\begin{equation}\label{fibop}
\hat{H}^0:=\mathscr{V}_\Gamma H^0\mathscr{V}_\Gamma^{-1}=
\int_{E_*}^\oplus\hat{H}^0(\theta)d\theta,\quad 
{\hat{H}^0(\theta):=\big(-i\nabla-\theta\big)^2+V}\; {\rm in}\; L^2(\mathbb{T})\,.
\end{equation}

\noindent 2.  The map $E_*\ni\theta\mapsto\hat{H}^0(\theta)$ has an extension to $\X^*$ given by $$\hat{H}^0(\theta+\gamma^*)=e^{i<\gamma^*,\cdot >}
 \hat{H}^0(\theta)e^{-i<\gamma^*,\cdot >},
$$
and it is analytic in the norm resolvent topology.

\noindent 3.
There exists a family of continuous functions
$E_*\ni\theta\mapsto\lambda_j(\theta)\in\mathbb{R}$  with periodic continuous extensions to $\X^*\supset E_*$, indexed by
$j\in\mathbb{N}$
such that $\lambda_j(\theta)\leq\lambda_{j+1}(\theta)$ for every
$j\in\mathbb{N}$ and $\theta\in E_*$, and
$$
\sigma\big(\hat{H}^0(\theta)\big)=\underset{j\in\mathbb{N}}{\bigcup}\{
\lambda_j(\theta)\}.
$$

\noindent 4. There exists an orthonormal family of measurable eigenfunctions
$
E_*\ni\theta\mapsto\hat{\phi}_j(\theta,\cdot)\in L^2(\mathbb{T})$, $j\in\mathbb{N}$, such that $\|\hat{\phi}_j(\theta,\cdot)\|_{L^2(\mathbb{T})}=1$ and
$
\hat{H}^0(\theta)\hat{\phi}_j(\theta,\cdot)=\lambda_j(\theta)\hat{\phi}_j(\theta,\cdot)\,.$

\medskip

\begin{remark}\label{Rem-lambda-0}
It was proved in \cite{KS} that $\lambda_0(\theta)$ is always simple in a neighborhood of $\theta=0$ and has a nondegenerate global minimum on $E_*$ at $\theta=0$, minimum which we may take equal to zero (up to a shift in energy). For the convenience of the reader, we include a short proof of these facts in Appendix \ref{A-lambda0}.
\end{remark}

We shall also need one of the following two conditions.

\begin{hypothesis}\label{Hyp-1b} {\it Either $ \sup (\lambda_0)<\inf(\lambda_1)$, i.e. a non-crossing condition with a gap},
\end{hypothesis}

{\it or}

\begin{hypothesis}\label{Hyp-1a} {\it The eigenvalue $\lambda_0(\theta)$ remains simple for all  $\theta\in\mathbb{T}_*$, but $ \sup(\lambda_0) \geq\inf(\lambda_1)$, i.e. a non-crossing condition with range overlapping and no gap.} 
\end{hypothesis}

Because $H^0$ has a real symbol, we have $\overline{\hat{H}^0(\theta)}= \hat{H}^0(-\theta)\,$. Also, since $\lambda_0(\cdot)$ is simple, it must be an even function  
\begin{equation}\label{evenness}
\lambda_0(\theta) = \lambda_0(-\theta)\,.
\end{equation}

\subsection{The main result}

Now let us consider the magnetic field perturbation, which is a $2$-parameter family of magnetic fields
\beq\label{Bek}
B_{\epsilon,\kappa}(x)\,:=\,\epsilon B_0\,+\,\kappa\epsilon B(\epsilon x)\,,
\eeq
indexed by $(\epsilon,\kappa)\in[0,1]\times[0,1]\,$. \\
Here $B_0>0$ is constant, while  
$B:\X\rightarrow\mathbb{R}$ is smooth and bounded together with all its derivatives.  
 Let us choose some smooth  {\it vector potentials} $A^0:\X\rightarrow\X$ and
$A:\X\rightarrow\X$ such that:
\begin{equation}\label{defAepsilon0}
B_0=\partial_1A^0_2-\partial_2A^0_1\,, B=\partial_1A_2-\partial_2A_1\,,
\end{equation}
and 
\begin{equation}\label{defAepsilon}
A^{\epsilon,\kappa}(x):=\epsilon A^0(x)+\kappa A(\epsilon x)\,, \, B_{\epsilon,\kappa}=\partial_1A^{\epsilon,\kappa}_2-\partial_2A^{\epsilon,
\kappa}_1\,.
\end{equation}
The vector potential $A^0$ is always in the {\it transverse gauge}, i.e.
\beq \label{defA0}
A^0(x)\,=\,(B_0/2)\big(-x_2,x_1\big).
\eeq
We consider the following magnetic Schr\"{o}dinger operator, essentially self-adjoint on $\mathscr{S}(\X)$:
\begin{equation}\label{mainH}
H^{\epsilon,\kappa}:=(-i\partial_{x_1} - A^{\epsilon,\kappa}_1)^2 + (-i\partial_{x_2} - A^{\epsilon,\kappa}_2)^2 + V\,.
\end{equation} 
When $\kappa=\epsilon=0$ we recover
 the periodic Schr\"odinger  Hamiltonian without magnetic
field $H^{0}$.\\

Our main goal is to prove  that for $\epsilon$ and $\kappa$ small enough, the bottom of the spectrum of 
$H^{\epsilon,\kappa}$ develops gaps of width of order $\epsilon$ separated by spectral islands 
(non-empty but not necessarily connected) whose width is slightly smaller than $\epsilon$, see below for the precise statement.

\begin{theorem}  \label{mainTh} Consider either Hypothesis \ref{Hyp-1b} or Hypothesis \ref{Hyp-1a}. 
Fix an integer $N>1$. Then there exist some constants  $C_0,C_1,C_2 >0\,$, and 
$\epsilon_0,\kappa_0\in(0,1)\,$,  such
that for
any $\kappa\in(0,\kappa_0]$ and $\epsilon\in(0,\epsilon_0]$ there exist $a_0<b_0<a_1<\cdots <a_N<b_N$ 
with $a_0= \inf\{\sigma(H^{\epsilon,\kappa})\}$ so that:
\begin{align}
&\sigma(H^{\epsilon,\kappa})\cap \left [a_0,b_N\right ]\subset \bigcup_{k=0}^N [a_k,b_k]\,,\quad{\rm dim}
\big({\rm Ran} E_{[a_k,b_k]}(H^{\epsilon,\kappa}) \big )= + \infty\,,\nonumber \\
& b_k-a_k\leq C_0 \,\kappa\epsilon + C_1\, \epsilon^{4/3}\, , \; 0\leq k\leq N\,,\quad {\rm and}\quad  a_{k+1}-b_k\geq\frac{1}{C_2}  \epsilon \, ,\; 0\leq k\leq N-1\,.\label{feb-19}
\end{align}
Moreover, given any compact set $K\subset\mathbb{R}$, there exists $C >0\,$, such that, for $(\kappa,\epsilon) \in [0,1]\times[0,1]\,,$ we have (here ${\rm dist}_H$ means Hausdorff distance):
\beq\label{febr-20}
{\rm dist}_H \big (\sigma (H^{\epsilon,\kappa}) \cap K, \sigma (H^{\epsilon,0}) \cap K \big )\leq C \sqrt{\kappa \epsilon}\,.
\eeq
\end{theorem}

\medskip

\begin{remark}
Exactly one of the two non-crossing conditions described in Hypothesis \ref{Hyp-1b} and \ref{Hyp-1a} is generically satisfied when the potential $V$ does not have any special symmetries. The crossing case will be considered elsewhere.
\end{remark}

\begin{remark} A natural conjecture is that the spectrum is close (say modulo $o (\epsilon)$) to the union of the spectra of the Schr\"odinger operators with a constant magnetic field $ \epsilon B_0 + \epsilon \kappa \beta$ with $\beta$ in the range of $B$. This leads us to the conjecture that the best $C_0$ in \eqref{feb-19} could be the variation of $B$, i.e. $\sup B -\inf B$. Anyhow, from \eqref{feb-19} we see that a condition for the appearance of gaps is $\kappa<1/(C_0 C_2)$. 
\end{remark}

{
\subsection{More on the state of the art}

This paper is dedicated to the proof of \eqref{feb-19}. The estimate \eqref{febr-20} is a direct consequence of the results of \cite{CP-1}, but we included it here in order to make a comparison with previous results obtained for constant magnetic fields (i.e. $\kappa=0$), where the values of the $a_k$'s and $b_k$'s from \eqref{feb-19} are found with much better accuracy, see Theorem \ref{partTha}.

To the best of our knowledge, in the overlapping case (see Hypothesis
\ref{Hyp-1a}) there are no previous results about the appearance of spectral
gaps of order $\epsilon$ at the bottom of the spectrum of magnetic
Bloch-Schr\"odinger operators, not even in the constant magnetic field case
($\kappa=0$). Our proof treats both the gapped and the overlapping case on an
equal footing. Also, as we will explain below, our approach is not just a
natural generalization of the constant magnetic field case with a gap; on the
contrary. See the next Section for a detailed discussion of the constant field
case with a gap. 

Our first important new achievement is showing that the spectrum of the
Hofstadter-like magnetic matrix constructed in Subsection 3.2 (acting on
$\ell^2(\Gamma)$) is at a Hausdorff distance of order $\epsilon^2$ from the
bottom of the spectrum of the full magnetic Schr\"odinger operator
$H^{\epsilon,\kappa}$. The proof of this fact relies in an essential way on the
magnetic pseudodifferential calculus. Moreover, in Subsection 3.3 we show that
the spectrum of this effective magnetic matrix is at a Hausdorff distance of
order $\epsilon\kappa$ from the spectrum of a magnetic pseudodifferential
operator with an $\epsilon$-dependent periodic symbol acting on $L^2(\R^2)$. We
note that in the case of a constant magnetic field perturbation ($\kappa=0$),
these two effective objects would be isospectral with an one dimensional
$\epsilon$-pseudodifferential operator (see \eqref{Opepsilon}). If $\kappa\neq
0$, the reduction to just one dimension is no longer possible.  

Our second main achievement is a careful and detailed spectral analysis of the
above mentioned effective magnetic pseudodifferential operator.
One  major complication in the non-constant magnetic field case is that we need
to deal with two dimensional Landau-type comparison operators, unlike the
constant magnetic field case where the comparison operators were one dimensional
harmonic oscillators. One of the most difficult technical issues which we
encounter here is that we need to show that the $\epsilon$-spaced mini-gaps of
the comparison operator do not close down when the slowly variable non-constant
field perturbation (of order  $\epsilon \kappa$) is different from zero. In
other words, the perturbation is of the same order of magnitude as the size of
the mini-gaps, which demands Lipschitz continuity of the inner mini-gap
boundaries with respect to the magnetic perturbation. See \cite{B} for a
thorough discussion of why this is not true in general.

In order to control the various error terms appearing in this case, we had to
refine the already existent magnetic pseudodifferential calculus in order to
deal with slowly variable magnetic fields. As far as we can tell, our calculus
cannot be directly compared with the one developed by Lein and de Nittis in
\cite{dNL}; moreover these authors are not controlling how the errors behave as
a function of the spectral parameter $z$ and this control is crucial in 
our proof.
 }

\subsection{Main steps in the proof}

Some properties of the bottom of the spectrum of the non-magnetic periodic operator can be found in Appendix \ref{A-lambda0}. The magnetic quantization is summarized in Appendix \ref{A-sl-var-symb}. 
Given a magnetic field $B$  of class $BC^\infty(\X)$ (i.e. bounded together with all its derivatives) 
and a choice of a vector potential $A$ for it (i.e. 
such that $\curl A =B$), one can define a {\it  twisted pseudo-differential
calculus} (see \cite{MP1}), that generalizes the standard  Weyl calculus 
and associates with any  symbol $F\in S^m_\rho(\Xi)$ the following operator in
$\mathcal{H}$ (for all $u\in\mathscr{S}(\X)$ and $x\in\X$):
\beq\label{OpA}
\big(\mathfrak{Op}^A(F)u\big)(x)\ :=\
(2\pi)^{-2}\int_\X\int_{\X^*}e^{i\langle
\xi,x-y\rangle}e^{-i\int_{[x,y]}A}\, F\left(\frac{x+y}
{2},\xi\right )u(y)\,d\xi\,dy\,,
\eeq
 where $\int_{[x,y]}A$ is the integral over the oriented segment $[x,y|$ of the $1$-form associated with $A$.\\
For  $F(x,\xi)= \xi^2 + V (x)\,$, $\mathfrak {Op}^A (F)$ is the usual Schr\"odinger operator in \eqref{mainH}.

\medskip

\noindent For the convenience of the reader, we now list the  main ideas appearing in the proof.

\medskip

\noindent {\bf Step 1: Construction of an effective magnetic matrix, see Subsections \ref{SS-W-func} and \ref{SS-inv-magn-band}}. Because $\lambda_0(\theta)$ is always  assumed to be isolated, one can associate with it an orthonormal projection $\pi_0$ which commutes with $H^0$. Note that $\pi_0$ might {\bf not} be a spectral projection for $H^0$, unless there is a gap between the first band and the rest. But in both cases, the range of $\pi_0$ has a basis consisting of  exponentially localized Wannier functions \cite{CHN, Ne-RMP, FMP}. 
When $\epsilon$ and $\kappa$ are small enough, we can construct an orthogonal system of exponentially localized {\it almost Wannier functions} starting from the unperturbed Wannier basis, and show that the corresponding orthogonal projection $\pi^{\epsilon,\kappa}_0$ is {\bf almost invariant} for $H^{\epsilon,\kappa}$. Note that in the  case with a gap, $\pi^{\epsilon,\kappa}_0$ can be constructed to be a true spectral projection for $H^{\epsilon,\kappa}$. Next, using a Feshbach-type argument, we prove that the low lying spectrum of $H^{\epsilon,\kappa}$ is at a Hausdorff distance of order $\epsilon^2$ from the spectrum of the reduced operator $\pi^{\epsilon,\kappa}_0H^{\epsilon,\kappa}\pi^{\epsilon,\kappa}_0$. 

In the basis of magnetic almost Wannier functions, the  reduced operator $\pi^{\epsilon,\kappa}_0H^{\epsilon,\kappa}\pi^{\epsilon,\kappa}_0$ defines an {\it effective magnetic matrix}   acting on  $ \ell^2(\Gamma)$. Hence if this magnetic matrix has spectral gaps of order $\epsilon$, the same holds true for the bottom of the spectrum of $H^{\epsilon,\kappa}$.

\medskip

\noindent {\bf Step 2: Replacing the magnetic matrix with a magnetic pseudodifferential operator with periodic symbol, see Subsection \ref{SS-magn-Bloch-func}}. Adapting the methods of \cite{HS, CIP} in the case of a constant magnetic field $\epsilon B_0$, i.e. for $\kappa=0$, one can define a {\it periodic magnetic Bloch band function} $\lambda^\epsilon$ which is a perturbation of order $\epsilon$ of $\lambda_0$. 
Considering this magnetic Bloch band function as a periodic symbol, we may define its {\it magnetic quantization} (see Subsection \ref{SS-Main-Def}) in the magnetic field $B_{\epsilon,\kappa}$, denoted by $\mathfrak{Op}^{A^{\epsilon,\kappa}}(\lambda^\epsilon)$. It turns out that the spectrum of $\mathfrak{Op}^{A^{\epsilon,\kappa}}(\lambda^\epsilon)$ is located at a Hausdorff distance of order $\kappa\epsilon$ from the spectrum of the effective operator $\pi^{\epsilon,\kappa}_0H^{\epsilon,\kappa}\pi^{\epsilon,\kappa}_0$. Hence if $\mathfrak{Op}^{A^{\epsilon,\kappa}}(\lambda^\epsilon)$ has gaps of order $\epsilon$ (provided that $\kappa$ is smaller than some constant independent of $\epsilon$), the same is true for the bottom of the spectrum of $H^{\epsilon,\kappa}$.

\medskip

\noindent {\bf Step 3: Spectral analysis of $\mathfrak{Op}^{A^{\epsilon,\kappa}}(\lambda^\epsilon)$, see Section \ref{S-magn-q-Bloch-func}.}  Here we compare the spectrum of $\mathfrak{Op}^{A^{\epsilon,\kappa}}(\lambda^\epsilon)$ with the spectrum of a Landau-type quadratic symbol defined using the Hessian of $\lambda^\epsilon$ near its simple, isolated minimum; this is achieved by proving that the magnetic quantization of an explicitly defined symbol is in fact a {\it quasi-resolvent} for the magnetic quantization of $\lambda^\epsilon$ (see Subsection \ref{SSS-quasi-inv}). The main technical result is Proposition \ref{mainP}.

An important technical component of the proof of Proposition \ref{mainP} is the expansion of a magnetic Moyal calculus for symbols with {\it weak spatial variation} (see Appendix \ref{A-cl-sl-var}), that replaces the Moyal calculus for a constant field as appearing in \cite{BGH, HS3}. Some often used notations and definitions are recalled in Appendix \ref{A-not}.

\section*{Acknowledgements}

HC acknowledges financial support from Grant 4181-00042 of the Danish Council for Independent Research $|$ Natural Sciences. 
  BH and  RP would like to thank  Aalborg University for the financial support of numerous stays in Aalborg in 2015 and 2016.

\section{{Previous spectral results under the gap condition and with a constant magnetic field}}

When $\kappa=0$ and under the gap-Hypothesis \ref{Hyp-1b}, some sharp spectral results were proved by B.~Helffer and J. Sj\"ostrand in \cite{HS}. {They not only show that gaps open, but they also prove that the width of the spectral islands is of order $\mathcal O( \epsilon^\infty)$.} 

For the convenience of the reader, we recall their precise statement.
Let $\lambda_0(\theta_1,\theta_2)$ denote the isolated eigenvalue and let $B_0=1\,$. Then if $\epsilon$ is small enough, $H^{\epsilon,0}$ continues to have an isolated spectral island $I_\epsilon$ near the range of $\lambda_0\,$ \cite{B2, CP-1}. {In this case, one can prove \cite{Ne-LMP, Ne-RMP} that the  magnetic Hamiltonian restricted to the spectral subspace corresponding to $I_\epsilon$ is {\bf unitarily equivalent} with a magnetic matrix acting on $\ell^2(\Gamma)$. The crucial advantage of working with a constant magnetic field in the gapped case is that the above mentioned magnetic matrix is {\bf isospectral} with an {\bf one-dimensional} 
$\epsilon$-pseudo-differential operator, whose symbol $\lambda_\epsilon$
admits an asymptotic expansion in $\epsilon$ and 
  whose principal symbol is $\lambda_0$ (see formula (6.13) in \cite{HS}).} By
$\epsilon$-pseudo-differential operator, acting in $L^2(\mathbb{R})$, we mean the following object:
  \beq\label{Opepsilon}
\big(\mathfrak{Op^w_\epsilon}(\lambda_\epsilon)u\big)(x)\ :=\
(2\pi \epsilon)^{-1} \,\int_\mathbb{R}\int_{\mathbb{R}}e^{i\langle
\xi,x-y\rangle/\epsilon } \, \lambda_\epsilon \big(\frac{x+y}
{2},\xi\big)u(y)\,d\xi\,dy\,,\quad  u\in \mathscr{S}(\mathbb{R})\,.
\eeq
Note that the operator $ \mathfrak{Op^w_\epsilon}(\lambda_\epsilon)$ can also
be considered as the usual Weyl quantification  of the symbol $(x,\xi) \mapsto \lambda_\epsilon (x,\epsilon \xi)$.
 It is important to note that we have quantified the symbol into an operator 
 on $L^2(\mathbb R)$. 
  The spectral analysis of a Hamiltonian of the form $\mathfrak{Op}^w_\epsilon
(\lambda)$ in $(1D)$ (that is standard for $\lambda(x,\xi)=x^2+\xi^2$) has been
extended to general Hamiltonians in \cite{HelRo} (see also \cite{HS1}).
 Near the minimum of $\lambda_0$  one can perform a
semi-classical analysis of the bottom of the spectrum under the hypothesis that the minimum is
non degenerate, assumption which is satisfied in our case (see Appendix \ref{A-lambda0}). One can also perform a semi-classical analysis
near each energy  $E\in \lambda_0(\mathbb R^2)$ such that $\nabla \lambda_0 \neq
0$ on $\lambda_0^{-1} (E)$ (under an assumption that the connected components of
$\lambda_0^{-1} (E)$ are compact) and prove a Bohr-Sommerfeld  formula for the
eigenvalues of one-well microlocal problems. Each time this implies as a
byproduct the existence of gaps, the spectrum at the bottom  being contained,
due to the tunneling effect,  inside the union 
of exponentially small (as $\epsilon \rightarrow 0$)  intervals centered at a
point asymptotically close to the value computed by the Bohr-Sommerfeld rule.  
This is indeed what is done in great detail in \cite{HS1} and \cite{HS3} under 
specific conditions on the potential and for the square lattice, which imply that $\lambda_0(\theta)$ 
is close in a suitable sense to ${\bf t} \left(\cos \theta_1 + \cos \theta_2)\right)$, where $\bf t$ 
is a tunneling factor. See Section 9 in \cite{HS1}.
In our context, one of the main results in \cite{HS} is the following.
\begin{theorem}  \label{partTh} Under Hypothesis \ref{Hyp-1b}, given a positive integer $N>1$ there exist $\epsilon_0>0$ and $C>0$,  such
that for $\epsilon\in(0,\epsilon_0]$ there exist $a_0<b_0<...<a_N<b_N$ such that $a_0= \inf\{\sigma(H^{\epsilon,0})\}$ and:
\begin{align*}
&\sigma(H^{\epsilon,0})\cap \left [a_0,b_N\right ]\subset \bigcup_{k=0}^N [a_k,b_k]\,,\\
&b_k-a_k= \mathcal O ( \epsilon^\infty), \; 0\leq k\leq N,\quad {\rm and}\quad  a_{k+1}-b_k\geq \epsilon/C,\; 0\leq k\leq N-1\,.
\end{align*}
Moreover, 
$$ {\rm dim\big (Ran}E_{[a_k,b_k]}(H^{\epsilon,0}) \big )= + \infty\,.$$
\end{theorem}

One can actually have the  same description with $N=N(\epsilon)$ in an interval
$[-\infty, E]$  under the following conditions that are always satisfied for $|E-\inf \lambda_0|$ small enough:
\begin{itemize}
\item 
$ \lambda_0^{-1}((-\infty,E]) = \cup_{\gamma^* \in \Gamma^*} \tau_{\gamma^*} W_0(E)$\,.
\item  $W_0(E)$ is connected and  contains $\theta=0$ which is the unique critical point of $\lambda_0$ in $W_0(E)\,$.
\item $\tau_{\gamma^*} W_0(E)\cap\tau_{\tilde\gamma^*} W_0(E) = \emptyset$ if $\gamma^* \neq \tilde \gamma^*\,$.
\end{itemize}

The proof in Section 2 in \cite{HS1} gives also  in each interval $[a_k,b_k]$ an orthonormal basis $\phi_{\gamma^*}$ (indexed by $\Gamma^*$)
of the image of  $E_{[a_k,b_k]}(H^{\epsilon,0}) $  of functions $\phi_{\gamma^*}$ being localized near $\gamma^*$. 
Here we have identified  this image with the image of $E_{[a_k,b_k]}(\mathfrak{Op}_\epsilon (\lambda_\epsilon))$ in $L^2(\mathbb R)$. 
Finally these results imply the following statement.

\begin{theorem}  \label{partTha} Under Hypothesis \ref{Hyp-1b} and  the above assumptions, for any $L \in \mathbb N^*$,  there exist $\epsilon_0>0$ and $C>0$,  such
that for $\epsilon\in(0,\epsilon_0]$ there exist $N(\epsilon)$ and $a_0<b_0<...<a_N<b_N$ such that $a_0= \inf\{\sigma(H^{\epsilon,0})\}$ and:
\begin{align*}
&\sigma(H^{\epsilon,0})\cap (-\infty, E)  \subset \bigcup_{k=0}^N [a_k,b_k]\,,\\
&|b_k-a_k| \leq C\, \epsilon^L  \mbox{ for } 0\leq k\leq N(\epsilon)\,,\quad {\rm and}\quad  a_{k+1}-b_k\geq \epsilon/C \mbox{ for } 0\leq k\leq N(\epsilon)-1\,.
\end{align*}
Moreover $a_k$ is determined by a Bohr-Sommerfeld rule 
$$a_k = f( (2k+1) \epsilon,\epsilon)\,,$$
 where $t \mapsto f(t,\epsilon)$ 
has a complete expansion in powers of $\epsilon$, $f(0,0)=\inf \lambda_0\,,$ and $\partial_t f(0,0) \neq 0$  (see \cite{HelRo} or \cite{HS1}).
\end{theorem}
\begin{remark}\label{remarca-marts}
We conjecture that Theorem \ref{partTha} still holds with bands of size $\exp(-\frac{S}{\epsilon})$ for some $S>0$\,.  
What is missing are some "microlocal" decay estimates which are only established in particular cases in \cite{HS1} for Harper's like models. 
We conjecture that if in addition  $E < \inf \lambda_1$, Theorem \ref{partTha} is true  even under the weaker Hypothesis \ref{Hyp-1a}.
\end{remark}

\begin{remark}
The case of purely magnetic Schr\"odinger operators when the magnetic field has a global non-degenerate minimum has been treated in \cite{HK, RV}.
\end{remark}

\section{The effective band Hamiltonian}\label{S-eff-band-Ham}

\subsection{The magnetic almost Wannier functions}\label{SS-W-func} 
We recall  some results from \cite{Ne-RMP,HS, CHN,CIP}. Under both Hypothesis
\ref{Hyp-1b} and \ref{Hyp-1a}, using the analyticity of the application
$\X^*\ni\theta\rightarrow\hat{H}^0(\theta)$ in the norm resolvent sense and
the contour integral formula for the spectral projection, it is well known
(see for example Lemma 1.1 in \cite{HS}) that  one can choose
its first $L^2$-normalized eigenfunction as
an analytic function $\X^*\ni\theta\rightarrow\hat{\phi}_0(\theta,\cdot)\in
L^2(\mathbb{T})$ such that $$\hat{\phi}_0(\theta+\gamma^*,x)=e^{i<\gamma^*,x>}
\hat{\phi}_0(\theta,x)$$  and
\beq
\hat{H}^0(\theta)\,\hat{\phi}_0( \theta,\cdot )\,=\,\lambda_0(\theta)\,
\hat{\phi}_0(\theta,\cdot)\,.
\eeq
Then the principal 
{\it Wannier function} $\phi_0$ is defined  (see \eqref{inv}) by:
\beq\label{expl-FBZ}
\phi_0(x) = \big[\mathscr{V}_\Gamma^{-1}\hat{\phi}_0\big](x)\,=
\,|E_*|^{-\frac 12}\int_{E_*}e^{i<\theta,x>}\hat{\phi}_0(\theta,x)\,d\theta\,.
\eeq
 $\phi_0$ has rapid decay. In fact, using the analyticity property of $\theta \mapsto \hat{\phi}_0(\theta,\cdot)$ and  deforming the contour of integration in \eqref{expl-FBZ} (cf (1.8) and (1.9) in \cite{HS}),  we get the existence of $C>0$ and for any $\alpha \in \mathbb N^2$ of $C_\alpha >0$ such that
$$
|\partial_x^\alpha \phi_0(x)| \leq  C_\alpha \,\exp (- |x|/C)\,,\, \forall x \in \mathbb R^2\,.
$$

We shall also consider the associated orthogonal projections
\beq
\hat{\pi}_0(\theta):=|\hat{\phi}_0(\theta,\cdot)><\hat{\phi}_0(\theta,\cdot)|\,,\quad\pi_0:=\mathscr{V}_\Gamma^{-1}\left(\int_{E_*}^\oplus
\hat{\pi}_0(\theta)d\theta\right)\mathscr{V}_\Gamma\,.
\eeq 
Then the family  generated by $\phi_0$ by translation with $\gamma$ over the lattice $\Gamma$: 
$$\phi_\gamma:=\tau_{-\gamma}\phi_0$$ 
is an
orthonormal basis for the subspace $\pi_0\mathcal{H}\,$.\\
\begin{remark}
We note that under Hypothesis \ref{Hyp-1b}, $\pi_0$ is the spectral projection attached to the first simple band. This is no longer true under Hypothesis \ref{Hyp-1a} where we might have $\inf(\lambda_1)<\sup(\lambda_0)\,$.
\end{remark}

We now start the construction of some {\it magnetic
almost Wannier functions}. In order to obtain the best properties for these
functions adapted to  the specific structure of the  magnetic field (see \eqref{Bek}) we shall proceed  in two steps. First we consider the constant magnetic field
$\epsilon B_0$, and then we add the second perturbation
$\kappa\epsilon B(\epsilon x)$.

\begin{definition}\label{Def-q-W-magn}~
\begin{enumerate}
\item For any $\gamma\in\Gamma$ and with $A^0$ defined in \eqref{defA0} we introduce: 
$$
\hspace{-2cm} \overset{\circ}{\phi_\gamma^\epsilon}(x)\,:=\,\Lambda^\epsilon(x,
\gamma)\phi_0(x-\gamma)\, ,\quad\Lambda^\epsilon(x,y)\,:=\,\exp\left\{-i\,\epsilon\int_{[x,y]}A^0\right\}\,.$$
\item $\widetilde{\pi}_0^\epsilon$ will denote the orthogonal projection on the closed linear span of
$\{\overset{\circ}{\phi_\gamma^\epsilon}\}_{\gamma\in\Gamma}\,$.
\item  
$\mathbb{G}^\epsilon_{\alpha\beta}\,:=\,\langle\overset{\circ}{
\phi_\alpha^\epsilon},\overset{\circ}{\phi_\beta^\epsilon}\rangle_{\mathcal{H}}$
denotes the infinite \textit{Gramian matrix}, indexed by $\Gamma\times\Gamma\,$.
\end{enumerate}
\end{definition}
\noindent Due to the exponential decay of $\phi_0$, we obtain (see Lemma 3.15 in \cite{CIP}):
\begin{proposition}\label{P-2}
There exists $\epsilon_0 >0$ such that, for any $\epsilon \in[0,\epsilon_0]$, 
the matrix $\mathbb{G}^\epsilon$ defines a positive bounded
operator on $\ell^2(\Gamma)$. 
Moreover, for any $m\in\mathbb{N}$, there exists $C_m >0$ such that
\beq  \label{Gepsilon}
\underset{(\alpha,\beta)\in\Gamma\times \Gamma}{\sup}\
<\alpha-\beta>^m\big| \mathbb{G}^\epsilon_{\alpha\beta}\,-\,\bb1\big| \,\leq\,
C_m\, \epsilon\,.
\eeq
\end{proposition}
Actually,  by (4.5) in \cite{HS}, we have even exponential decay but this is not needed in this paper.
From \eqref{Gepsilon}  and  the construction of the Wannier functions through the magnetic translations (cf  Lemmas~3.1 and 3.2 in \cite{CHN})), we obtain the following result:
\begin{proposition}\label{P-F-prop} For $\epsilon\in [0,\epsilon_0]\,$,  $\mathbb{F}^\epsilon\,:=\big(\mathbb{G}^\epsilon\big)^{-1/2}$ has the following properties:
\begin{enumerate}
\item  $\mathbb{F}^\epsilon\in\,\mathcal L (\ell^2(\Gamma)) \cap \mathcal L (\ell^\infty(\Gamma)) \,$.
\item For any $m\in\mathbb{N}$, there exists $C_m >0$ such that
\beq\label{decay-F}
\underset{(\alpha,\beta)\in\Gamma\times \Gamma}{\sup}\
<\alpha-\beta>^m\big| \mathbb{F}^\epsilon_{\alpha\beta}\,-\,\bb1\big| \,\leq\,
C_m\, \epsilon\,,\, \forall \epsilon \in [0,\epsilon_0]\,.
\eeq
\item There exists a rapidly decaying function $\mathbf{F}_\epsilon:\Gamma\rightarrow\mathbb{C}$ such that for any pair $(\alpha,\beta)\in\Gamma\times\Gamma$ we have:
{$$
\mathbb{F}^\epsilon_{\alpha,\beta}=  \Lambda^\epsilon(\alpha,\beta)\, \mathbf{F}_\epsilon(\alpha-\beta)\,.
$$
}
\end{enumerate}
\end{proposition}

\medskip

Thus for all  $\epsilon\in[0,\epsilon_0]$ we can define the following orthonormal basis of $\widetilde{\pi}_0^\epsilon\mathcal{H}$\,:
\beq\label{def-q-Wannier}
\phi^\epsilon_\gamma\,:=\,\underset{\alpha\in\Gamma}{\sum}\, \mathbb{F
}^\epsilon_{\alpha\gamma}\, \overset{\circ}{\phi_\alpha^\epsilon}\,,   \forall \gamma \in \Gamma \,.
\eeq
\begin{remark}\label{R3.5}
Note that the  operators $\{\Lambda^\epsilon(\cdot,
\gamma)\tau_{-\gamma}\}_{\gamma\in\Gamma}$, appearing in Definition \ref{Def-q-W-magn} (1), are the {\it magnetic translations} considered in \cite{Ne-RMP} and 
in Section~2 (p. 147) of \cite{HS}. As proved there, these operators satisfy the following composition relation:
\begin{equation*}
\big[\Lambda^\epsilon(\cdot,\alpha)\tau_{-\alpha}\big]\big[\Lambda^\epsilon(\cdot,\beta)\tau_{-\beta}\big]\,=
\,\Lambda^\epsilon(\beta,\alpha)\big[\Lambda^\epsilon(\cdot,\alpha+\beta)\tau_{-\alpha-\beta}\big],
\qquad\forall(\alpha,\beta)\in\Gamma\times\Gamma\,.
\end{equation*}
\end{remark}
\begin{remark}\label{comm-magn-transl}
One can verify \cite{HS,CHN} that for any $\gamma\in \Gamma$, $\Lambda^\epsilon(\cdot,
\gamma)\tau_{-\gamma}$ commute with $H^{\epsilon,0}$ and with the {\it magnetic momenta} $\big(-i\partial_j-\epsilon A_j^0(x)\big)$.
\end{remark}

Using the second point of Theorem 1.10 from \cite{CHN} one  shows that the second point of our Proposition \ref{P-F-prop} implies the next two statements.
\begin{proposition}
With $\psi^\epsilon_0$  in $\mathscr S(\mathbb R^2)$ defined by
 {\beq\label{F-psi-e0}
\psi^\epsilon_0(x)=\underset{\alpha\in\Gamma}{\sum}\mathbf{F}^\epsilon(\alpha)\, \Lambda^\epsilon(x,\alpha)\, \phi_0(x-\alpha)\,,
\eeq
}
we have 
 $$
\phi^\epsilon_\gamma\,=\,\Lambda^\epsilon(\cdot ,\gamma)(\tau_{-\gamma}
\psi^\epsilon_0)\,,\qquad\forall\gamma\in\Gamma\,.
$$
\end{proposition}
\begin{corollary}\label{est-psi-phi}
There exists $\epsilon_0 >0$   and,  for  any $m\in\mathbb{N}\,$,
$\alpha\in\mathbb{N}^2$, there exists  $C_{m,\alpha} >0$ s. t. 
$$
<x>^m\left|[\partial_x^\alpha(\psi^\epsilon_0-\phi_0)](x)\right|\leq\,C_{m,\alpha}\, \epsilon\,,
$$
for any $\epsilon\in[0,\epsilon_0]$ and any $x\in \mathbb R$\,.
\end{corollary}

\medskip

We are now ready to consider the case with $\kappa\neq 0\,$.
\begin{definition}\label{Def-q-Wannier}
The magnetic almost Wannier functions associated with  $B_{\epsilon,\kappa}$ in \eqref{Bek} are  defined as:
$$
\overset{\circ\quad}{\phi^{\epsilon,\kappa}_\gamma}\ :=\
\Lambda^{\epsilon,\kappa}(\cdot ,\gamma)\big(\tau_{-\gamma}\psi^\epsilon_0\big)\,,
$$
 where $\Lambda^{\epsilon,\kappa}$ is given by \eqref{defLambda} with $A_{\epsilon,\kappa}$ as defined in \eqref{defAepsilon}.
\end{definition}

If we choose some smooth vector potential $A(y)$ such that  $\curl A= B $ (we recall that $B_{\epsilon,\kappa}:=\epsilon B_0+\kappa\epsilon B(\epsilon x)$) we introduce the quantities:
$$
A_\epsilon(x):=A(\epsilon x), \quad 
\widetilde{\Lambda}^{\epsilon,\kappa}(x,y)\ :=\
\exp\left\{-i\kappa\int_{[x,y]}A_\epsilon\right\}\,,
$$
and (the second identity below is a consequence of the Stokes theorem)
{
\begin{equation}\label{D-Omega'}
\Omega^{\epsilon,\kappa}(x,y,z):=\exp\left\{-i\int_{<x,y,z>}\hspace*{-10pt}
B_{\epsilon,\kappa}\right\}=\Lambda^{\epsilon,\kappa}(x,y)\Lambda^{\epsilon,\kappa}(y,z)\Lambda^{\epsilon,\kappa}(z,x).
\end{equation}
}
Then we have the equalities
$$\Lambda^{\epsilon,\kappa}(x,y)=\Lambda^{\epsilon}(x,y)\widetilde{\Lambda}^{\epsilon,\kappa}(x,y),\quad \overset{\circ\quad}{\phi^{\epsilon,\kappa}_\gamma}=\widetilde{\Lambda}^{
\epsilon,\kappa}(\cdot ,\gamma)\phi^\epsilon_\gamma\,.
$$

The following statement is proved in Lemma 3.1 of \cite{CHN}. It is based on the rapid decay of the Wannier functions and the polynomial growth of the derivatives of $\Omega^{\epsilon,\kappa}(\alpha,\beta,x)$ (see \eqref{Est-Omega}).
\begin{proposition}\label{P-ort-ek}
There exists $\epsilon_0 >0$ such that, for any $(\epsilon,\kappa) \in[0,\epsilon_0]\times [0,1]$, 
the Gramian matrix $
\left(\mathbb{G}^{\epsilon,\kappa}_{\alpha\beta}\right)_{ (\alpha,\beta)\in 
\Gamma^2}$ defined by 
$$\mathbb{G}^{\epsilon,\kappa}_{\alpha\beta}:=\langle
\overset{\circ\quad}{\phi^{\epsilon,\kappa}_\alpha}\,,\overset{\circ\quad}{\phi^
{\epsilon,\kappa}_\beta}\rangle_{\mathcal{H}}$$
 has the form:
$$
\mathbb{G}^{\epsilon,\kappa}_{\alpha\beta}\,=\,\delta_{\alpha\beta}\,+\,
\widetilde{\Lambda}^{\epsilon,\kappa}(\alpha,\beta)\mathbb{X}^{\epsilon,\kappa}(\alpha,
\beta),
$$
where, for all $m\in\mathbb{N}\,$,  there exists $
C_m >0$ such that
$$
\left|\mathbb{X}^{\epsilon,\kappa}(\alpha,
\beta)\right|\,\leq\,C_m\, \kappa\epsilon<\alpha-\beta>^{-m},\
\forall(\alpha,\beta)\in \Gamma^2\, .
$$
\end{proposition}

\begin{definition}\label{Def-W-magn} If $\epsilon\in [0,\epsilon_0]$  and  $\kappa\in[0,1]$ we define:
\begin{enumerate}
\item
$\widetilde{\phi}^{\epsilon,\kappa}_\gamma\,:=\,\underset{\alpha\in\Gamma}{\sum}
\mathbb{F}^{\epsilon,\kappa}_{\alpha\gamma}\, \overset{\circ\quad}{\phi^{\epsilon,
\kappa}_\alpha}$ with $\mathbb{F}^{\epsilon,\kappa}\,:=\big[\mathbb{G}^{\epsilon,\kappa}\big]^
{-1/2}$, 
\item  $\widetilde{\pi}^{\epsilon,\kappa}_0\,$ to be the 
orthogonal projection on the closed linear span of 
$\{\overset{\circ\quad}{\phi^{\epsilon,\kappa}_\gamma}\}_{\gamma\in\Gamma}\,$.
\end{enumerate}
\end{definition}

\subsection{The almost invariant magnetic subspace}
\label{SS-inv-magn-band}
We shall prove (see Proposition \ref{P-quasi-inv} below) that the orthogonal projection $\widetilde{\pi}^{\epsilon,\kappa}_0$ is almost invariant (modulo an error of order $\epsilon$) for the Hamiltonian $H^{\epsilon,\kappa}$. 
\begin{proposition}\label{P-quasi-inv}
There exist $\epsilon_0>0$  and  $C>0$ such that, for any $(\epsilon,\kappa)\in[0,\epsilon_0]\times[0,1]\,$, the range of $\widetilde{\pi}^{\epsilon,\kappa}_0$ belongs to the domain of $H^{\epsilon,\kappa}$ and
$$
\left\|\big[H^{\epsilon,\kappa},\widetilde{\pi}^{\epsilon,\kappa}_0\big]\right\|_{\mathcal L(\mathcal{H})}\leq\,C\, \epsilon\,.
$$
\end{proposition}
\begin{proof}
All the Wannier-type functions introduced in Subsection \ref{SS-W-func} belong to $\mathscr{S}(\X)$ and are in the domain of the respective Hamiltonians. We follow the ideas of  Subsection 3.1 in \cite{CIP} and compare the orthogonal projection $\widetilde{\pi}^{\epsilon,\kappa}_0$ with $\pi_0$. Let us denote by $p^{\epsilon,\kappa}$ (resp. $p$) the distribution symbol of the orthogonal projection $\widetilde{\pi}^{\epsilon,\kappa}_0$ (resp. $\pi_0$) for the corresponding quantizations, i.e.
\beq
\widetilde{\pi}^{\epsilon,\kappa}_0:=\mathfrak{Op}^{\epsilon,\kappa}\big(p^{\epsilon,\kappa}\big),\quad
\pi_0:=\mathfrak{Op}\big(p\big).
\eeq
We have
\beq
\widetilde{\pi}^{\epsilon,\kappa}_0\,=\,\underset{\gamma\in\Gamma}{\sum}|\widetilde{\phi}^{\epsilon,\kappa}_\gamma><\widetilde{\phi}^{\epsilon,\kappa}_\gamma|\,,\qquad
\pi_0\,=\,\underset{\gamma\in\Gamma}{\sum}|\phi_\gamma><\phi_\gamma|\,.
\eeq
 For any $\gamma \in \Gamma$,  both projections $|\widetilde{\phi}^{\epsilon,\kappa}_\gamma><\widetilde{\phi}^{\epsilon,\kappa}_\gamma|$ and $|\phi_\gamma><\phi_\gamma|$ are   magnetic pseudodifferential operators with associated symbols $p^{\epsilon,\kappa}_\gamma$ and  $p_\gamma$ of class $\mathscr{S}(\Xi)$. We conclude then that the symbols 
\beq
p^{\epsilon,\kappa}\,=\,\underset{\gamma\in\Gamma}{\sum}p^{\epsilon,\kappa}_\gamma,\qquad
p\,=\,\underset{\gamma\in\Gamma}{\sum}p_\gamma \,,
\eeq
where the series converge for the weak distribution topology, are in fact  of class $S^{-\infty}(\Xi)$ (see Definition \ref{Def-symb} in Appendix \ref{A-not}).
Then
\beq
\big[H^{\epsilon,\kappa},\widetilde{\pi}^{\epsilon,\kappa}_0\big]\,=\,\mathfrak{Op}^{\epsilon,\kappa}\big(h\,\sharp^{\epsilon,\kappa}\, p^{\epsilon,\kappa}-p^{\epsilon,\kappa}\,\sharp^{\epsilon,\kappa}h\big)\,,
\eeq
and Proposition \ref{L.II.1.1} in Appendix \ref{SS-wmf} shows that
$$
h\,\sharp^{\epsilon,\kappa}p^{\epsilon,\kappa}=h\,\sharp \, p^{\epsilon,\kappa}+\kappa\epsilon\,r^{\epsilon,\kappa}(h,p^{\epsilon,\kappa})\,,
$$
with $r^{\epsilon,\kappa}(h,p^{\epsilon,\kappa})\in S^{-\infty}(\Xi)$ uniformly in $(\epsilon,\kappa)\in[0,\epsilon_0]\times[0,1]$  and similarly for $ p^{\epsilon,\kappa}\,\sharp^{\epsilon,\kappa}h $\,. 

By construction, we have the commutation relation
$$
\big[H^{0},\pi_0\big]\,=\,0\,,
$$
 so that $$h\,\sharp\,p-p\,\sharp\,h=0\,.
 $$
 Hence:
$$
\big[H^{\epsilon,\kappa},\widetilde{\pi}^{\epsilon,\kappa}_0\big] =   \mathfrak{Op}^{\epsilon,\kappa}\big(
h\,\sharp\, (p^{\epsilon,\kappa}-p)-(p^{\epsilon,\kappa}-p)\,\sharp\,
h\big)\,+\,\kappa\epsilon\,
\mathfrak{Op}^{\epsilon,\kappa}\big(r^{\epsilon,\kappa}(h,p^{\epsilon,\kappa}
)-r^ {\epsilon,\kappa}(p^{\epsilon,\kappa},h)\big)\,.
$$

Finally let us compare the two regularizing symbols $p^{\epsilon,\kappa}$ and $p\,$. We notice that for any function $\psi\in\mathscr{S}(\X)$ having $L^2$-norm equal to one,  its associated 1-dimensional orthogonal projection ${|\psi><\psi|}$ has the integral kernel 
$
K_\psi(x,y)\,:=\,\psi(x)\, \overline{\psi(y)}$
and thus (see \eqref{magn-quant})
its magnetic symbol $p^A_\psi$ is given by:
\begin{align*}
p^A_\psi(x,\xi)&=\,(2\pi)^{-1/2}\int_\X e^{-i<\xi,z>}
\psi\big(x+\frac z 2 \big)\, \overline{\psi\big(x-\frac z 2 \big)}\, 
\Lambda^A\big(x-\frac z 2 ,x+\frac z 2 \big)\, dz\,.
\end{align*}
Let us consider the difference $p^{\epsilon,\kappa}_\gamma-p_\gamma$ for some $\gamma\in\Gamma$ and compute

\begin{align*}
p^{\epsilon,\kappa}_\gamma(x,\xi)-p_\gamma(x,\xi)&=
(2\pi)^{-1/2}\int_\X e^{-i<\xi,z>}
\,\times\\
&\qquad  \times\left[\Lambda^{\epsilon,\kappa}\big(x-\frac z 2 ,x+\frac z 2 \big)\,\widetilde{\phi}^{\epsilon,\kappa}_\gamma\big(x+\frac z 2 \big)\,
\overline{\widetilde{\phi}^{\epsilon,\kappa}_\gamma\big(x-\frac z 2 \big)}\,\right.\\
&
\left.\qquad\qquad \qquad\qquad  \hspace{+3cm}-
\phi_\gamma\big(x+\frac z 2 \big)\,
\overline{\phi_\gamma\big(x-\frac z 2 \big)}\right]\,dz\,.
\end{align*}
 In order to estimate the above difference we  compare both terms with a third symbol $q^{\epsilon,\kappa}_\gamma$ associated with $\overset{\circ}{\phi}{}^{\epsilon,\kappa}_\gamma $ (see Definition \ref{Def-q-Wannier}):
$$
q^{\epsilon,\kappa}_\gamma(x,\xi)\,:=\,(2\pi)^{-1/2}\int_\X e^{-i<\xi,z>}\Lambda^{\epsilon,\kappa}\big(x-\frac z 2 ,x+\frac z 2 \big)\, \overset{\circ}
{\phi}{}^{\epsilon,\kappa}_\gamma\big(x+\frac z 2 \big)\, \overline{\overset{\circ}{\phi}{}^{\epsilon,\kappa}_\gamma\big(x-\frac z 2 \big)}\,dz
$$
and estimate the difference
$$
\begin{array}{l}
\hspace{-2cm} p^{\epsilon,\kappa}_\gamma(x,\xi)\ -
q^{\epsilon,\kappa}_\gamma(x,\xi)\\
\quad =(2\pi)^{-1/2}\underset{(\alpha,\beta)\in\Gamma\times\Gamma}{\sum}
\widetilde\Lambda^{\epsilon,\kappa}(\gamma,\alpha)\mathbb{X}^{\epsilon,\kappa}(\gamma,\alpha)
\widetilde\Lambda^{\epsilon,\kappa}(\beta,\gamma)
\overline{\mathbb{X}^{\epsilon,\kappa}(\gamma,\beta)} 
\, s^{\epsilon,\kappa}_{\alpha,\beta}(x,\xi) \,,
\end{array}
$$
with
$$
s^{\epsilon,\kappa}_{\alpha,\beta}(x,\xi) :=
\int_\X e^{-i<\xi,z>}\Lambda^{\epsilon,\kappa}\big(x-\frac z 2 ,x+\frac z 2 \big)\,\, \overset{\circ}
{\phi}{}^{\epsilon,\kappa}_\alpha\big(x+\frac z 2 \big)\, \overline{\overset{\circ}{\phi}{}^{\epsilon,\kappa}_\beta\big(x-\frac z 2 \big)}\,dz\,.
$$
Let us consider the above integrals after using the Stokes Theorem as in \eqref{D-Omega} or \eqref{D-Omega'}:
$$
\begin{array}{ll} 
\hspace{-1cm} s^{\epsilon,\kappa}_{\alpha,\beta}(x,\xi) 
 &  =\,\Lambda^{\epsilon,\kappa}(\beta,\alpha)
\int_\X e^{-i<\xi,z>}\, \Omega^{\epsilon,\kappa}\big(x-\frac z 2 ,\alpha,\beta\big)\Omega^{\epsilon,\kappa}\big(x-\frac z 2 ,x+\frac z 2 ,\alpha\big)\times \\
&\qquad\qquad  \qquad\quad \qquad \qquad\qquad\qquad  \times\, \psi^\epsilon_0\big(x+\frac z 2 -\alpha\big)\, \psi^\epsilon_0
\big(x-\frac z 2 -\beta\big)\, dz\,.
\end{array}
$$
Having in mind the estimate \eqref{Est-Omega} we conclude that all the functions $s^{\epsilon,\kappa}_{\alpha,\beta}$ are symbols of class $S^{-\infty}(\Xi)$, uniformly for $(\epsilon,\kappa)\in[0,\epsilon_0]\times[0,1]$ and that, for any seminorm $\nu$ defining the topology of $S^{-\infty}(\Xi)$, there exist exponents $(a(\nu),b(\nu))\in\mathbb{N}^2$  and, for  any pair $(N,M)\in \mathbb N^2$,  a constant $C_{N,M}$ such that
$$
\nu\big(s^{\epsilon,\kappa}_{\alpha,\beta}\big)\,\leq\,C_{N,M}\,
|\beta-\alpha|^{a(\nu)}\int_\X
|z|^{b(\nu)}|x-\frac z 2 -\beta|^{-N}|x+\frac z 2 -\alpha|^{-M}\, dz\,.
$$
Thus, choosing $M$ and $N$ large enough, we finally obtain that there exist  $C_{\nu}>0$ and $p(\nu)\in\mathbb{N}$ such that:
$$
\nu\big(s^{\epsilon,\kappa}_{\alpha,\beta}\big)\,\leq\,C_{\nu}\, |\beta-\alpha|^{p(\nu)},\qquad\forall(\epsilon,\kappa)\in[0,\epsilon_0]\times[0,1]\,, \, \forall (\alpha,\beta)\in \Gamma^2\,.
$$
In order to control the convergence of the series in $\gamma\in\Gamma$ in the definition of $p^{\epsilon,\kappa}$, we need to consider some weights of the form $\rho_{n,\gamma}(x):=<x-\gamma>^n$ for $(n,\gamma)\in\mathbb{N}\times\Gamma$ and notice that there exist  $C_{\nu,N,M,n}>0$ and $p(\nu,n)\in\mathbb{N}$ such that:
$$
\nu\big(\rho_{n,\gamma}s^{\epsilon,\kappa}_{\alpha,\beta}\big)\,\leq\,C_{\nu,N,M,n}\, |\alpha-\gamma|^{p(\nu,n)}
|\beta-\gamma|^{p(\nu,n)},\qquad\forall(\epsilon,\kappa)\in[0,\epsilon_0]\times[0,1]\,, \,\forall (\alpha,\beta,\gamma)\in \Gamma^3 \,.
$$
From  Proposition \ref{P-ort-ek}  we now conclude that $p^{\epsilon,\kappa}_\gamma-q^{\epsilon,\kappa}_\gamma$ are symbols of class $S^{-\infty}(\Xi)$ uniformly for $(\epsilon,\kappa)\in[0,\epsilon_0]\times[0,1]$ and for any seminorm $\nu$ defining the topology of this space we  can find $C(\nu,n,N)$ and $\widehat C(\nu,n,N)$ such that:
$$
\nu\big(\rho_{n,\gamma}(p^{\epsilon,\kappa}_\gamma-q^{\epsilon,\kappa}_\gamma)\big)\leq\,C(\nu,n,N)\,\kappa\epsilon\hspace{-6pt}
\underset{(\alpha,\beta)\in\Gamma\times\Gamma}{\sum}\hspace{-12pt}<\alpha-\gamma>^{-N}<\beta-\gamma>^{-N}\leq\,\widehat C(\nu,n,N)\, \kappa \epsilon\,,\,  \forall \gamma \in \Gamma \,.
$$

Finally, if we define $q^{\epsilon,\kappa}:=\underset{\gamma\in\Gamma}{\sum}q^{\epsilon,\kappa}_\gamma$ and consider the weights $\rho_{n,\gamma}$ with $n\in\mathbb{N}$ large enough, we easily conclude that the series defining $p^{\epsilon,\kappa}-q^{\epsilon,\kappa}$ converges in the weak distribution topology to a limit that is a symbol in $S^{-\infty}(\Xi)$ having the defining seminorms of order $\kappa\epsilon\,$.

We still have to estimate the difference $q^{\epsilon,\kappa}-p$ in $S^{-\infty}(\Xi)$. We have, with $\psi^\epsilon_\gamma:=\tau_{-\gamma}\psi^\epsilon_0$\,,
$$
\begin{array}{ll}
q^{\epsilon,\kappa}_\gamma(x,\xi)-\ p_\gamma(x,\xi) 
&=(2\pi)^{-1/2}\int_\X e^{-i<\xi,z>}\Lambda^{\epsilon,\kappa}\big(x-\frac z 2 ,x+\frac z 2 \big)\, \overset{\circ}
{\phi}{}^{\epsilon,\kappa}_\gamma \, \big(x+\frac z 2 \big)\overline{\overset{\circ}{\phi}{}^{\epsilon,\kappa}_\gamma\big(x-\frac z 2 \big)}\,dz
\\
&\quad  -(2\pi)^{-1/2}\int_\X e^{-i<\xi,z>}
\phi_\gamma\big(x+\frac z 2 \big)\,
\overline{\phi_\gamma\big(x-\frac z 2 \big)}\,dz\,
\\
&=(2\pi)^{-1/2}\int_\X e^{-i<\xi,z>}\Omega^{\epsilon,\kappa}\big(x-\frac z 2 ,x+\frac z 2 ,\gamma\big)\, \psi^\epsilon_\gamma\big(x+\frac z 2 \big)\, \overline{
\psi^\epsilon_\gamma\big(x-\frac z 2 \big)}\, dz\,
\\
&\quad -(2\pi)^{-1/2}\int_\X e^{-i<\xi,z>}
\phi_\gamma\big(x+\frac z 2 \big)\, \overline{\phi_\gamma
\big(x-\frac z 2 \big)}\,dz\,.
\end{array}
$$
Hence 
$$
\begin{array}{l}
q^{\epsilon,\kappa}_\gamma(x,\xi)-\ p_\gamma(x,\xi)\\
\qquad =(2\pi)^{-1/2}\int_\X e^{-i<\xi,z>}\left[\Omega^{\epsilon,\kappa}\big(x-\frac z 2 ,x+\frac z 2 ,\gamma\big)\,-\,1\right]\psi^\epsilon_\gamma\big(x+\frac z 2 \big)\, \overline{
\psi^\epsilon_\gamma\big(x-\frac z 2 \big)}\, dz\,
\\
\qquad \quad +(2\pi)^{-1/2}\int_\X e^{-i<\xi,z>}\left[
\psi^\epsilon_\gamma\big(x+\frac z 2 \big)\, \overline{
\psi^\epsilon_\gamma\big(x-\frac z 2 \big)}\,-\,
\phi_\gamma\big(x+\frac z 2 \big)\, \overline{\phi_\gamma
\big(x-\frac z 2 \big)}\right]\,dz\,.
\end{array}
$$
Let us first consider the second integrand and use \eqref{F-psi-e0} and \eqref{decay-F} in order to get the estimate
$$
<x>^n|\psi^\epsilon_0\big(x+\frac z 2 \big)\, \overline{
\psi^\epsilon_0\big(x-\frac z 2 \big)}\,-\,
\phi_0\big(x+\frac z 2 \big)\,\overline{\phi_0
\big(x-\frac z 2 \big)}\,|\,\leq\,C_n\, \epsilon\,.
$$
Using once again \eqref{Est-Omega} and the above estimate, arguments very similar to the above ones allow us to prove that $q^{\epsilon,\kappa}-p$ is in $S^{-\infty}(\Xi)$ uniformly for $(\epsilon,\kappa)\in[0,\epsilon_0]\times[0,1]$ and for any seminorm $\nu$ defining the topology of $S^{-\infty}(\Xi)$ there is a constant $C(\nu)$ such that:
\beq\label{q-est}
\nu\big(q^{\epsilon,\kappa}-p\big)\,\leq\,C(\nu)\,\epsilon\,.
\eeq

Summarizing we have proved that $p^{\epsilon,\kappa}-p$ is in $S^{-\infty}(\Xi)$ uniformly for $(\epsilon,\kappa)\in[0,\epsilon_0]\times[0,1]$ and that, for any seminorm $\nu$ defining the topology of $S^{-\infty}(\Xi)$, there exists $C(\nu)$ such that:
\beq\label{p-est}
\nu\big(p^{\epsilon,\kappa}-p\big)\,\leq\,C(\nu)\, \epsilon\,.
\eeq

In order to finish the proof we still have to control the operator norms of \break $\mathfrak{Op}^{\epsilon,\kappa}\big(h\,\,\sharp\,(p^{\epsilon,\kappa}-p)\big)$ and $\mathfrak{Op}^{\epsilon,\kappa}\big((p^{\epsilon,\kappa}-p)\,\,\sharp\, h\big)$. But the above results and the usual theorem on Moyal compositions of H\"{o}rmander type symbols imply that $h\,\sharp\, (p^{\epsilon,\kappa}-p)$ and $(p^{\epsilon,\kappa}-p)\,\sharp\, h$ are symbols of type $S^{-\infty}(\Xi)$ uniformly for $(\epsilon,\kappa)\in[0,\epsilon_0]\times[0,1]$ and all their seminorms (defining the topology of $S^{-\infty}(\Xi)$) are of order $\epsilon$. Finally, by Theorem 3.1 and Remark 3.2 in \cite{IMP1} we know that the operator norm is bounded by some symbol seminorm and thus will be of order $\epsilon\,$.
\end{proof}

\begin{definition}\label{magn-band Ham}   We call  quasi-band magnetic Hamiltonian, the operator $ \widetilde{\pi}^{\epsilon,\kappa}_0H^{\epsilon,\kappa}\widetilde{\pi}^{\epsilon,\kappa}_0$
and   quasi-band magnetic matrix,  its expression in  the orthonormal basis $\{\widetilde{\phi}^{\epsilon,\kappa}_\gamma\}_{\gamma\in\Gamma}$. 
\end{definition}

In order to apply a Feshbach type argument we need to control the invertibility on the orthogonal complement of $\widetilde{\pi}^{\epsilon,\kappa}_0\mathcal{H}$. Let us introduce 
\beq\label{p-ort}
\widetilde{\pi}^{\epsilon,\kappa}_\bot\,:=\,\bb1\,-\,\widetilde{\pi}^{\epsilon,\kappa}_0,\quad m_1:=\underset{\theta\in\mathbb{T}_*}{\inf}\lambda_1(\theta)\,,
\eeq
where $\lambda_1$ is the second Bloch eigenvalue. Define: 
\beq\label{corr-band-Ham}
K^{\epsilon,\kappa}\,:=\,H^{\epsilon,\kappa}\,+\, m_1 \, \widetilde{\pi}^{\epsilon,\kappa}_0\,.
\eeq
We have:
\begin{proposition}\label{L-1}
There exist $\epsilon_0$ and  $C>0$ such that, for $\epsilon\in[0,\epsilon_0]\,$,
$$K^{\epsilon,\kappa} \geq  m_1 -C\epsilon >0\,.$$ \end{proposition}
\begin{proof}
Using \eqref{p-est}, the conclusion just above it and the notation introduced in the proof of Proposition~\ref{P-quasi-inv} we can write
$$
K^{\epsilon,\kappa}\,=\,\mathfrak{Op}^{\epsilon,\kappa}
\big(h+ m_1p^{\epsilon,\kappa}\big)\,=\,\mathfrak{Op}^{\epsilon,\kappa}
\big(h+ m_1p\big)\,+   \epsilon R_{\epsilon,\kappa} \,,
$$
with $\|R_{\epsilon,\kappa}\|_{\mathcal{L}(\mathcal{H})}$ bounded uniformly in $(\epsilon,\kappa)\in[0,\epsilon_0]\times[0,1]\,$. \\
By Corollary 1.6 in \cite{CP-2} we have  
$$
\left|\inf\sigma\big(\mathfrak{Op}^{\epsilon,\kappa}
\big(h+ m_1p\big)\big)-\inf\sigma\big(\mathfrak{Op}
\big(h+ m_1p\big)\big)\right|\leq\,C\, \epsilon\,.
$$
Since $\mathfrak{Op}(h)$ commutes with $\mathfrak{Op}(p)$,  we have
$$
\inf\sigma\big(\mathfrak{Op}\big(h+ m_1p\big)\big)=\inf\sigma\big(
H^0+ m_1\pi_0\big)=m_1\,,
$$
and we are done.
\end{proof}

 An immediate consequence is the existence of $\epsilon_0 >0$ such that,  if $\epsilon \in [0,\epsilon_0]$ 
and  $\Re \z \leq \frac 23 m_1$, the operator $K^{\epsilon,\kappa}-\z$ is invertible on $\mathcal{H}$ with a uniformly bounded inverse $R^{\epsilon,\kappa}_\z$ in $\mathcal{L}(\mathcal{H})\,$.

\begin{proposition}
There exists $\epsilon_0 >0$ such that for $\epsilon \in [0,\epsilon_0]$, the Hausdorff distance between the spectra of $H^{\epsilon,\kappa}$ and $\widetilde{\pi}^{\epsilon,\kappa}_0
H^{\epsilon,\kappa}\widetilde{\pi}^{\epsilon,\kappa}_0$, both restricted to the interval $[0,  \frac{m_1}{2}]\,$, is of order $\epsilon^2$.
\end{proposition}
\begin{proof}~\\
{\bf Step 1}.\\
We first establish the invertibility of $\widetilde{\pi}^{\epsilon,\kappa}_\bot
\big(H^{\epsilon,\kappa}-\z\big)
\widetilde{\pi}^{\epsilon,\kappa}_\bot$ on the range of  $\widetilde{\pi}^{\epsilon,\kappa}_\bot$.  We have 
$$
\widetilde{\pi}^{\epsilon,\kappa}_\bot
\big(K^{\epsilon,\kappa}-\z\big)
\widetilde{\pi}^{\epsilon,\kappa}_\bot\,=\,
\widetilde{\pi}^{\epsilon,\kappa}_\bot
\big(H^{\epsilon,\kappa}-\z\big)
\widetilde{\pi}^{\epsilon,\kappa}_\bot.
$$
We also observe that:
 $$
\left( \widetilde{\pi}^{\epsilon,\kappa}_\bot
R^{\epsilon,\kappa}_\z
\widetilde{\pi}^{\epsilon,\kappa}_\bot \right) \left(\widetilde{\pi}^{\epsilon,\kappa}_\bot
\big(K^{\epsilon,\kappa}-\z\big)
\widetilde{\pi}^{\epsilon,\kappa}_\bot \right) 
=\widetilde{\pi}^{\epsilon,\kappa}_\bot\big(\bb1-R^{\epsilon,\kappa}_\z
[H^{\epsilon,\kappa},\widetilde{\pi}^{\epsilon,\kappa}_0]\big)
\widetilde{\pi}^{\epsilon,\kappa}_\bot\, .
$$
A similar formula holds for $ \left(\widetilde{\pi}^{\epsilon,\kappa}_\bot
\big(K^{\epsilon,\kappa}-\z\big)
\widetilde{\pi}^{\epsilon,\kappa}_\bot \right) \left( \widetilde{\pi}^{\epsilon,\kappa}_\bot
R^{\epsilon,\kappa}_\z
\widetilde{\pi}^{\epsilon,\kappa}_\bot \right)  \,$.\\
Using Proposition \ref{P-quasi-inv} we conclude that $\widetilde{\pi}^{\epsilon,\kappa}_\bot
\big(H^{\epsilon,\kappa}-\z\big)
\widetilde{\pi}^{\epsilon,\kappa}_\bot$ is invertible in the subspace $\widetilde{\pi}^{\epsilon,\kappa}_\bot\mathcal{H}$ and its inverse $R^{\epsilon,\kappa}_{\z,\bot}$ verifies the estimate
$$
\left\|R^{\epsilon,\kappa}_{\z,\bot}-\widetilde{\pi}^{\epsilon,\kappa}_\bot R^{\epsilon,\kappa}_\z
\widetilde{\pi}^{\epsilon,\kappa}_\bot\right\|_{\mathcal{L}(\widetilde{\pi}^{\epsilon,\kappa}_\bot\mathcal{H})}\leq\,C\,\epsilon\, .
$$
{\bf Step 2}.\\
The Feshbach inversion formula implies that 
if $\Re \z \leq  \frac{m_1}{2}\,$, the operator $H^{\epsilon,\kappa}-\z$ is invertible in $\mathcal{H}$ if and only if the operator
\beq\label{abb}
T^{\epsilon,\kappa}(\z):=\widetilde{\pi}^{\epsilon,\kappa}_0
H^{\epsilon,\kappa}\widetilde{\pi}^{\epsilon,\kappa}_0-\z\widetilde{\pi}^{\epsilon,\kappa}_0- \left(\widetilde{\pi}^{\epsilon,\kappa}_0
H^{\epsilon,\kappa}\, \widetilde{\pi}^{\epsilon,\kappa}_\bot\right) \,
R^{\epsilon,\kappa}_{\z,\bot}\, \left(\widetilde{\pi}^{\epsilon,\kappa}_\bot\,
H^{\epsilon,\kappa}\, \widetilde{\pi}^{\epsilon,\kappa}_0\right)
\eeq
is invertible in $\widetilde{\pi}^{\epsilon,\kappa}_0\mathcal{H}\,$. Since $\widetilde{\pi}^{\epsilon,\kappa}_0
H^{\epsilon,\kappa}\widetilde{\pi}^{\epsilon,\kappa}_\bot=\widetilde{\pi}^{\epsilon,\kappa}_0\,
[\widetilde{\pi}^{\epsilon,\kappa}_0,H^{\epsilon,\kappa}]\, \widetilde{\pi}^{\epsilon,\kappa}_\bot$ and using once again Proposition~\ref{P-quasi-inv}, we get that the last term in \eqref{abb} has a norm which is bounded by $C'\epsilon^2$, with $C'$ denoting a generic constant, uniformly in $\z \in [0,m_1/2]\,$.  

First, we assume that $\z\in [0,m_1/2]\setminus \sigma(\widetilde{\pi}^{\epsilon,\kappa}_0
H^{\epsilon,\kappa}\widetilde{\pi}^{\epsilon,\kappa}_0)\,$. A Neumann series argument implies that if the distance between $\z$ and $\sigma(\widetilde{\pi}^{\epsilon,\kappa}_0
H^{\epsilon,\kappa}\widetilde{\pi}^{\epsilon,\kappa}_0)$ is larger than $C'\epsilon^2$, then $T^{\epsilon,\kappa}(\z)$ is invertible, hence $\z$ is also in the resolvent set of $H^{\epsilon,\kappa}\,$. 

Second, we assume that $\z\in [0,m_1/2]\setminus \sigma(
H^{\epsilon,\kappa})$ and moreover, the distance between $\z$ and $\sigma(
H^{\epsilon,\kappa})$ is larger than $C'\epsilon^2$. From the Feshbach formula we get that $T^{\epsilon,\kappa}(\z)^{-1}$ exists and 
$$T^{\epsilon,\kappa}(\z)^{-1}=\widetilde{\pi}^{\epsilon,\kappa}_0(H^{\epsilon,\kappa}-\z)^{-1}\widetilde{\pi}^{\epsilon,\kappa}_0,\quad ||T^{\epsilon,\kappa}(\z)^{-1}||<C'^{-1}\epsilon^{-2}.$$ 
Then \eqref{abb} implies 
$$\widetilde{\pi}^{\epsilon,\kappa}_0
H^{\epsilon,\kappa}\widetilde{\pi}^{\epsilon,\kappa}_0-\z\widetilde{\pi}^{\epsilon,\kappa}_0=T^{\epsilon,\kappa}(\z)+\left(\widetilde{\pi}^{\epsilon,\kappa}_0
H^{\epsilon,\kappa}\, \widetilde{\pi}^{\epsilon,\kappa}_\bot\right) \,
R^{\epsilon,\kappa}_{\z,\bot}\, \left(\widetilde{\pi}^{\epsilon,\kappa}_\bot\,
H^{\epsilon,\kappa}\, \widetilde{\pi}^{\epsilon,\kappa}_0\right),$$
and a Neumann series argument shows that $\z$ also belongs to the resolvent set of the quasi-band Hamiltonian. 

Thus we have shown that, the Hausdorff distance between the two spectra restricted to the interval $[0,m_1/2]$ must be of order $\epsilon^2$. 
\end{proof}

\subsection{The magnetic quasi-Bloch function $\lambda^\epsilon$} \label{SS-magn-Bloch-func}

In this subsection,  we  study the spectrum of the operator  $\widetilde{\pi}^{\epsilon,\kappa}_0
H^{\epsilon,\kappa}\, \widetilde{\pi}^{\epsilon,\kappa}_0$ acting  in $\widetilde{\pi}^{\epsilon,\kappa}_0\mathcal{H}$ by looking at its associated magnetic matrix (see also Definition \ref{Def-W-magn}) in the basis $(\widetilde{\phi}^{\epsilon,\kappa}_\alpha)_{\alpha \in \Gamma}$\,:
\begin{equation*}\begin{array}{l}
\hspace{-1cm}\left\langle\widetilde{\phi}^{\epsilon,\kappa}_\alpha\,,\,H^{\epsilon,\kappa}\widetilde{\phi}^{
\epsilon,\kappa}_\beta\right\rangle_{\mathcal{H}}\\
\quad  = \underset{(\tilde{\alpha},\tilde{\beta})\in\Gamma\times\Gamma}{\sum}\hspace{-10pt}
\overline{\mathbb{F}^{\epsilon,\kappa}_{\tilde{\alpha}\alpha}}\,\mathbb{F}^{\epsilon,\kappa}_{\tilde{\beta}\beta}
\left\langle\widetilde{\Lambda}^{\epsilon,\kappa}(\cdot,\tilde{\alpha})\, \phi^{\epsilon}
_{\tilde{\alpha}}\,,\,H^{\epsilon,\kappa}\widetilde{\Lambda}^{\epsilon,\kappa
}(\cdot,\tilde{\beta}) \, \phi^{\epsilon}_{\tilde{\beta}}\right
\rangle_{\mathcal{H}}
\\
\quad  =\hspace{-5pt}\underset{(\tilde{\alpha},\tilde{\beta})\in\Gamma\times\Gamma}{\sum}\hspace{-10pt}
\overline{\mathbb{F}^{\epsilon,\kappa}_{\tilde{\alpha}\alpha}}\,\mathbb{F}^{\epsilon,\kappa}_{\tilde{\beta}\beta}
\left\langle\phi^{\epsilon}
_{\tilde{\alpha}}\,,\,\widetilde{\Lambda}^{\epsilon,\kappa}(\cdot,\tilde{\alpha})^{-1}\big((-i\nabla -\epsilon A^0(\cdot ) -\kappa A(\epsilon \cdot ))^2+V\big)\widetilde{\Lambda}^{\epsilon,\kappa}(\cdot,\tilde{\beta})\,\phi^{\epsilon}_{\tilde{\beta}}\right
\rangle_{\mathcal{H}}\,.
\end{array}
\end{equation*}
Introducing 
$$
a_\epsilon(x,\gamma)_j\,=\,\underset{k}{\sum}(x-\gamma)_k\int_0^1\epsilon
B_{jk}\big(\epsilon\gamma+s\epsilon(x-\gamma)\big)s\,ds\, \mbox{ for } j=1,2\,,
$$
and using the intertwining formula \eqref{comm-B-moment}, we find:
\begin{align}\label{26-02-1}
(-i\nabla-\epsilon A^0(x)-\kappa A(\epsilon x))\widetilde{\Lambda}^{\epsilon,\kappa}(x,\tilde{\beta})
=\widetilde{\Lambda}^{\epsilon,\kappa}(x,\tilde{\beta})\left\{(-i\nabla-\epsilon A^0(x))\,+\,\kappa a_\epsilon(x,\tilde{\beta})\right\}\,,
\end{align}
and 
\begin{align}\label{26-02-2}
(-i\nabla-\epsilon A^0(x)-\kappa A(\epsilon x))^2\widetilde{\Lambda}^{\epsilon,\kappa}(x,\tilde{\beta}) 
=\widetilde{\Lambda}^{\epsilon,\kappa}(x,\tilde{\beta})\left\{(-i\nabla-\epsilon A^0(x))\,+\,\kappa a_\epsilon(x,\tilde{\beta})\right\}^2.
\end{align}
Let us notice that
\begin{equation}\label{est-aepsilon}
|a_\epsilon(x,\gamma)|\leq C\epsilon<x-\gamma>\,.
\end{equation}
By the Stokes Formula we have 
\begin{equation}\label{L-ek}
\widetilde{\Lambda}^{\epsilon,\kappa}(x,\tilde{\alpha})^{-1}\widetilde{\Lambda}^{\epsilon,\kappa}(x,\tilde{\beta})\,=\,\widetilde{\Lambda}^{\epsilon,\kappa}(\tilde{\alpha},\tilde{\beta})\, \Omega^{\kappa,\epsilon}(\tilde{\alpha},x,\tilde{\beta})\,,
\end{equation}
and we know that
$$|\Omega^{\kappa,\epsilon}(\tilde{\alpha},x,\tilde{\beta})-\bb1|\,
\leq\kappa\epsilon\, |x-\tilde { \alpha}|\, |x-\tilde{\beta}|\,.
$$
Moreover, using
Remark \ref{comm-magn-transl}, one easily proves the following estimates
(remember the notation $H^\epsilon=H^{\epsilon,0}$):
\begin{equation}\label{L-ek-2}
\left|\left\langle \big(\Omega^{\kappa,\epsilon}(\tilde{\alpha},\cdot ,\tilde{\beta})-\bb1\big)\phi^\epsilon_{\tilde{\alpha}}\,,\,H^\epsilon\phi^\epsilon_{\tilde{\beta}}\right\rangle_{\mathcal{H}}\right|
\leq C_m\kappa\epsilon<\tilde{\alpha}-\tilde{\beta}>^{-m},\qquad\forall m\in\mathbb{N}\,.
\end{equation}

\begin{lemma} For any $m\in\mathbb{N}\,$, there exists $C_m$ such that if $\psi$
equals either $\big(\Omega^{\kappa,\epsilon}(\tilde{\alpha},\cdot
,\tilde{\beta})-\bb1\big)\phi^\epsilon_{\tilde{\alpha}}$  or
$\phi^\epsilon_{\tilde{\alpha}}\,$, we have:
$$
\left|\left\langle\psi\,,\big[\big((-i\nabla-\epsilon A^0)+\kappa a_\epsilon(\cdot ,\tilde{\beta})\big)^2+V\big]\phi^{
\epsilon}_{\tilde{\beta}}\right\rangle_{\mathcal{H}}-\,\left\langle\psi\,,\,H^{\epsilon}\phi^{
\epsilon}_{\tilde{\beta}}\right\rangle_{\mathcal{H}}\right|\leq
C_m\, \kappa\epsilon<\tilde{\alpha}-\tilde{\beta}>^{-m}\,.
$$
\end{lemma}
\begin{proof}~\\
The difference of the two scalar products is  equal to
\begin{equation*}
 \kappa^2
\left\langle\psi\,,\,a_\epsilon(\cdot,{\tilde{\beta}})^2\phi^{
\epsilon}_\beta\right\rangle_{\mathcal{H}} 
+\kappa\left\langle\psi\,,\,(-i\nabla-\epsilon A^0)\cdot a_\epsilon(\cdot ,{\tilde{\beta}})\phi^{
\epsilon}_{\tilde{\beta}}\right\rangle_{\mathcal{H}}\\
+\kappa \left\langle\psi\,,\,a_\epsilon(\cdot ,{\tilde{\beta}})\cdot(-i\nabla-\epsilon A^0(x))\phi^{
\epsilon}_{\tilde{\beta}}\right\rangle_{\mathcal{H}}.
\end{equation*}
Remark \ref{comm-magn-transl}, the estimate \eqref{est-aepsilon} and some
arguments similar to those leading to \eqref{L-ek-2}
finish the proof.
\end{proof}

\begin{proposition}\label{propo-marts1}
For any $m\in\mathbb{N}$, there exists $C_m$ such that,  $\forall \, (\alpha,\beta) \in \Gamma\times \Gamma\,$,
$$
\left|\left\langle\widetilde{\phi}^{\epsilon,\kappa}_\alpha\,,\,H^{\epsilon,\kappa}\widetilde{\phi}^{
\epsilon,\kappa}_\beta\right\rangle_{\mathcal{H}}\,-\,\widetilde{\Lambda}^{\epsilon,\kappa}(\alpha,\beta)
\left\langle\phi^{\epsilon}_\alpha\,,\,H^{\epsilon}\phi^{
\epsilon}_\beta\right\rangle_{\mathcal{H}}\right|\leq
C_m\, \kappa\epsilon<\alpha-\beta>^{-m}\,.
$$
\end{proposition}
\begin{proof}
Putting together all the previous estimates and using Proposition \ref{P-ort-ek} we obtain:
\begin{equation*}
\begin{array}{l}
\hspace{-1cm} \left|\left\langle\widetilde{\phi}^{\epsilon,\kappa}_\alpha\,,\,H^{\epsilon,\kappa}\widetilde{\phi}^{
\epsilon,\kappa}_\beta\right\rangle_{\mathcal{H}}\,-\,\widetilde{\Lambda}^{\epsilon,\kappa}(\alpha,\beta)
\left\langle\phi^{\epsilon}_\alpha\,,\,H^{\epsilon}\, \phi^{
\epsilon}_\beta\right\rangle_{\mathcal{H}}\right|\,\\
\quad  \leq\,\left|
\left\langle H^\epsilon\,\big(\Omega^{\kappa,\epsilon}(\alpha,x,\beta)-\bb1\big)\phi^\epsilon_\alpha\,,\,\phi^\epsilon_\beta\right\rangle_{\mathcal{H}}\right|
\\ 
\qquad +C(m_1,m_2,m_3)\, \kappa\epsilon\,\left( \underset{(\tilde{\alpha},\tilde{\beta})\in(\Gamma\setminus\{\alpha\})\times(\Gamma
\setminus\{\beta\})}{\sum}\hspace{-30pt}
<\alpha-\tilde{\alpha}>^{-m_1}
<\beta-\tilde{\beta}>^{-m_2}<\tilde{\alpha}-\tilde{\beta}>^{-m_3}\right)\,,
\end{array}
\end{equation*}
for any triple $(m_1,m_2,m_3)\in\mathbb{N}^3$ and the Proposition follows
from \eqref{L-ek-2}.
\end{proof}

\medskip

\noindent Remarks \ref{comm-magn-transl} and \ref{R3.5} imply:
\begin{align*}
\left\langle\phi^{\epsilon}_\alpha\,,\,H^{\epsilon}
\phi^{\epsilon}_\beta\right\rangle_{\mathcal{H}}& =\left\langle\psi^{\epsilon}_0\,,\,(\Lambda^\epsilon(\cdot,\alpha))^{-1}\tau_{\alpha-\beta}\Lambda^\epsilon(\cdot,\beta) 
H^{\epsilon}\psi^{\epsilon}_0\right
\rangle_{\mathcal{H}}\\ & 
=\Lambda^\epsilon(\alpha,\beta)\left\langle\psi^{\epsilon}_0\,,\,\Lambda^\epsilon(\cdot,\beta-\alpha)\tau_{-(\beta-\alpha)}
H^{\epsilon}\psi^{\epsilon}_0\right\rangle_{\mathcal{H}}.
\end{align*}
\begin{definition}\label{magn-Bloch-func}   We define $\mathfrak h^{\epsilon}\in
\ell^2(\Gamma)$ by:
\beq\label{def-h}
\mathfrak{h}^\epsilon(\gamma)\,:=\,\left\langle\psi^{\epsilon}_0\,,\,\Lambda^\epsilon(x,\gamma)\tau_{-\gamma}
H^{\epsilon}\psi^{\epsilon}_0\right\rangle_{\mathcal{H}}\,=\,
\left\langle\phi^{\epsilon}_0\,,\,H^{\epsilon}
\phi^{\epsilon}_\gamma\right\rangle_{\mathcal{H}} \mbox{ for } \gamma \in \Gamma\,,
\eeq
and the  magnetic quasi Bloch function $\lambda^\epsilon$ as   its discrete Fourier transform:
\beq\label{def-l-e}
\lambda^\epsilon:\mathbb{T}_*\rightarrow\mathbb{R},\qquad
\lambda^\epsilon(\theta):=\underset{\gamma\in\Gamma}{\sum}
\mathfrak{h}^\epsilon(\gamma)e^{-i<\theta,\gamma>}\,.
\eeq
\end{definition}

The conclusion of this subsection is contained in:
\begin{proposition}\label{effective-H}
There exists $\epsilon_0>0$ such that, for $\epsilon \in [0,\epsilon_0]$ and $\kappa \in [0,1]\,$, 
the Hausdorff distance between the spectra of the  magnetic quasi-band Hamiltonian $\widetilde{\pi}^{\epsilon,\kappa}_0
H^{\epsilon,\kappa}\widetilde{\pi}^{\epsilon,\kappa}_0$ and $\mathfrak{Op}^{\epsilon,\kappa}(\lambda^\epsilon)$ is of order  $\kappa\epsilon\,$.
\end{proposition}
\begin{proof}
 Proposition \ref{propo-marts1} and a Schur type estimate imply  that the spectrum of our magnetic quasi-band Hamiltonian is at a Hausdorff distance of order $\kappa \epsilon$ from the spectrum of the matrix with elements 
\begin{equation}\label{marts2}
\mathcal{E}^{\epsilon,\kappa}(\alpha,\beta):=\widetilde{\Lambda}^{\epsilon,\kappa}(\alpha,\beta)\Lambda^\epsilon(\alpha,\beta)\, \mathfrak{h}^\epsilon(\alpha-\beta)=\Lambda^{{\epsilon,\kappa}}(\alpha,\beta)\, \mathfrak{h}^\epsilon(\alpha-\beta)
\end{equation}
acting on $\ell^2(\Gamma)$. 

We may extend  the function $\lambda^\epsilon$ from \eqref{def-l-e} as a $\Gamma_*$-periodic function defined on all $\X^*$ and thus as a symbol defined on $\Xi$, constant with respect to the variables in $\X$.  Using the natural unitary isomorphism $L^2(\X)\rightarrow \ell^2(\Gamma)\otimes L^2(E)$ we may compute the integral kernel of $\mathfrak{Op}^{\epsilon,\kappa}(\lambda^\epsilon)$ as in \cite{CP-2} and obtain 
$$\Lambda^{{\epsilon,\kappa}}(\alpha+\underline{x},\beta+\underline{x})\, \mathfrak{h}^\epsilon(\alpha-\beta)\, \delta(\underline{x}-\underline{x}'),\quad \forall \underline{x}\,, \underline{x}'\in E\,,\quad \forall  \alpha\,,\beta\in \Gamma\,.$$ 
Following \cite{CP-2}, it turns out that one can perform an $(\epsilon,\kappa)$-dependent unitary gauge transform in $\ell^2(\Gamma)\otimes L^2(E)$ such that after conjugating $\mathfrak{Op}^{\epsilon,\kappa}(\lambda^\epsilon)$ with it, the rotated operator will be (up to an error of order $\kappa\epsilon$ in the norm topology) given by an operator whose integral kernel is given by 
$\mathcal{E}^{\epsilon,\kappa}(\alpha,\beta)\,\delta(\underline{x}-\underline{x}')\,$. This operator is nothing but $\mathcal{E}^{\epsilon,\kappa}\otimes \bb1\,$, and it is isospectral with the matrix $\mathcal{E}^{\epsilon,\kappa}\,$. 
\end{proof}

\section{The magnetic quantization of the magnetic Bloch function.}\label{S-magn-q-Bloch-func}

\subsection{Study of the magnetic Bloch function $\lambda^\epsilon$.}

Let us recall from \eqref{evenness} and \eqref{lbtheta}  that $\lambda_0\in C^\infty(\mathbb{T}_*;\mathbb{R})$ is even
and has  a non-degenerate
absolute minimum in $0\in\mathbb{T}_*\,$.
Thus in a neighborhood
of $0\in\mathbb{T}_*$ we have the Taylor expansion
\beq\label{l-dev}
\lambda_0(\theta)\,=\,\underset{1\leq
j,k\leq2}{\sum}a_{jk}\theta_j\theta_k\,+\,\mathscr{O}(|\theta|^4)\,,\qquad
a_{jk}:=\big(\partial^2_{jk}\lambda_0\big)(0)\,.
\eeq
\begin{proposition}\label{rho-epsilon}
For $\lambda^\epsilon$ defined in \eqref{def-l-e}, there exists $\epsilon_0 >0$ such that
$$
\lambda^\epsilon(\theta)\,=\,\lambda_0(\theta)\,+\,\epsilon\rho^\epsilon(\theta)\,,
$$
with $\rho^\epsilon\in BC^\infty(\mathbb{T}_*)$ uniformly in
$\epsilon\in[0,\epsilon_0]$ and such that
$\rho^\epsilon-\rho^0=\mathscr{O}(\epsilon)\,$.
\end{proposition}
\begin{proof}
From the  definition of $\lambda^\epsilon$ in  Definition \ref{magn-Bloch-func} we have
$$
\lambda^\epsilon(\theta)=\underset{\gamma\in\Gamma}{\sum}
\left\langle\psi^{\epsilon}_0\,,\,\Lambda^\epsilon(\cdot,\gamma)\tau_{-\gamma}
H^{\epsilon}\psi^{\epsilon}_0\right
\rangle_{\mathcal{H}}e^{-i<\theta,\gamma>}\,.
$$
Due to Corollary \ref{est-psi-phi}, there exists $f^\epsilon\in\mathscr{S}(\X)$
such that $
\psi^{\epsilon}_0=\phi_0+\epsilon f^\epsilon\,$, 
and for any $m\in\mathbb{N}$ there exists $C_m>0$ such that, for any $\epsilon \in [0,\epsilon_0]$\,, 
$$
\underset{x\in\X}{\sup}<x>^m|f^\epsilon(x)|\leq\,C_m\,.
$$
With $A^0(x)=(1/2)\big(-B_0x_2,B_0x_1\big)\,$, we have
$$
H^\epsilon =(-i\nabla_x-\epsilon A^0(x))^2+V\,=\,H^0+2i\epsilon A^0\cdot\nabla_x+\epsilon^2(A^0)^2\,.
$$
Thus
\begin{equation*}
\begin{array}{l}
\hspace{-1cm} \left\langle\psi^{\epsilon}_0\,,\,\Lambda^\epsilon(\cdot,\gamma)\tau_{-\gamma}
H^{\epsilon}\psi^{\epsilon}_0\right
\rangle_{\mathcal{H}} \\
\hspace{1cm} =
\left\langle\tau_{\gamma}\Lambda^\epsilon(\gamma,\cdot)\psi^{\epsilon}_0\,,\,
H^0\psi^{\epsilon}_0\right
\rangle_{\mathcal{H}}+\left\langle\tau_{\gamma}\Lambda^\epsilon(\gamma,\cdot)\psi^{\epsilon}_0\,,\,\left(2i\epsilon A^0\cdot\nabla_x+\epsilon^2(A^0)^2\right)\psi^{\epsilon}_0\right
\rangle_{\mathcal{H}}
\\
\hspace{1cm}  =\left\langle\tau_{\gamma}\Lambda^\epsilon(\gamma,\cdot)\phi_0\,,\,
H^{0}\phi_0\right
\rangle_{\mathcal{H}}+C_m\epsilon<\gamma>^{-m}
\\
\hspace{1cm}  =\left\langle\tau_{\gamma}\phi_0\,,\,
H^{0}\phi_0\right
\rangle_{\mathcal{H}}+C^\prime_m\, \epsilon\, <\gamma>^{-m}\,.
\end{array}
\end{equation*}
In the last two estimates we have used the rapid decay of the functions $\psi^\epsilon_0$ and $\phi_0$ and the fact that with our choice of gauge 
$$ \left|\Lambda^\epsilon(\gamma,x)-1\right|\leq \epsilon\, |\gamma|\, |x|\,.$$
Hence
$$
\lambda^\epsilon(\theta)=\underset{\gamma\in\Gamma}{\sum}
\left\langle\tau_{\gamma}\phi_0\,,\,
H^{0}\phi_0\right
\rangle_{\mathcal{H}}e^{-i<\theta,\gamma>}+ \mathcal O (\,\epsilon)=\lambda_0(\theta)+\mathcal O (\,\epsilon)\,,
$$
 in  $C^m(\mathbb{T}_*)$  for any $m\,$.\\
 Moreover, for any $m\in \mathbb{N}\,$, one can differentiate $m$-times the Fourier 
series of $\lambda^\epsilon(\theta)$, term by term, because of the 
exponential decay of its Fourier coefficients. Then the Fourier series 
of the derivative has an asymptotic expansion in $\epsilon$ which starts 
with $\lambda_0^{(m)}(\theta)$, just as in the case with $m=0$. In other 
words, differentiation in $\theta$ commutes with taking the asymptotic 
expansion in powers of $\epsilon$. 
\end{proof}

Considering the function $\lambda_\epsilon$ as a $\Gamma_*$-periodic function on $\X^*$, a consequence of Proposition ~\ref{rho-epsilon} is that the modified Bloch eigenvalue $\lambda^\epsilon\in C^\infty(\X^*)$ will also have an isolated non-degenerate minimum at some point $\theta^\epsilon\in\X^*$ close to $0\in\X^*$. 
More precisely, if we denote by $\mathbf{a}$ the
$2\times 2$ positive definite matrix $(a_{jk})_{1\leq j,k\leq2}$ introduced in \eqref{l-dev} and by
$\mathbf{a}^{-1}$ its inverse, we have $$
\theta^\epsilon_j\,=\epsilon \, \hat \theta_j\,+\,\mathcal{O}(\epsilon^2)\,,$$ 
with
$$
\hat \theta_j := 
\,- \underset{1\leq
k\leq2}{\sum}\big(\mathbf{a}^{-1}\big)_{jk}
\big(\partial_k\rho_0\big)(0)\,.
$$
Then we can write the Taylor expansion at $\theta^\epsilon$
\beq \label{a-epsilon} 
\begin{array}{ll}
\lambda^\epsilon(\theta)\,&=\,\lambda^\epsilon(\theta^\epsilon)\,+\underset{
1\leq
j,k\leq2}{\sum}
\big(\partial_j\partial_k\lambda^\epsilon\big)(\theta^\epsilon)
(\theta_j-\theta^\epsilon_j)
(\theta_k-\theta^\epsilon_k)\\ 
&\qquad +\,\underset{1\leq j,k,\ell \leq2}{\sum}
\big(\partial_j\partial_k\partial_\ell \lambda^\epsilon\big)(\theta^\epsilon)
(\theta_j-\theta^\epsilon_j)(\theta_k-\theta^\epsilon_k)
(\theta_\ell -\theta^\epsilon_\ell )\,+\,\mathcal{O}(|\theta-\theta^\epsilon|^4)\,. \end{array}
\eeq
Using the evenness of $\lambda_0$ we get:
\beq \label{formenormale}
\lambda^\epsilon (\theta) - \lambda^\epsilon (\theta^\epsilon) =
 \,\underset{
1\leq j,k\leq2}{\sum} a_{jk}^\epsilon  (\theta_j-\theta^\epsilon_j)(\theta_k-\theta^\epsilon_k)  + \epsilon 
\mathscr{O}(|\theta-\theta^\epsilon|^3) + \,\mathcal{O}(|\theta-\theta^\epsilon|^4)\,,
\eeq
where 
$$
\lambda^\epsilon (\theta^\epsilon) = \epsilon \rho^0(0) +\mathcal{O}(\epsilon^2)\,,
$$
and 
\begin{align*}
a^\epsilon_{jk}\,
&=\,a_{jk}\,+\,\epsilon 
\big(\partial_j
\partial_k\rho^0\big)(0) +\,\mathcal{O}(\epsilon^2)\, .
\end{align*}
Hence, after a shift of energy  and a change of variable $\theta \mapsto (\theta -\theta_\epsilon)$, the new $\lambda^\epsilon$ has the same structure as $\lambda_0$  except that we have lost the symmetry \eqref{evenness}.\\
There exists $\epsilon_0 >0$ such that, for $\epsilon\in[0,\epsilon_0]\,$, we
can  choose a local coordinate system on a neighborhood of
$ \theta^\epsilon\in \X^*$ that diagonalizes the symmetric positive
definite matrix $a^\epsilon$ and we denote by $0<m^\epsilon_1\leq
m^\epsilon_2$ its eigenvalues. We denote by $0<m_1\leq m_2$ the two
eigenvalues of the matrix $a_{jk}$ and notice that
$$
m^\epsilon_j\,=\,m_j\,+\,\epsilon\mu_j\,+\,\mathscr{O}(\epsilon^2) \mbox{  for  } j=1,2\,,
$$
with $\mu_j$ explicitly computable in terms of the matrix  $a_{jk}$ and 
$\big(\partial_j
\partial_k\rho^0\big)(0)\,$.  

\subsection{Spectral analysis of the model operator}

\subsubsection{Quadratic approximation}
In studying the bottom of the spectrum of the operator
$\mathfrak{Op}^{\epsilon,\kappa}(\lambda^\epsilon)$ we shall start with the quadratic term given by the Hessian close to the minimum (see \eqref{a-epsilon}). We introduce
\beq\label{marts11}
m^\epsilon\,:=\,\sqrt{m^\epsilon_1m^\epsilon_2}\,=\,m+\epsilon \, \mu^*(\epsilon)\,,
\eeq
where $m :=\sqrt{m_1m_2}$ is the square root of the determinant of the
matrix
$a_{jk}$ and $\mu^*(\epsilon)$ is uniformly bounded for
$\epsilon\in[0,\epsilon_0]$.
Our goal is  to obtain spectral information concerning the Hamiltonian
$\mathfrak{Op}^{\epsilon,\kappa}(\lambda^\epsilon)$ starting from the spectral
information about $\mathfrak{Op}^{\epsilon,\kappa}(h_{m^\epsilon})$ with
\begin{equation}\label{defhm}
h_{m^\epsilon}(\xi)\,:=\,m^\epsilon_1\xi_1^2\,+\,m^\epsilon_2\xi_2^2\,,
\end{equation}
defining an elliptic symbol of class $S^2_1(\Xi)^\circ$
(i.e. that does not depend on the configuration space variable, see \eqref{S-circ} in Appendix \ref{SS-const-m-field}). 

\subsubsection{A perturbation result}
 Now we compare the bottom of the spectrum of the magnetic Hamiltonians $\mathfrak{Op}^{\epsilon,\kappa}(h_{m^\epsilon})$ with the one of the constant field magnetic Landau operator $\mathfrak{Op}^{\epsilon,0}(h_{m^\epsilon})$.

First, we have to perform a dilation.
Starting from the initial model (with the magnetic field in \eqref{Bek}), we make the change of variable $y = \sqrt{\epsilon} x$ and factorize an $\epsilon$. This leads to a scaled Landau operator with magnetic field $B_0 + \kappa B (\sqrt{\epsilon} x)\,$. Then, using also  \eqref{marts11},  we can prove the following statement:

\begin{proposition}\label{model-spectrum} For any compact set {$K$} in
$\mathbb R$, there exist  $\epsilon_{K}>0\,$, $C>0$  and
$\kappa_{K}\in(0,1]\,$, such  that for any
$(\epsilon,\kappa)\in[0,\epsilon_K]\times[0,\kappa_K]\,$,  the spectrum of the operator
$\mathfrak{Op}^{\epsilon,\kappa}(h_{m^\epsilon})$ in {$\epsilon K$} is
contained in bands of
width $ C \kappa\epsilon $ centered at 
$\{(2n+1)\, \epsilon\, m^\epsilon\, 
B_0\}_{n\in\mathbb{N}}\,$.
\end{proposition}

 The above result follows from the following slightly more general
proposition.
\begin{proposition}\label{pertprop}
Assume that  $B(x)=B_0+b(x)$ where $B_0>0$ and $b\in BC^1(\X)$ and consider 
$$
L_b\,:= \big(-i\partial_1-a_1(x)\big)^2\,+\,
\big(-i\partial_2-B_0 x_1-a_2(x)\big)^2\,,
$$
where $a(x)=a(x,0)$ with
$$ a(x,x')=(a_1(x,x'),a_2(x,x'))= \left(\int_0^1 ds \; s\; b(x'+s(x-x'))\right) \big ( -x_2+x_2',x_1-x_1'\big )\,.$$
With  
$\beta :=||b||_{C^1(\X)}$,   for any $N\in \mathbb N\,$,  there exist  $C>0$ and $\beta_0 >0$ such that, for any $0\leq \beta \leq \beta_0\,,$ we have:
\beq \label{aaa1}
\sigma(L_b)\cap [0, 2(N+1)\, B_0]\subset \bigcup_{j=0}^{N} \big [(2j+1)B_0-C\beta,(2j+1)B_0+ C\beta\big ]\, .
\eeq
\end{proposition}
\begin{proof}
 Define 
$\phi_b(x,x')$ to be the magnetic flux generated by the magnetic field $b$ through the triangle with vertices at $0$, $x$ and $x'$. Let $\widehat A_0(x)=(0,B_0x_1)$. We have the following variant of  \eqref{26-02-2}, for the composition of $L_b$ and  the multiplication operator by 
$ e^{i\phi_b(\cdot ,x')}$, with $x'\in\X$ arbitrary:
\begin{equation} \label{aaa2}
L_b \,  e^{i\phi_b(x,x')}=e^{i\phi_b(x,x')}\big ( L_0+a^2(x,x')-2a(x,x')\cdot (-i\nabla_x-\widehat A_0(x))+i\nabla_x a(x,x')\big )\,,
\end{equation}

Let $z\not\in \sigma(L_0)$.  Denote by $K_0(x,x';z)$ the integral kernel of $(L_0-z)^{-1}$. Let  $S_b(z)$  be the operator whose   integral kernel is given by
$$S_b(x,x';z):=e^{i\phi_b(x,x')}K_0(x,x';z)\,.$$
Using \eqref{aaa2} and the fact that $\phi_b(x,x)=0\,$, one can prove that: 
\begin{align} \label{aaa3}
[(L_b-z)S_b(\cdot,x';z)](x)=\delta(x-x')+T_b(x,x';z)\,,
\end{align}
where the kernel $T_b(x,x';z)$ generates an operator $T_b$ such that, for any compact set $K$ in the resolvent set of $L_0\,$, there exists  $C_K>0$ s. t. 
$$\sup_{z\in K}||T_b(z)||\leq C_K \beta\,.$$
From  \eqref{aaa3} it follows that there exists $\beta_0 >0$, such that if $\beta \in [0,\beta_0]$ then $K$ is in the resolvent set of $L_b$ and we have:
$$(L_b-z)^{-1}=S_b(z)(1+T_b(z))^{-1}\,.$$
In particular,   there exists  $C_K>0$ s. t. 
\begin{align} \label{aaa4}
\sup_{z\in K}\Vert (L_b-z)^{-1}-S_b(z)\Vert \leq C_K \beta\,.
\end{align}
By a Riesz integral on a contour $\Gamma_j$ encircling the eigenvalue $(2j+1)B_0$ of $L_0\,$, we can define the band operator 
$$L_{b,j}:=\frac{i}{2\pi}\int_{\Gamma_j}z \, (L_b-z)^{-1}\,dz$$
acting on the range of the projector  $P_{j,b}=\frac{i}{2\pi}\int_{\Gamma_j} (L_b-z)^{-1}dz\,$.\\
 From \eqref{aaa4} we get:
\begin{equation}\label{marts9}
\big (L_{b,j}-(2j+1)B_0\big )P_{j,b}=\mathcal{O}(\beta)\,,
\end{equation}
which shows that the spectrum of $L_b$ in $[0,2(N+1)\, B_0]$ is contained in intervals of width of order $\beta$ centered around the Landau levels $(2j+1)\,B_0\,$. 
\end{proof}
\begin{remark}
 Let us introduce  $r^{\epsilon,\kappa}(\z)$, resp. $r^{\epsilon}(\z)$, in $
S^{-2}_1(\Xi)$ such that
\beq\label{nou-1}
\left(\mathfrak{Op}^{\epsilon,\kappa}(h_{m^\epsilon})-\z\right)^{-1}=:\mathfrak{
Op}^{\epsilon,\kappa}\big({r^{\epsilon,\kappa}}(\z)\big),\quad
\left(\mathfrak{Op}^\epsilon(h_{m^\epsilon})-\z\right)^{-1}=:\mathfrak{Op}
^\epsilon\big(r^\epsilon(\z)\big)\,
\eeq
for $z$ in the resolvent set of $\mathfrak{Op}^{\epsilon,\kappa}(h_{m^\epsilon})$, resp. $\mathfrak{Op}^\epsilon(h_{m^\epsilon})\,$.\\
 Coming back to Proposition \ref{model-spectrum},  after a scaling, the estimate \eqref{aaa4} combined with \eqref{nou-1} leads, for $j\in \mathbb N$, to the existence of $C_j$ such that
\beq\label{nou-3}
\sup_{|\z-(2j+1)\epsilon m^\epsilon \,B_0|=\epsilon m^\epsilon \,B_0}\|\mathfrak{Op}^{\epsilon,\kappa}\big(r^{\epsilon,\kappa}(
\z)-r^\epsilon(
\z)\big)\|\leq C_j\, \kappa\epsilon^{-1\,}.
\eeq

\end{remark}

\subsection{The resolvent of $\mathfrak{Op}^{\epsilon,\kappa}(\lambda^\epsilon)$}\label{SSS-quasi-inv}

 Using the results of Section \ref{S-eff-band-Ham}, we need for finishing the proof of
Theorem \ref{mainTh}  a spectral analysis of
$\mathfrak{Op}^{\epsilon,\kappa}(\lambda^\epsilon)$.  In the rest of this section
we use  definitions, notation and results  from the
Appendix~\ref{A-sl-var-symb} dedicated to the magnetic Moyal calculus.

\paragraph{Cut-off functions near the minimum.}
We introduce a cut-off around $\theta=0\,$.  We choose an even function $\chi$ in
$C_0^\infty(\R)$ with $0\leq\chi\leq 1\,$, with  support in $(-2,+2)$ and such that $\chi(t)=1$ on $[-1,+1]\,$.  We note that the
following relation holds on $[0,+\infty)$
\beq\label{28aug-2}
(1-t)\big(\chi(t)-1\big)\geq0\,.
\eeq
For $\delta>0$ we define
\beq\label{g-delta}
g_{1/\delta}(\xi):=\chi\big(h_{m^\epsilon}(\delta^{-1}\xi)\big),\quad \xi\in \X^*\,,
\eeq
 where $h_{m^\epsilon}$ is defined in \eqref{defhm}. \\
Thus $0\leq g_{1/\delta}\leq1$ for any $\delta>0$ and
\beq\label{g-delta-p}
g_{1/\delta}(\xi)\ =\ 
\left\{
\begin{array}{ccl}
1&\text{ if }&|\xi|^2\leq (2m_2)^{-1}\delta^2\,,\\
0&\text{ if }&|\xi|^2\geq m_1^{-1}\delta^2\,.
\end{array}
\right.
\eeq
We choose {$\delta_0$} such that 
\begin{equation}\label{delta-0} D(0,\sqrt{2m_1^{-1}}{\delta_0}) \subset
E^{^{^{\hspace{-1.5mm}\circ}}}_* \,,
\end{equation}
where $D(0,\rho)$ denotes the disk centered at $0$ of radius $\rho$ and  $E^{^{^{\hspace{-1.5mm}\circ}}}_* $ denotes the interior of $E_*\,$.\\ 

For any
$\delta\in (0,\delta_0]\,,$ $g_{1/\delta}{\in}C_0^\infty(E_*)$ 
and we shall consider it as an element of $C_0^\infty(\X^*)$ by extending  it
by  $0$. We introduce 
\begin{align}\label{marts5}
\widetilde{g}_{1/\delta}(\xi):=\sum_{ \gamma\in \Gamma^*} g_{1/\delta}(\xi-\gamma)\,,
\end{align}
the $\Gamma_*$-periodic
continuation of $g_{1/\delta}$ to $\X^*$. The multiplication with
these functions defines bounded linear maps in  $L^2(\mathbb{T}_*)$ or
in $ L^2(\X^*)$ with operator norm bounded uniformly for $\delta\searrow0\,$.
For any $\delta\in(0,\delta_0]$ we introduce:
\beq \label{defdeltacirc}
\delta^\circ:=\sqrt{m_1/2m_2}\; \delta\,,
\eeq 
so we have
 $$g_{1/\delta^\circ}=g_{1/\delta}\;
g_{1/\delta^\circ}\,.$$

\paragraph{Behavior near the minimum.} We use now the expansion of
$\lambda^\epsilon$ near its minimum
\eqref{formenormale} and the coordinates that diagonalize the Hessian (see the discussion after \eqref{formenormale}) in order to get
\beq\label{marts4}
\lambda^\epsilon(\xi)\,=\,\big(m^\epsilon_1\xi_1^2+m^\epsilon_2\xi_2^2\big)\,+\,
\epsilon\, \mathcal{O}
(|\xi|^3)\,+\,\mathcal{O}(|\xi|^4)\,.
\eeq
In order to localize near the minimum $\xi=0\,$, we use the cut-off function
$g_{1/\delta}$ and write
$$
\lambda^\epsilon\ = 
\lambda^\epsilon_\circ\,+\,\widetilde{\lambda}^\epsilon\,.
$$
with
$$ 
\lambda^\epsilon_\circ:= g_{1/\delta}\, \lambda^\epsilon\,\mbox{ and } \, \widetilde{\lambda}^\epsilon:=(1-g_{1/\delta}
)\, \lambda^\epsilon\,.
$$
We introduce a second scaling:
\begin{equation} \label{ScalHyp}
\epsilon =\delta^\mu\,,  \mbox{ for some  } \mu \in(2,4)\,. 
\end{equation}

Due to \eqref{marts4} and the inequality $\mu+3>4\,$, we can find a  bounded family
$\{\mathfrak{f}^{
\delta}\}_{\delta\in(0,\delta_0]}$ in $ C^\infty_0(\X^*)$ such that 
\beq\label{4.0}
\lambda^\epsilon_\circ(\xi)=
g_{1/\delta}(\xi)\, h_{m^\epsilon}(\xi)+\delta^4 \, \mathfrak{f}^{
\delta}(\xi)\,.
\eeq

\paragraph{The shifted operator outside the minima.}~\\
For the region outside the minima, we  need the operator
$\mathfrak{Op}^{\epsilon,\kappa}\big(\lambda^\epsilon\, +\,(\delta^\circ)^2\widetilde
{g}_
{1/\delta^\circ}\big)$ with 
$\delta^\circ$ as in \eqref{defdeltacirc} and $\widetilde
{g}_
{1/\delta^\circ}$ as in \eqref{marts5}. 
The inequality \eqref{28aug-2} implies that there exists 
$ C>0$ such that
$$
\lambda^\epsilon(\xi)+(\delta^\circ)^2 \, \widetilde{g}_{1/\delta^\circ}(\xi)\,\geq\,C\,  (\delta^{\circ})^2 = C \frac{m_1}{2m_2}\,  \delta^2\, .
$$
Hence there exists $C' >0$ such that for any $\delta \in (0,\delta_0)$
\begin{equation}
\mathfrak{Op}^0\big(\lambda^\epsilon+(\delta^\circ)^2\, \widetilde{g}_
{1/\delta^\circ}\big)\geq  C^\prime\,\delta^2\, \bb1\, .
\end{equation}

As the magnetic field is slowly variable and the symbol
$\lambda^\epsilon+(\delta^\circ)^2\widetilde{g}_
{1/\delta^\circ}$ is $\Gamma^*$-periodic in $S^0_0(\Xi)$, by
Corollary 1.6 in \cite{CP-2}  there exists $\epsilon_0>0$ and for $(\epsilon,
\kappa,\delta)\in[0,\epsilon_0]\times[0,1]\times (0,\delta_0]\,$, 
   some constant
$C^\prime(\epsilon,\delta)>0$ such that:
\beq\label{l-bd-est}
\mathfrak{Op}^{\epsilon,\kappa}\big(\lambda^\epsilon+(\delta^\circ)^2 \, \widetilde{
g}_
{1/\delta^\circ}\big)\geq \left(C'\,\delta^2-C^\prime(\epsilon,\delta)\,\epsilon\,\right) \,\bb1\,.\eeq
 The proof of   Corollary 1.6 in \cite{CP-2} gives a control of $C'(\epsilon,\delta)$ by:
 \begin{equation*}
\begin{array}{ll}
C^\prime(\epsilon,\delta)
 & \leq \,C_1 \Big[\underset{|\alpha|\leq 2}{\max}
\left\|\mathcal{F}^-_{\mathbb{T}_*}
\big(\partial_\xi^\alpha\lambda^\epsilon\big)\right\|_{{ \ell^1(\Gamma)}}
+(\delta^\circ)^2\underset{|\alpha|\leq 2}{\max}\left\|\mathcal{F}^-_{\mathbb{T}_*}\big[\partial_\xi^\alpha\big(
g_{1/\delta^\circ}\big)\big]\right\|_{{ \ell^1(\Gamma)}}\Big]\, .
\end{array}
\end{equation*}
The differentiation of  $g_{1/\delta^\circ}$ with respect to $\xi$ produces $|\alpha|\leq 2$ negative powers of $\delta$. Denote by $g\equiv g_1$. Then we change the integration variable $\xi\mapsto \delta^\circ \xi$ under the Fourier transform (this gives us an extra factor of $(\delta^\circ)^2$) and we obtain: 
\begin{align*}
C^\prime(\epsilon,\delta)& \leq\,C_2
\Big[1
+\underset{|\alpha|\leq 2}{\max} (\delta^\circ)^2\sum_{\gamma\in \Gamma}\left \vert
\int_{\mathbb{T}_*}e^{i\langle \delta^\circ\gamma,\xi\rangle }(\partial_\xi^\alpha
g) d\xi \right\vert \Big].
\end{align*}
We notice
that $\partial^\alpha g\in C^\infty_0 (\X^*)$ has its support strictly
included in the dual
elementary cell, so that the above series is just a Riemann sum for the integral representing
the norm $\left\|\mathcal{F}^-_{\X^*}\partial^\alpha g\right\|_{L^1(\X)}$. Hence:
\begin{align*}
C^\prime(\epsilon,\delta)
 & \leq\,C_3 \Big[1
+\underset{|\alpha|\leq 2}{\max}\left\|\mathcal{F}^-_{\X^*}\big(\partial_\xi^\alpha
g\big)\right\|_{L^1(\X)}\Big] \,.
\end{align*}
Thus we have shown the existence of $C_4 >0$ such that 
$
C^\prime(\epsilon,\delta)\leq C_4\,
$, uniformly in $\epsilon$. 
Hence any
$\z\in (-\infty,C'\delta^2-C_4 \epsilon)$ is in the resolvent  set of 
$\big(\mathfrak{Op}^{\epsilon,\kappa}
\big(\lambda_\epsilon+(\delta^\circ)^2\, \widetilde{g}_
{1/\delta^\circ}\big) \big)$ and
we denote by $r_{\delta,\epsilon,\kappa}(\z)$ its
magnetic symbol.

 Due to the choice of $\epsilon$ made in \eqref{ScalHyp}, it follows that $C'\delta^2-C_4 \epsilon$ is of order $\epsilon^{2/\mu}\,$,  i.e. much larger than $\epsilon$ as $\epsilon \rightarrow 0\,.$ 
  Since we are interested in inverting our operators for $\z$ in an interval
of the form $[0,c\epsilon]$ for some $c>0$, we conclude that the distance
between $\z$ and the bottom of the spectrum of $\mathfrak{Op}^{\epsilon,\kappa}
\big(\lambda_\epsilon+(\delta^\circ)^2\, \widetilde{g}_
{1/\delta^\circ}\big)$ is of order $\delta^2=\epsilon^{2/\mu}$. Thus given any $C'>0$, there exists $\epsilon_0$  and $C>0$ such that for every $\epsilon<\epsilon_0$ and $\kappa<1$ we have 
\beq\label{r-delta-est}
\sup_{\z\in [0,C'\epsilon]}\|r_{\delta,\epsilon,\kappa}(\z)\|_{B_{\epsilon,\kappa}}\leq C\,\epsilon^{-2/\mu}=C\delta^{-2}\,.
\eeq

\paragraph{Definition of the \textit{quasi-inverse}.}
Let us fix  $\mu\in (2,4)$ and some compact set $K\subset\mathbb{C}$ such that:
\begin{equation}\label{condK}
K\subset \mathbb C \setminus \{(2n+1)m \, B_0\}_{n\in\mathbb{N}}\,.
\end{equation} 

We deduce from Proposition \ref{model-spectrum}  that there exist $\epsilon_K >0$  and $\kappa_K\in[0,1]$ such that for $(\epsilon,\kappa) \in [0,\epsilon_K] \times [0,\kappa_K]$ and  
 for $a\in K\,$, the point $\epsilon a\in\mathbb{C}$ belongs to the
resolvent set of
$\mathfrak{Op}^{\epsilon,\kappa}(h_{m^\epsilon})\,$.  We denote by 
$r^{\epsilon,\kappa}( \epsilon a)$  the magnetic  symbol of the resolvent of 
$\mathfrak{Op}^{\epsilon,\kappa}(h_{m^\epsilon})-\epsilon a \,$,
i.e.
\beq\label{z-epsilon-a}
\big(\mathfrak{Op}^{\epsilon,\kappa}(h_{m^\epsilon})-\epsilon a\big)^{
-1}\, := \,\mathfrak{Op}^{\epsilon,\kappa}(r^{\epsilon,\kappa}( \epsilon a))\,.
\eeq
 For {$ a\in K$}  we want to define the following symbol in
$\mathscr{S}^\prime(\X^*)$ as the sum
of the series on the right hand side:
\beq\label{q-rez}
\widetilde{r}_{\lambda}(\epsilon a):=\underset{\gamma^*\in\Gamma_*}{\sum}
\tau_{\gamma^*}
\big(g_{1/\delta}\,\,\sharp^{\epsilon,\kappa}\,
r^{\epsilon,\kappa}(\epsilon a)\big)+\left(1-\widetilde{g}_{1/\delta}\right)\sharp^{
\epsilon , \kappa } \, r_{\delta,\epsilon,\kappa}(\epsilon a)\,,\quad \delta=\epsilon^{1/\mu}.
\eeq
Then  the proof of Theorem \ref{mainTh} will be  a consequence of the following key result:

\begin{proposition}\label{mainP}  Let $\mu=3$ in \eqref{q-rez}. For any compact set $K$  satisfying  
\eqref{condK},
there exist $C>0$,  $\kappa_0\in(0,1]$ and $\epsilon_0>0$ such that
for  $(\kappa,\epsilon,a) \in[0,\kappa_0] \times (0,\epsilon_0]\times K $,   we have
$$\| \mathfrak{Op}^{\epsilon,\kappa}(\widetilde{r}_{
\lambda}
(\epsilon a))\|\leq C\epsilon^{-1}\,,
$$ and 
\begin{align}\label{marts7}
\big(\lambda_\epsilon-\epsilon a\big)\,\,\sharp^{\epsilon,\kappa}\,\,\widetilde{r}_{
\lambda}
(\epsilon a)\ =\
1\,+\,\mathfrak{r}_{\delta,a},\;  \mbox{ with } \quad \|\mathfrak{Op}^{\epsilon,\kappa}
(\mathfrak{r}_{\delta,a})\|\,\leq\,C\;\epsilon^{1/3}\,.
\end{align}
For $N>0\,$, there exist $C$, $\epsilon_0$ and $\kappa_0$ such that  the spectrum of 
$\mathfrak{Op}^{\epsilon,\kappa}(\lambda_\epsilon)$ in  $[0,(2N+2)mB_0\epsilon]$  consists of spectral islands centered at $(2n+1)mB_0\epsilon \,$, $0\leq n \leq N\,$, with a width bounded by $ C\,(\epsilon \kappa +\epsilon^{4/3})\,$.
\end{proposition}

\subsection{Proof of Proposition \ref{mainP}}

For the time being, we allow $\mu$ to belong to the interval $(2,4)$ and we will explicitly mention when we make the particular choice $\mu=3$. We begin with proving the convergence of the first series in
\eqref{q-rez} by analyzing the properties of the symbol
$r^{\epsilon,\kappa}(\epsilon a)$. Then  assuming that \eqref{marts7} is
true and
using the properties of $r^{\epsilon,\kappa}( \epsilon a)$ we will show how this
implies
the second part of Proposition \ref{mainP}, namely the localization of
the spectral islands   (see Subsection \ref{SSS-nou}). In the last part of this
section we shall prove \eqref{marts7}.

\subsubsection{Properties of $r^{\epsilon,\kappa}( \epsilon a)$.}
From  its definition in \eqref{z-epsilon-a} and from Proposition 6.5 in
\cite{IMP2} we know that $r^{\epsilon,\kappa}( \epsilon a)$ defines a 
symbol of class $ S^{-2}_1(\Xi)$.  Moreover, from Proposition \ref{model-spectrum} we have:
\begin{equation}\label{maii1}
\left\|\mathfrak{Op}^{\epsilon,\kappa}(r^{\epsilon,\kappa}( \epsilon a))\right\|\leq C(K) \,  \epsilon^{-1}.
\end{equation}

In order to study the convergence of the series in \eqref{q-rez} we need
a more involved expression for $r^{\epsilon,\kappa}(\epsilon a)$ obtained by
using the  resolvent  equation with respect to some point far from the spectrum. All the operators
$\mathfrak{Op}^{\epsilon,\kappa}(h_{m^\epsilon})$ (indexed by
$(\epsilon,\kappa)\in[0,\epsilon_0]\times[0,\kappa_0]$) are non-negative due to the diamagnetic inequality, hence 
the point $\z=-1$ belongs to their resolvent sets and:
$$
\left\|\big(\mathfrak{Op}^{\epsilon,\kappa}(h_{m^\epsilon})+\bb1\big)^{-1}
\right\|\leq 1\, .
$$
We denote by $r^{\epsilon,\kappa}_1$ the magnetic symbols of these resolvents, so 
$$
\big(\mathfrak{Op}^{\epsilon,\kappa}(h_{m^\epsilon})+\bb1\big)^{-1}\,=\,
\mathfrak{Op}^{\epsilon,\kappa}(r^{\epsilon,\kappa}_1)\, .
$$
We shall consider all these symbols as \textit{slowly varying symbols on $\Xi$}
(see Appendix \ref{A-cl-sl-var}).

{In the sequel we shall consider families of the form $\left\{
r^{\epsilon,\kappa}_1
\right\}_{(\epsilon,\kappa)\in[0,\epsilon_0]\times[0,\kappa_0]}$ either as
families of symbols in some class $S^m_\rho(\Xi)$ indexed by two indices or
alternatively as families of \textit{slowly varying symbols} in $
S^{m}_\rho(\Xi)^\bullet$ indexed by the index $\kappa\in[0,\kappa_0]$ (see
Definition \ref{S-epsilon}).}
\begin{lemma}\label{2sept-1}
The family $\left\{ r^{\epsilon,\kappa}_1 \right\}_{(\epsilon,\kappa)\in[0,\epsilon_0]\times[0,\kappa_0]}$ is bounded in
$ S^{-2}_1(\Xi)$ and {also} in 
$ S^{-2}_1(\Xi)^\bullet$.
\end{lemma}
\begin{proof}
We have
$r^{\epsilon,\kappa}_1=\big(h_{m^\epsilon}+1\big)^-_{\epsilon,
\kappa}$ with $ h_{m^\epsilon}\in S^2_1(\Xi)^\circ$.
A straightforward verification using the arguments in 
the proof of Proposition 6.7 in \cite{IMP2} gives the first conclusion; then Proposition \ref{epsilon-comp-symb-inv}
allows to obtain the second one.
\end{proof}

\begin{lemma} 
 For any $N\in\mathbb{N}^*\,,$  there exist two bounded families
$ \left\{\phi_N[r^{\epsilon,\kappa}_1]\right\}_{(\epsilon,\kappa)\in[0,\epsilon_0]\times [0,\kappa_0]}$ in
$S^{-2}_1(\Xi)^\bullet$ and 
$\left\{\psi_N[r^{\epsilon,\kappa}_1]\right\}_{(\epsilon,\kappa)\in[0,\epsilon_0]\times [0,\kappa_0]} $ in
$S^{-2N}_1(\Xi)^\bullet\,,$
such that
$$
r^{\epsilon,\kappa}(\epsilon a)\,=\,\phi_N[r^{\epsilon,\kappa}_1]+r^{\epsilon,\kappa}
(\epsilon a)\, \,\sharp^{\epsilon,\kappa}\,\psi_N[r^{\epsilon,\kappa}_1]=\phi_N[r^{\epsilon,
\kappa } _1]+\psi_N[r^{\epsilon,\kappa}_1]\, \sharp^{\epsilon,\kappa
}\, r^{\epsilon,\kappa}(\epsilon a)\,.
$$
\end{lemma}
\begin{proof}
From the resolvent equation we  deduce that, for any $N\in\mathbb{N}^*\,,$ 
\begin{align}\label{rez-eq-1} 
r^{\epsilon,\kappa}(\epsilon a)\,& =\,\underset{1\leq n\leq N}{\sum}(1+\epsilon
a)^{n-1}\big[r^{\epsilon,\kappa}_1\big]^{\,\sharp^{\epsilon,\kappa}\,n}\,+\,
(1+\epsilon a)^Nr^{\epsilon,\kappa}
(\epsilon a)\, \,\sharp^{\epsilon,\kappa}\, \, \big[r^{\epsilon,\kappa}_1\big]^
{\,\sharp^{\epsilon,\kappa}\,N}\, \nonumber\\
 & =\,\underset{1\leq n\leq N}{\sum}(1+\epsilon
a)^{n-1}\big[r^{\epsilon,\kappa}_1\big]^{\,\sharp^{\epsilon,\kappa}\,n}\,+\,
(1+\epsilon a)^N\,  \big[r^{\epsilon,\kappa}_1\big]^
{\,\sharp^{\epsilon,\kappa}\,N}\,\sharp^{\epsilon,\kappa}\,\,r^{\epsilon,\kappa}(\epsilon a)\,.
\end{align}
\end{proof}

\begin{proposition}\label{Cor-4.1}
For any $N\in\mathbb{N}^*\,$,  there exist two bounded families
$ \left\{\phi_N[r^{\epsilon,\kappa}_1]\right\}_{(\epsilon,\kappa)\in[0,\epsilon_0]\times [0,\kappa_0]}$ in
$S^{-2}_1(\Xi)^\bullet$ and 
$\left\{\psi_N[r^{\epsilon,\kappa}_1]\right\}_{(\epsilon,\kappa)\in[0,\epsilon_0]\times [0,\kappa_0]} $ in
$S^{-2N}_1(\Xi)^\bullet,$  such that
$$
g_{1/\delta}\, \,\sharp^{\epsilon,\kappa}\,\, r^{\epsilon,\kappa}(\epsilon a)=
g_{1/\delta}\, \,\sharp^{\epsilon,\kappa}\,\, \phi_N[r^{\epsilon,\kappa}_1]
+g_{1/\delta}\, \,\sharp^{\epsilon,\kappa}\,r^{\epsilon,\kappa}
(\epsilon a)\, \,\sharp^{\epsilon,\kappa}\,\, \psi_N[r^{\epsilon,\kappa}_1]\, .
$$
\end{proposition}
\begin{proof}
We use Lemma \ref{2sept-1}, noticing that the constants appearing in
these formulas are all uniformly bounded for $(\epsilon,\kappa)\in[0,\epsilon_0]\times[0,\kappa_0]$ and  Proposition
\ref{epsilon-comp-symb}.
\end{proof}

For the first term in the formula given by Proposition \ref{Cor-4.1}  we can use
Proposition \ref{P-trunc-mag-comp} for any $\delta\in(0,\delta_0]$ but we have
to control the behavior when $\delta\searrow 0\,$.

\begin{proposition}\label{nou-2}
If $f^\bullet\in { S^{-m}_1(\Xi)}^\bullet$ for some $m>0$ there exists a
bounded family $\left\{F^\epsilon\right\}_{\epsilon\in[0,\epsilon_0]}\subset
S^{-\infty}(\Xi)$ such that  for any $\epsilon\in[0,\epsilon_0]$ and
$\delta\in(0,\delta_0]$ satisfying  \eqref{ScalHyp} we have the relation
$$
g_{1/\delta}\, \,\sharp^{\epsilon,\kappa}\,f^\epsilon\ =\
F^\epsilon_{(\epsilon,\delta^{-1})}\,,
$$
and if we use the notation from Remark \ref{rem-sl-var-class} (i.e.
$f^\epsilon=\widetilde{f}^\epsilon_{(\epsilon,1)}$) the map
$$
{ S^{-m}_1(\Xi)}\ni\widetilde{f}^\epsilon\mapsto F^\epsilon\in S^{-\infty}(\Xi)
$$
is continuous.
\end{proposition}

\begin{proof}
For any $\varphi\in S^{-m}_1(\Xi)$ with $m>0$ let us proceed as in the proof of
Proposition \ref{P-trunc-mag-comp} after a change of variables
$(\eta,z)\mapsto(\delta^{-1}\eta,\delta z)$:
\begin{equation*}
\begin{array}{l}
\hspace{-1cm}\big(g_{1/\delta}\,\sharp^{\epsilon,\kappa}\,\varphi_{(\epsilon,1)}\big)(x,\xi)
\\ \qquad 
=\,\pi^{-4}\int\limits_{\Xi\times\Xi}e^{-2i\sigma(Y,Z)}\omega^{B_{\epsilon,
\kappa}}(x,y,z)g\big(\delta^{-1}\xi-\delta^{-1}\eta\big)\varphi\big(\epsilon
x-\epsilon z,\xi-\zeta\big)\, dYdZ\,\\ \qquad 
=\,\pi^{-4}\int\limits_{\Xi\times\Xi}e^{-2i\sigma(Y,Z)}\omega^{B_{\epsilon,
\kappa}}(x,y,\delta^{-1}z)g\big(\delta^{-1}\xi-\eta\big)\varphi\big(\epsilon
x-\epsilon\delta^{-1} z,\xi-\zeta\big)\,dYdZ\,\\ \qquad 
=\,\pi^{-4}\int\limits_{\Xi\times\Xi}e^{-2i\sigma(Y,Z)}
e^{i\, \epsilon\Phi_\kappa(\epsilon
x,y,\delta^{-1}z)}g\big(\delta^{-1}\xi-\eta\big)\varphi\big(\epsilon
x-\epsilon\delta^{-1} z,\xi-\zeta\big)\,\times
\\
\qquad \qquad\qquad \qquad
\times\left(1\,+\,i\, \kappa\epsilon^2\Psi^{\epsilon}(\epsilon
x,y,\delta^{-1}z)\int_0^1\Theta^{\epsilon,\kappa}_\tau(\epsilon
x,y,\delta^{-1}z)d\tau\right) dY dZ \,.
\end{array}
\end{equation*}
Let us note that due to \eqref{ScalHyp} 
$ \tau:=\epsilon\delta^{-1}=\delta^{\mu-1}$  is bounded for
$\delta\in(0,\delta_0]$ and  we can write 
\begin{equation*}
\hspace{-1cm} \epsilon\Phi_\kappa(x,y,\delta^{-1}z)=\epsilon\delta^{-1}\underset{j,k}{\sum
}y_jz_k
\left(B_0+\kappa B(x)\right)=\tau\Phi_\kappa(x,y,z)\,,
\end{equation*}
and 
\begin{align*}
&\hspace{-1cm}\epsilon^2\Psi^{\epsilon}(x,y,\delta^{-1}z)
\\&=
-8\epsilon^2\delta^{-2}\underset{j,k}{\sum}y_jz_k
\underset{
\ell = 1, 2}{\sum}\left[\delta y_\ell R_1\big(\partial_\ell B_{jk}\big)(x,\epsilon
y,\epsilon\delta^{-1} z) + 
z_\ell R_2\big(\partial_\ell B_{jk}\big)(x,\epsilon y,\epsilon\delta^{-1} z)\right]
\\ &
= - 8 \,\tau^2\underset{j,k}{\sum}y_jz_k
\underset{\ell =1,2}{\sum}\left[\delta y_\ell R_1\big(\partial_\ell B_{jk}\big)(x,\delta^\mu
y,\tau z)\,+\,
z_\ell R_2\big(\partial_\ell B_{jk}\big)(x,\delta^\mu y,\tau z)\right]\, .
\end{align*}

These equalities imply that  
$$(x,y,z) \mapsto \epsilon\Phi_\kappa(x,y,\delta^{-1}z),
\quad (x,y,z) \mapsto \kappa\epsilon^2\Psi^{\epsilon}(x,y,\delta^{-1}z)\; {\rm and
}\; (x,y,z) \mapsto \Theta^{\epsilon,\kappa}_t(\epsilon x,y,\delta^{-1}z)$$ define three families of
functions on $\X^3$  indexed by
$(\kappa,\delta,\tau,t)\in[0,\kappa_0]\times(0,\delta_0]\times  [0,\epsilon_0^{1-1/
\mu}] \times[0,1]$ that are bounded in $BC^\infty\big(\X;C^\infty_{\text{\sf
pol}}(\X\times\X)\big)$. In conclusion we can apply Proposition
\ref{comp-epsilon-symb} with $\mu_1=\mu_2=\mu_4=1$ and $\mu_3=\tau$.
\end{proof}We can now rephrase Proposition \ref{Cor-4.1} as
\begin{proposition}\label{Cor-4.2}
For any $N\in\mathbb{N}^*$, under Hypothesis  \eqref{ScalHyp},  there exist two bounded families
$\left\{\mathfrak{g}^{\epsilon,\kappa}_N\right\}_{(\epsilon,\kappa)\in[0,\epsilon_0]\times[0,\kappa_0]}$ in
$S^{-\infty}(\Xi)$ and
$ \left\{\psi_N[r^{\epsilon,\kappa}_1]\right\}_{(\epsilon,\kappa)\in[0,\epsilon_0] \times [0,\kappa_0]}$ in
$S^{-2N}_1(\Xi)^\bullet$\,,  such that
\begin{equation}\label{zza}
\begin{array}{ll}
g_{1/\delta}\, \,\sharp^{\epsilon,\kappa}\,\, r^{\epsilon,\kappa}(\epsilon a)\ &=\ 
\big(\mathfrak{g}^{\epsilon,\kappa}_N\big)_{(\epsilon,\delta^{-1})}\,
+\,g_{1/\delta}\, \,\sharp^{\epsilon,\kappa}\,\, r^{\epsilon,\kappa}
(\epsilon a)\, \,\sharp^{\epsilon,\kappa}\,\, \psi_N[r^{\epsilon,\kappa}_1]\\
&=\
\big(\mathfrak{g}^{\epsilon,\kappa}_N\big)_{(\epsilon,\delta^{-1})}\,+\,\psi_N[
r^{\epsilon,\kappa}_1]\, \,\sharp^{\epsilon,\kappa}\,\, r^{\epsilon,\kappa}
(\epsilon a)\, \,\sharp^{\epsilon,\kappa}\,g_{1/\delta}\,.
\end{array}
\end{equation}
\end{proposition}

\subsubsection{The first term of the quasi-inverse in \eqref{q-rez}.}
Using Proposition \ref{Cor-4.2} we conclude that in order to study the convergence
of the series
$$
\underset{\gamma^*\in\Gamma_*}{\sum}
\tau_{\gamma^*}
\big(g_{1/\delta}\, \,\sharp^{\epsilon,\kappa}\,\, 
r^{\epsilon,\kappa}(\epsilon a)\big)\,,
$$
we may according to \eqref{zza} separately study the convergence of the following two series
\begin{align}
&\underset{\gamma^*\in\Gamma_*}{\sum}
\tau_{\gamma^*}
\big(\big(\mathfrak{g}^{\epsilon,\kappa}_N\big)_{(\epsilon,\delta^{-1})}\big),
\label{25aug-2}\\
&\underset{\gamma^*\in\Gamma_*}{\sum}
\tau_{\gamma^*}
\big(g_{1/\delta}\, \,\sharp^{\epsilon,\kappa}\,\, r^{\epsilon,\kappa}
(\epsilon a)\, \,\sharp^{\epsilon,\kappa}\,\, \psi_N[r^{\epsilon,\kappa}_1]\big)\,,\label{25aug-3}
\end{align}
that are both of the form discussed in Lemma \ref{periodic-sums} and Proposition \ref{P.sum}, but
we stil have to control the behavior when $\delta\searrow 0\,$.

\paragraph{The series \eqref{25aug-2}.} Let us consider for some $
\varphi\in S^{-\infty}(\Xi)$, the series
$$
\Phi^{\epsilon,\delta}\ :=\
\underset{\gamma^*\in\Gamma_*}{\sum}\tau_{\gamma^*}\big(\varphi_{(\epsilon,\delta^{
-1})}\big)\,,
$$
and the operator
$\mathfrak{Op}^{\epsilon,\kappa}\big(\Phi^{\epsilon,\delta}\big)$ 
that are well defined as shown in  Lemma \ref{periodic-sums} and Proposition
\ref{P.sum}.

\begin{proposition}\label{P-25aug-2}
There exist positive constants $C$, $\epsilon_0$ and $\delta_0$ such that
$$ \left\|\mathfrak{Op}^{\epsilon,\kappa}\big(\Phi^{\epsilon,\delta}\big)\right\|_
{{ \mathcal L}(\mathcal{H})}\leq C\,,
$$ 
for any
$(\epsilon,\delta,\kappa)\in[0,\epsilon_0]\times(0,\delta_0]\times[0,1]$
verifying \eqref{ScalHyp}. Moreover the application
$$
S^{-\infty}(\Xi)\ni\varphi\,\mapsto\,\Phi^{\epsilon,\delta}\in\big({
S^0_0(\Xi)},\|\cdot\|_{B_{\epsilon,\kappa}}\big)
$$
is continuous uniformly for $(\epsilon,\delta,\kappa)\in[0,\epsilon_0]\times[0,\delta_0]\times[0,\kappa_0]$ verifying  \eqref{ScalHyp}.
\end{proposition}
\begin{proof}
All we have to do is to control the behavior of the norm on the right hand side
of \eqref{CS-est} when $\delta\searrow 0\,$. We note that $ \epsilon\delta^{-2}=\delta^{\mu-2}\leq\delta_0^{\mu-2}$ with $ \mu-2>0$ and due
to   \eqref{ScalHyp} we can use Proposition \ref{P-trunc-mag-comp} with
$\tau=\delta^{-1}$, in order to obtain that, for any
$\alpha^*\in\Gamma_*$\,, 
\begin{equation*} 
\begin{array}{l}
\left\|\overline{\widetilde{\varphi}^\epsilon}_{(\epsilon,\delta^{-1})}
\, \,\sharp^{\epsilon,\kappa}\,\, \big(\tau_{\alpha^*}\widetilde{\varphi}^\epsilon_{
(\epsilon,\delta^{-1})} \big)\right\|_{B_{\epsilon,\kappa}} \\ \quad  
 \leq\left\|\big(\overline{\widetilde{\varphi}^\epsilon}\, \natural^{
\epsilon}\, 
\big(\tau_{ \alpha^* )}\widetilde{\varphi}^\epsilon\big)_{(\epsilon,\delta^{-1})}
\big)\right\|_{B_{\epsilon,\kappa}}\\
 \qquad +\,(\epsilon\delta^{-1}/2)\underset{ \ell\leq 2}{\sum}\left(
\left\|\big(\partial_{x_\ell }\overline{\widetilde{\varphi}^\epsilon}\, \natural^{
\epsilon}\, 
\big(\tau_{ \alpha^* }\partial_{\xi_\ell }\widetilde{\varphi}^\epsilon\big)\big)_{
(\epsilon,\delta^{-1})}
\right\|_{B_{\epsilon,\kappa}}   +\left\|\big(\partial_{\xi_\ell}\overline{\widetilde{
\varphi}^\epsilon}\, \natural^{\epsilon}\, 
\big(\tau_{ \alpha^* }\partial_{x_\ell}\widetilde{\varphi}^\epsilon\big)\big)_{
(\epsilon,\delta^{-1})}
\right\|_{B_{\epsilon,\kappa}}\right) \\
 \qquad +\,(\epsilon\delta^{-1})^2 \left\|\mathscr{R}_{\epsilon,\kappa,
\delta^{-1}}\big(\overline{\widetilde{\varphi}^\epsilon},
\big(\tau_{ \alpha^* }\widetilde{\varphi}^\epsilon\big)_{(\epsilon,\delta^{-1})}
\right\|_{B_{\epsilon,\kappa}}\,.
\end{array}
\end{equation*}
We use Proposition \ref{scaled-norm} and obtain that, for any $p>2$ there exists $C_p>0$ such that:
\begin{equation*}
\begin{array}{l}
\hspace{-1cm}<\alpha^*>^N\left\|\overline{\widetilde{\varphi}^\epsilon}_{(\epsilon,\delta^{-1
})}
\,\sharp^{\epsilon,\kappa}\,\big(\tau_{ \alpha^* }\widetilde{\varphi}^\epsilon_{
(\epsilon,\delta^{-1})}
\big)\right\|_{B_{\epsilon,\kappa}} \\
\quad \leq C_p\, 
<\alpha^*>^N\left[\nu^{-p,0}_{(0,p)}
\left(\overline{\widetilde{\varphi}
^\epsilon}\, \natural^{\epsilon}
\big(\tau_{ \alpha^* }\widetilde{\varphi}^\epsilon\big)
\right)\,\right. \\ 
 \qquad +\,(\epsilon\delta^{-1}/2)\underset{\ell=1,2}{\sum}
 \left. \left(
\nu^{-p,0}_{(0,p)}\left(
\partial_{x_\ell }\overline{\widetilde{\varphi}^\epsilon}\, \natural^{\epsilon}\, 
\big(\tau_{ \alpha^* }\partial_{\xi_\ell }\widetilde{\varphi}^\epsilon\big)
\right) \right. +\nu^{-p,0}_{(0,p)}\left(
\partial_{\xi_\ell }\overline{\widetilde{\varphi}^\epsilon}\, \natural^{\epsilon}\,
\big(\tau_{ \alpha^* }\partial_{x_\ell }\widetilde{\varphi}^\epsilon\big)
\right)\right) \\ 
\qquad +\left.\,(\epsilon\delta^{-1})^2\nu^{-p,0}_{(0,p)}\left(\mathscr{R}_{\epsilon,
\kappa,\delta^{-1}}\big(\overline{\widetilde{\varphi}^\epsilon},
\big(\tau_{ \alpha^* }\widetilde{\varphi}^\epsilon\big)
\right)\right]\,.
\end{array}
\end{equation*}

The four terms appearing above can
now be dealt with by similar arguments to those in the proof of Proposition
\ref{P.sum} by using Proposition~\ref{comp-epsilon-symb} and choosing $N>2$ in order to control the convergence of the series indexed by
$\alpha^*\in\Gamma_*$ appearing in the Cotlar-Stein criterion.
\end{proof}

\paragraph{The series \eqref{25aug-3}.} Let us consider now some $f^\bullet\in {
S^{-m}_1(\Xi)}^\bullet$ for some $m>0\,$, the associated family of
symbols $\{\widetilde{f}^\epsilon\}_{\epsilon\in[0,\epsilon_0]}$ given by Remark~\ref{rem-sl-var-class} and the series
$$
\Psi^{\epsilon,\kappa,\delta}:=\underset{\gamma^*\in\Gamma_*}{\sum}\tau_{
\gamma^*}\big(g_{\delta^{-1}}\, \,\sharp^{\epsilon,\kappa}\,\, r^{\epsilon,\kappa}
(\epsilon a)\, \,\sharp^{\epsilon,\kappa}\,\, f^\epsilon\big)\,,
$$
with its associated operator
$\mathfrak{Op}^{\epsilon,\kappa}\big(\Psi^{\epsilon,\kappa,\delta}\big)$
obtained by using Lemma \ref{periodic-sums} and Proposition~\ref{P.sum} after
noticing that  the magnetic composition property (Theorem 2.2 in
\cite{IMP1}) gives:
$$
g_{\delta^{-1}}\,\sharp^{\epsilon,\kappa}\,r^{\epsilon,\kappa}(\epsilon a)\sharp^{\epsilon,
\kappa}f^\epsilon\ \in\ S^{-\infty}(\Xi)\,,
$$
for any
$(\epsilon,\kappa,\delta)\in[0,\epsilon_0]\times[0,\kappa_0]\times(0,\delta_0]\,$.

\begin{proposition}\label{P-25aug-3}
There exist positive constants $C$, $\epsilon_0$, $\delta_0$  and $\kappa_0$ such that
$$
\left\|\mathfrak{Op}^{\epsilon,\kappa}\big(\Psi^{\epsilon,\kappa,\delta}
\big)\right\|_{{ \mathcal L}(\mathcal{H})}\leq C\epsilon^{-1},
$$ 
for any
$(\epsilon,\delta,\kappa)\in[0,\epsilon_0]\times(0,\delta_0]\times[0,\kappa_0]$
verifying \eqref{ScalHyp}. Moreover the application
$$
{ S^{-m}_1(\Xi)}\ni f^\epsilon\,\mapsto\,\epsilon\Psi^{\epsilon,\kappa,\delta}
\in\big({ S^0_0(\Xi)} , \|\cdot\|_{B_{\epsilon,\kappa}}\big)
$$
is continuous uniformly for $(\epsilon,\delta,\kappa)\in[0,\epsilon_0]\times[0,\delta_0]\times[0,\kappa_0]$ verifying   \eqref{ScalHyp}.
\end{proposition}
\begin{proof}The main ingredient is to note that
\begin{align*}
\big(\overline{g_{\delta^{-1}}\,\sharp^{\epsilon,\kappa}\,r^{\epsilon,\kappa}
(\epsilon a)\,\sharp^{\epsilon,\kappa}\,f^\epsilon}\big)&\,\sharp^{\epsilon,\kappa}\,\tau_{
\alpha^*}\big(g_{\delta^{-1}}\,\sharp^{\epsilon,\kappa}\,r^{\epsilon,\kappa}
(\epsilon a)\,\sharp^{\epsilon,\kappa}\,f^\epsilon\big)\\
&=\ \big(\overline{f^\epsilon}
\,\sharp^{\epsilon,\kappa}\,r^{\epsilon,\kappa}(\epsilon a)\big)\, \,\sharp^{\epsilon,\kappa}\,\,
\big(g_{\delta^{-1}}\,\,\sharp^{\epsilon,\kappa}\,\,\tau_{\alpha^*}(g_{\delta^{-1}}
)\big)
\,\,\sharp^{\epsilon,\kappa}\,\,\tau_{\alpha^*}\big(r^{\epsilon,\kappa}(\epsilon a)\sharp^{
\epsilon,\kappa}f^\epsilon\big),
\end{align*}
and
\begin{align*}
&\big(g_{\delta^{-1}}\,\,\sharp^{\epsilon,\kappa}\,\,r^{\epsilon,\kappa}(\epsilon a)\sharp^{
\epsilon,\kappa}f^\epsilon\big)\,\,\sharp^{\epsilon,\kappa}\,\,\tau_{\alpha^*}
\big(\overline{g_{\delta^{-1}}\,\,\sharp^{\epsilon,\kappa}\,\,r^{\epsilon,\kappa}
(\epsilon a)\,\,\sharp^{\epsilon,\kappa}\,\,f^\epsilon}\big)\\
&=\ \big(g_{\delta^{-1}}
\,\,\sharp^{\epsilon,\kappa}\,\,r^{\epsilon,\kappa}(\epsilon a)\big)\,\,\sharp^{\epsilon,\kappa}\,\,
\big(\widetilde{f^\epsilon}_{(\epsilon,1)}\,\,\sharp^{\epsilon,\kappa}\,\,\tau_{\alpha^*
}(\overline{\widetilde{f^\epsilon}}_{(\epsilon,1)})\big)
\,\,\sharp^{\epsilon,\kappa}\,\,\tau_{\alpha^*}\big(r^{\epsilon,\kappa}(\epsilon a)\sharp^{
\epsilon,\kappa}g_{\delta^{-1}}\big)\,.
\end{align*}
Then we fix some $m>N+2>4$ and we repeat the arguments in the proof
of Proposition~\ref{P-25aug-2}\,. 
\end{proof}

\subsubsection{Locating the spectral islands}\label{SSS-nou}
 In this subsection we fix $\mu=3$ and assume moreover  \eqref{marts7}, in order to prove the second claim of Proposition \ref{mainP} concerning the location and size of the spectral islands. 

At the operator level, \eqref{marts7}
can be rewritten as
\beq\label{marts3}
\Big(\mathfrak{Op}^{\epsilon,\kappa}(\lambda^\epsilon)-\epsilon a\Big)\,
\mathfrak{Op}^{\epsilon,\kappa}(\widetilde{r}_{
\lambda}
(\epsilon a)) =\ \bb1\,+\mathfrak{Op}^{\epsilon,\kappa}(\mathfrak
{r}_{\delta,a})\,,\quad\|\mathfrak{Op}^{\epsilon,\kappa}(\mathfrak
{r}_{\delta,a})\|\,\leq\,C\, \epsilon ^{1/3}\,.
\eeq 
Let us choose $a\in\mathbb{C}$ on the positively oriented circle
$|a-(2n+1)mB_0|=mB_0$
centered at some Landau level $(2n+1)mB_0$ such that the distance between $a$
and all the Landau levels is bounded from below by $mB_0$
and take $\z=\epsilon a$.
The estimate \eqref{marts3} implies that 
$$
\left
\|\Big(\mathfrak{Op}^{\epsilon,\kappa}(\lambda^\epsilon)-\z\Big)^{-1}\right
\|\leq C\, \epsilon^{-1}\,.
$$
By iterating we get:
\begin{align}\label{marts8}
\Big(\mathfrak{Op}^{\epsilon,\kappa}(\lambda^\epsilon)-\z\Big)^{-1}=\mathfrak{Op
}^{\epsilon,\kappa}(\widetilde{r}_{
\lambda}
(\z))-\Big(\mathfrak{Op}^{\epsilon,\kappa}(\lambda^\epsilon)-\z\Big)^{-1}
\mathfrak{Op}^{\epsilon,\kappa}(\mathfrak
{r}_{\delta,a})\,.
\end{align}
In order to identify the band operator corresponding to the spectrum near
$(2n+1)m\epsilon B_0$ we compute the Riesz integral: 
$$
h_n:=\frac{i}{2\pi} \int_{|\z-(2n+1)m\epsilon B_0|=m\epsilon B_0}\big
(\z-(2n+1)m\epsilon B_0\big
)\Big(\mathfrak{Op}^{\epsilon,\kappa}(\lambda^\epsilon)-\z\Big)^{-1}\,d\z\,.
$$
We want to show that 
$$
\|h_n\|\leq C (\epsilon \kappa +\epsilon^{4/3})\,.
$$
Inserting \eqref{marts8} in the Riesz integral we obtain two contributions. The
second contribution containing the term
$\mathfrak{Op}^{\epsilon,\kappa}(\mathfrak
{r}_{\delta,a})$ is easy, and leads to something of order $\epsilon^{4/3}$. We
still need to estimate the term
$$
\frac{i}{2\pi} \int_{|\z-(2n+1)m\epsilon B_0|=m\epsilon B_0}\big
(\z-(2n+1)m\epsilon B_0\big )\mathfrak{Op}^{\epsilon,\kappa}(\widetilde{r}_{
\lambda}
(\z)) \,d\z\, .
$$
 Since $m_\epsilon-m=\mathcal{O}(\epsilon)$, replacing $m$
with $m_\epsilon$
in the expression of $h_n$ produces an error of order $\epsilon^2$, hence
the
relevant object becomes:
$$
h_n':=\frac{i}{2\pi} \int_{|\z-(2n+1)m_\epsilon \epsilon
B_0|=m_\epsilon\epsilon B_0}\big (\z-(2n+1)m_\epsilon\epsilon B_0\big
)\mathfrak{Op}^{\epsilon,\kappa}(\widetilde{r}_{
\lambda}
(\z))\, d\z\,.
$$

From the expression of $\widetilde{r}_{
\lambda}
(\z)$ in \eqref{q-rez} we see that the second term is analytic inside the circle
of integration, thus only 
$$\underset{\gamma^*\in\Gamma_*}{\sum}
\tau_{\gamma^*}
\big(g_{1/\delta}\,\,\sharp^{\epsilon,\kappa}\,\,
r^{\epsilon,\kappa}(\z)\big)$$
will contribute to $h_n'$. Let us introduce the shorthand notation
$
E^\epsilon_n:=(2n+1)m\epsilon B_0 $ 
and notice that $h_n'=\mathfrak{Op}^{\epsilon,\kappa}(\mathfrak{h}_n)$ with
$
\mathfrak{h}
_n:=\underset{\gamma^*\in\Gamma_*}{\sum}
\tau_{\gamma^*}
\left(g_{1/\delta}\,\sharp^{\epsilon,\kappa}\,\mathfrak{g}_n\right) \,,$
 where 
\beq
\mathfrak{g}_n:=\frac{i}{2\pi}
\int_{|\z-E^\epsilon_n|=m_\epsilon\epsilon B_0}\big
(\z-E^\epsilon_n\big
) \, r^{\epsilon,\kappa}(\z)\,d\z\,.
\eeq
This is a slowly varying symbol of class $S^{-2}_1(\Xi)^\bullet$ (see
Definition \ref{S-epsilon}) having an operator norm of order $\epsilon$.
Using the expansion in \eqref{rez-eq-1} we notice that the first sum of
$N$ terms is a polynomial in $\z$ and thus vanishes when integrated along the
circle $|\z-E^\epsilon_n|=m_\epsilon\epsilon B_0$ and we get:
\begin{align*}
\mathfrak{g}_n&=\frac{i}{2\pi}
\int_{|\z-E^\epsilon_n|=m_\epsilon\epsilon B_0}\big
(\z-E^\epsilon_n\big
)(1+\z)^Nr^{\epsilon,\kappa}
(\z)\, \,\sharp^{\epsilon,\kappa}\, \,
\big[r^{\epsilon,\kappa}_1\big]^
{\,\sharp^{\epsilon,\kappa}\,N}\,d\z \,,
\end{align*}
and 
\begin{align*}
\mathfrak{h}'_n&=\underset{\gamma^*\in\Gamma_*}{\sum}
\tau_{\gamma^*}
\left(g_{1/\delta}\,\sharp^{\epsilon,\kappa}\,\widetilde{\mathfrak{g}}_n
\,\sharp^{\epsilon,\kappa}\, \,
\big[r^{\epsilon,\kappa}_1\big]^
{\,\sharp^{\epsilon,\kappa}\,N}\right)\,,
\end{align*}
where
$$
\widetilde{\mathfrak{g}}_n:=\frac{i}{2\pi}
\int_{|\z-E^\epsilon_n|=m_\epsilon\epsilon B_0}\big
(\z-E^\epsilon_n\big)(1+\z)^N
r^{\epsilon,\kappa}(\z)\,d\z\,.
$$
The arguments in the above subsection may be applied in order to conclude that 
$\mathfrak{h'}_n\in S^0_0(\Xi)$ and a-priori:
$$
\|\mathfrak{Op}^{\epsilon,\kappa}(\mathfrak{h}'_n)\|\leq C\epsilon\,,
$$ 
but this is not enough in order to conclude the existence of gaps. 

Let us also consider $r^\epsilon(\epsilon a)\in
S^{-2}_1(\Xi)$ such that
\begin{equation*}
\left(\mathfrak{Op}^\epsilon(h_{m^\epsilon})-\epsilon
a\right)^{-1}=\mathfrak{Op}^\epsilon\big(r^\epsilon(\epsilon a)\big)\,.
\end{equation*}
 For $a$ on the circle $|a-E^1_n|=m_\epsilon B_0$, the above inverse is well defined and its symbol also. Let us
notice that
\begin{equation*}
\frac{i}{2\pi}
\int_{|\z-E^\epsilon_n|\,=\,m_\epsilon\epsilon B_0}\big
(\z-E^\epsilon_n\big
)r^{\epsilon}(\z)\,d\z
 =0\,,
\end{equation*}
 and 
 \begin{equation*}
\frac{i}{2\pi}
\int_{|\z-E^\epsilon_n|\,=\,m_\epsilon\epsilon B_0}\big
(\z-E^\epsilon_n\big)(1+\z)^Nr^{\epsilon}(\z)\,d\z =0\,,
\end{equation*}
both integrands being analytic inside the circle
$|\z-E^\epsilon_n|=m_\epsilon\epsilon B_0\,$.
Hence  we can write:
$$
\widetilde{\mathfrak{g}}_n =\frac{i}{2\pi}
\int_{|\z-E^\epsilon_n|=m_\epsilon\epsilon B_0}\big
(\z-E^\epsilon_n\big)(1+\z)^N
\big(r^{\epsilon,\kappa}(\z)-r^\epsilon(\z)\big)\,d\z\,.
$$

Using now the estimate \eqref{nou-3},
Proposition \ref{nou-2} and Proposition \ref{P-25aug-2} we obtain:
$$
\|\mathfrak{Op}^{\epsilon,\kappa}(h_n')\|\,\leq\, C\, \kappa\epsilon^{-1}\epsilon^2\,=\,C\, \kappa\epsilon\,.
$$
which concludes the proof of the size of the spectral islands. From now on we concentrate on proving \eqref{marts7}. 

\subsubsection{Estimating the product $\big(\lambda^\epsilon-\epsilon a\big)\,\,\sharp^{\epsilon,\kappa}\,\,\widetilde{
r}_\lambda( \epsilon a)$.} 
 Here we no longer fix $\mu=3\,$. Instead, we again allow $\mu$ to vary in the interval $(2,4)$ and actually prove that this is the maximal interval for which the operator defined in \eqref{q-rez} is a quasi-resolvent. 

Given the periodic lattice $\Gamma_*\subset\X^*$ and the function
$g_{1/\delta}$ defined in \eqref{g-delta}, for $\delta\in(0,\delta_0]\,$,
there exists a function $\widetilde{\chi}\in C^\infty_0(\X^*)$ such that:
\begin{enumerate}
 \item $0\leq\widetilde{\chi}\leq1\,$,
 \item $\widetilde{\chi}\, g_{1/\delta}=g_{1/\delta}\,$, for any
$\delta\in(0,\delta_0]\,$,
\item $\underset{\gamma^*\in\Gamma_*}{\sum}\tau_{\gamma^*}\widetilde{\chi}=1\,$.
\end{enumerate}

For any $\gamma^*\in\Gamma_*$ let us introduce 
\begin{align}\label{marts6}
\widetilde{\lambda}^{\epsilon,\kappa}_{\gamma^*}\ :=\
\big(\lambda^\epsilon\,-\epsilon a\big)\, \sharp^{\epsilon,\kappa}\,
\tau_{\gamma^*}\widetilde{\chi}\; .
\end{align}
\begin{proposition}
The families
$\{\lek\}_{(\epsilon,\kappa) \in[0,\epsilon_0]\times [0,1]}$ are bounded in
$S^{-\infty}(\Xi)^\bullet$\,.
\end{proposition}
\begin{proof}
We notice that $\lambda^\epsilon-\epsilon a\in S^{-\infty}(\Xi)^\circ$ for
any $\epsilon\in[0,\epsilon_0]$ and $\widetilde{\chi}$ also, so that we can use Proposition \ref{epsilon-comp-symb}.
\end{proof}

\medskip

We have the following properties:
\begin{enumerate}
 \item $\lek=\tau_{\gamma^*}\widetilde{\lambda}^\epsilon_0\,$,
 \item
$\big(\lambda^\epsilon-\epsilon a\big)=
\hspace{-0.2cm} \underset{\gamma^*\in\Gamma_*}{
\sum}\lek\,$, \\ 
\hspace{-5cm} the convergence following from Lemma
\ref{periodic-sums} and Proposition \ref{P.sum}\,.
\end{enumerate}

We shall use the following decomposition:
\begin{align}\label{ter-1}
\big(\lambda^\epsilon-\epsilon a\big)\,\,\sharp^{\epsilon,\kappa}\,\,\widetilde{
r}_\lambda(\epsilon a) &=
\underset{\gamma^*\in\Gamma_*}{\sum}\lek\,\,\sharp^{\epsilon,\kappa}\,\,\underset{
\alpha^*\in\Gamma_*}{\sum}
\tau_{\alpha^*}
\big(g_{1/\delta}\,\,\sharp^{\epsilon,\kappa}\,\,
r^{\epsilon,\kappa}(\epsilon a)\big)\, \nonumber \\
&\qquad +\underset{\gamma^*\in\Gamma_*}{\sum}\lek\,\,\sharp^{\epsilon,\kappa}\,\,
\left(1-\widetilde{g}_{1/\delta}\right)\sharp^{
\epsilon,\kappa}r_{\delta,\epsilon,\kappa}(\epsilon a)\,.
\end{align}

\paragraph{The main contribution to the series
\eqref{ter-1}.}

In this paragraph we shall consider the term in  the second line of \eqref{ter-1} with
$\gamma^*=\alpha^*=0\,$:
\beq\label{ter-1-main}
\widetilde{\lambda}^\epsilon_0\,\,\sharp^{\epsilon,\kappa}\,\,
\big(g_{1/\delta}\,\,\sharp^{\epsilon,\kappa}\,\,
r^{\epsilon,\kappa}(\epsilon  a)\big)\ =\
\big(\lambda^\epsilon-\epsilon a\big)\,\,\sharp^{\epsilon,\kappa}\,\,
\big(\widetilde { \chi}\,\,\sharp^{\epsilon,\kappa}\,\,
g_{1/\delta}\big)\,\,\,\sharp^{\epsilon,\kappa}\,\,
r^{\epsilon,\kappa}(\epsilon  a)\,.
\eeq

\begin{lemma}\label{P-1-ter-1-main} 
Given $\epsilon_0>0$ and $\delta_0>0$ satisfying  \eqref{delta-0},  for any $(\epsilon,\delta,\kappa)\in(0,
\epsilon_0]\times(0,\delta_0]\times[0,1]$ satisfying   \eqref{ScalHyp}, the following relation holds
$$
\widetilde { \chi}\,\,\sharp^{\epsilon,\kappa}\,\,
g_{1/\delta}\ =\ \varphi^{\epsilon,\delta,\kappa}_{(\epsilon,\delta^{-1})}\,,
$$
where the family of symbols 
$$
\big\{\varphi^{\epsilon,\delta,\kappa}\,\mid\,(\epsilon,\delta,\kappa)\in(0,
\epsilon_0]\times(0,\delta_0]\times[0,1],\,\epsilon= \delta^{\mu} \big\}\
$$
is  bounded in $S^{-\infty}(\Xi)$.
\end{lemma}
\begin{proof}
We use the arguments in the proof of  Proposition \ref{P-trunc-mag-comp} 
(formula \eqref{T-2}) and Proposition~\ref{comp-epsilon-symb}. A change of variables in the definition of 
\beq \label{chi-g}
\begin{array}{ll}
\big(\widetilde { \chi}\,\,\sharp^{\epsilon,\kappa}\,\,
g_{1/\delta}\big)(x,\xi)&=
\int\limits_{\Xi\times\Xi}\,dYdZ\, \left[e^{-2i\sigma(Y,Z)}e^{ i\,
\epsilon\Phi_\kappa(\epsilon x,y,z)}
\widetilde{\chi}(\xi-\eta)g(\delta^{-1}
(\xi-\zeta))\times\right.     \\
&\qquad\qquad \qquad \qquad \times\left.\left(1\,+\,i\, \kappa\epsilon^2\Psi^{\epsilon}(\epsilon
x,y,z)\int_0^1\Theta^{\epsilon,\kappa}_\tau(\epsilon x,y,z)d\tau\right)\right]
\end{array}
\end{equation}
allows to write
$$
\big(\widetilde { \chi}\,\,\sharp^{\epsilon,\kappa}\,\,
g_{1/\delta}\big)(x,\xi)\ =\ \varphi^{\epsilon,\delta,\kappa}(\epsilon x,\delta^{-1}\xi)\,,
$$
where
\begin{align*}
\varphi^{\epsilon,\delta,\kappa}(x,\xi)&:=
\int\limits_{\Xi\times\Xi}\,dYdZ\, e^{-2i\sigma(Y,Z)}\widetilde{\chi}_\delta
(\xi-\eta)g
(\xi-\zeta) e^{i\, \epsilon\Phi_\kappa(x,\delta^{-1}y,\delta^{-1}z)}\times\\
&\qquad \qquad \times\
\left(1\,+\,i\,
\kappa\epsilon^2\Psi^{\epsilon}(
x,\delta^{-1}y,\delta^{-1}z)\int_0^1\Theta^{\epsilon,\kappa}_\tau(x,\delta^{-1}y
,\delta^{-1}z)d\tau\right).
\end{align*}
We notice first that $\{\widetilde{\chi}_\delta\}_{\delta\in(0,\delta_0]}$ is a
bounded subset of $S^{-\infty}(\Xi)^\circ$ and that the symbol
$\varphi^{\epsilon,\delta,\kappa}$ is defined by an oscillatory integral of the
form $\mathcal{L}^{1,1}_{(1,1,1,1)}\big(\widetilde{\chi}_\delta,g
\big)$ (see Proposition \ref{comp-epsilon-symb} for this notation) with the
integral kernel
\beq\label{def-L}
L(x,y,z):=e^{ i\,  \epsilon\Phi_\kappa(x,\delta^{-1}y,\delta^{-1}z)}\left(1\,+\,i\,
\kappa\epsilon^2\Psi^{\epsilon}(
x,\delta^{-1}y,\delta^{-1}z)\int_0^1\Theta^{\epsilon,\kappa}_\tau(x,\delta^{-1}y
,\delta^{-1}z)d\tau\right).
\eeq
We notice that
\begin{enumerate}
 \item
$\epsilon\Phi_\kappa(x,\delta^{-1}y,\delta^{-1}z)=\epsilon\delta^{
-2}\big(B_0+\kappa
B(x)\big)(y_2z_3-y_3z_2)=\epsilon\delta^{-2}\Phi_\kappa(x,y,z)\,,$\\ and by 
 \eqref{ScalHyp} $\epsilon\delta^{-2}=\delta^{\mu-2}\leq
\delta_0^{\mu-2}$  goes to 0 when $\delta\searrow0\,$.
\item $\kappa\epsilon^2\Psi^{\epsilon}(
x,\delta^{-1}y,\delta^{-1}z)=\kappa\epsilon^2\delta^{-3}\Psi^{1}(
x,\epsilon\delta^{-1}y,\epsilon\delta^{-1}z)\,,$\\
 and by  \eqref{ScalHyp} we have $\epsilon^2\delta^{-3}=\delta^{2\mu-3}\leq
\delta_0^{2\mu-3}$ and $\epsilon\delta^{-1}=\delta^{\mu-1}\leq
\delta_0^{\mu-1}$ and both go to 0 when $\delta\searrow0\,$.
\item $\Theta^{\epsilon,\kappa}_\tau(x,\delta^{-1}y
,\delta^{-1}z)=\exp\left\{i \tau\kappa\epsilon^2\Psi^{\epsilon}(
x,\delta^{-1}y,\delta^{-1}z)\right\}=\exp \left\{i \tau\kappa\epsilon^2\delta^{-3}
\Psi^{1}(
x,\epsilon\delta^{-1}y,\epsilon\delta^{-1}z)\right\}\,$.
\end{enumerate}
We conclude that for $(\tau,\kappa)\in[0,1]\times[0,1]$ and
$(\epsilon,\delta)\in[0,\epsilon_0]\times(0,\delta_0]$ verifying \eqref{ScalHyp}, the corresponding family of integral
kernels $L$ defined in \eqref{def-L} is bounded in
${BC^\infty\big(\X;C^\infty_{\text{\sf
pol}}(\X\times\X)\big)}$ and we can apply Corollary \ref{corA3}  in order to finish the proof.
\end{proof}

\begin{lemma}\label{P-2-ter-1-main} 
Given $\epsilon_0>0$ and $\delta_0>0$ satisfying  \eqref{delta-0}, there exists $C>0$ such that, for any $(\epsilon,\delta,\kappa)\in(0,
\epsilon_0]\times(0,\delta_0]\times[0,1]$ satisfying   \eqref{ScalHyp}, we have
$$
\widetilde{\chi}\,\,\sharp^{\epsilon,\kappa}\,\,g_{1/\delta}\,=\,g_{1/\delta}
\,+\,\mathfrak{x}^{
\epsilon,\delta,\kappa}_{(\epsilon,\delta^{-1})}\,,
$$
where the family
$$
\{\mathfrak{x}^{\epsilon,\delta,\kappa}\,\mid\,(\epsilon,\delta,\kappa)\in[0,
\epsilon_0]\times(0,\delta_0]\times[0,1]\,,\,\epsilon=\delta^{\mu}\}
$$ 
is a bounded subset in $S^{-\infty}(\Xi)$ and satisfies
$$
\|\mathfrak{x}^{
\epsilon,\delta,\kappa}_{(\epsilon,\delta^{-1})}\|_{B_{\epsilon,\kappa}}\leq
C\, \delta^{2(\mu-1)}\,.
$$
\end{lemma}
\begin{proof}

{For a suitable class of symbols $\phi$ and $\psi$ we can define an operation 
$\phi\,\natural^\epsilon\, \psi(x,\xi)$ (see in Appendix \ref{SS-sl-var-m-field} the formulas \eqref{def-Psi} and \eqref{trunc-mag-comp}) which resembles a lot the magnetic composition formula when the magnetic field is constant, with the notable difference that the value of the slowly varying magnetic field is taken at $x$. This composition turns out to be very useful when one works with slowly varying magnetic fields.   
}
Coming back to the formula \eqref{chi-g} in the proof of   Lemma \ref{P-1-ter-1-main}. 
 we observe that
$$ 
 \widetilde{\chi}\,\,\sharp^{\epsilon,\kappa}\,\,g_{1/\delta}=\widetilde{\chi}
\,\natural^\epsilon\,
g_{\delta^{-1}}\,+\,\psi^{\epsilon,\delta,\kappa}_{(\epsilon,\delta^{-1})}\,,
$$ 
with (see \eqref{def-Psi} for the definition of $\Psi^\epsilon$):
\begin{equation*}
\begin{array}{l}
\psi^{\epsilon,\delta,\kappa}(x,\xi)\\
\qquad :=\kappa\epsilon^2\int_0^1d\tau
\int\limits_{\Xi\times\Xi}\,dYdZ \, e^{-2i\sigma(Y,Z)}e^{i \, 
\epsilon\Phi_\kappa(x,y,z)}\Psi^{\epsilon}(
x,y,z)\Theta^{\epsilon,\kappa}_\tau(x,y,z)
\widetilde{\chi}(\delta\xi-\eta)g
(\xi-\delta^{-1}\zeta)\,,
\end{array}
\end{equation*}
being similar with the symbols $\{\varphi^{\epsilon,\delta,\kappa}\}$  in   Lemma \ref{P-1-ter-1-main}\,.\\
 By the same arguments as in the previous proof we
have:
 \begin{align*}
\psi^{\epsilon,\delta,\kappa}(x,\xi)&= \delta^{2(\mu-1)}\kappa(\epsilon\delta^{-\mu})^2\int_0^1d\tau
\int\limits_{\Xi\times\Xi}\,dYdZ \, e^{-2i\sigma(Y,Z)}
\widetilde{\chi}(\delta\xi-\eta)g
(\xi-\zeta) \times\\
&\hspace{3cm}\times\,e^{i\,\epsilon\delta^{-1}\Phi_\kappa(x,y,z)} i
\widetilde{\Psi}_{\delta}(
x,\epsilon\delta^{-1}y,\epsilon z) \exp i \left\{
\tau\kappa\epsilon^2\delta^{-2}
\widetilde{\Psi}_{\delta}(
x,\epsilon\delta^{-1}y,\epsilon z)\right\}\, \end{align*}
with
$$
\widetilde{\Psi}_{\delta}(x,y,z)\,:=\,8 \underset{1\leq
j,k\leq2}{\sum}y_jz_k\underset{1\leq
\ell \leq2}{\sum}\left[y_\ell R_1\big(\partial_\ell B_{jk}\big)(x,\epsilon
y,\epsilon z)\,+\,\delta
z_\ell R_2\big(\partial_\ell B_{jk}\big)(x,\epsilon y,\epsilon z)\right]\,.
$$
Moreover the family 
$$
\{ \widetilde{\psi}^{\epsilon,\delta,\kappa} := \delta^{2(1-\mu)} \psi^{\epsilon,\delta,\kappa}\,
\,\mid\,(\epsilon,\delta,
\kappa)\in[0,
\epsilon_0]\times(0,\delta_0]\times[0,1],\,\epsilon=\delta^{\mu} \}$$  is a
bounded subset in $S^{-\infty}(\Xi)\,$.\\

We now use Proposition \ref{trunc-comp-N} with $N=2$ in order to obtain that
\begin{align}\label{F-ter-1-main}
\big(\widetilde{\chi}\,\natural^\epsilon\, g_{\delta^{-1}}\big)(x,\xi) & =
\widetilde{\chi}(\xi)g_{\delta^{-1}}(\xi)\,+\,\epsilon\big(B_0+\kappa B(\epsilon
x)\big){\underset{j=1,2}{\sum}\big(\partial_{\xi_j}\widetilde{\chi}\big)(\xi)\,
\big(\partial_{\xi^\bot_j}g_{\delta^{-1}}\big)(\xi)}\nonumber \\
&\qquad +\,\epsilon^2\big(B_0+\kappa B(\epsilon
x)\big)^2\underset{|\alpha|=2}{\sum}\mathscr{M}^{\epsilon,\kappa}
_2\big(\partial_{\xi}^\alpha\widetilde{\chi},\partial_{\xi^\bot}^\alpha
g_{\delta^{-1}}\big)(x,\xi)\nonumber \\ 
&=\,g_{\delta^{-1}}(\xi)\,+\,\epsilon^2\delta^{-2}\big(B_0+\kappa B(\epsilon
x)\big)^2\underset{|\alpha|=2}{\sum}\mathscr{M}^{\epsilon,\kappa}
_2\big(\partial_{\xi}^\alpha\widetilde{\chi},(\partial_{\xi^\bot}^\alpha
g)_{\delta^{-1}}\big)(x,\xi)\,,
\end{align}
because $g_{\delta^{-1}}=0$ on the support of $\nabla_\xi\widetilde{\chi}\,$.\\
Taking into account formula \eqref{def-M} in the proof of Proposition
\ref{trunc-comp-N} we see that
\beq\label{4-1}
\epsilon^2\delta^{-2}\big(B_0+\kappa
B_\epsilon\big)^2\underset{|\alpha|=2}{\sum}\mathscr{M}^{\epsilon,
\kappa}
_2\big(\partial_{\xi}^\alpha\widetilde{\chi},(\partial_{\xi^\bot}^\alpha
g)_{\delta^{-1}}\big)\ =\ \delta^{2(\mu-1)}
\theta^{\epsilon,\delta,\kappa}_{(\epsilon,\delta^{-1})}\,,
\eeq
for some family
$\{\theta^{\epsilon,\delta,\kappa}\,\mid\,(\epsilon,\delta\,,
\kappa)\in[0,
\epsilon_0]\times(0,\delta_0]\times[0,1],\,\epsilon=\delta^{\mu}\}$ that
is bounded in $S^{-\infty}(\Xi)\,$.

We define
$\mathfrak{x}^{\epsilon,\delta,\kappa}:=\delta^{2(\mu-1)}\big(\widetilde{\psi}^{
\epsilon,\delta,\kappa}+\theta^{\epsilon,\delta,\kappa}\big)$ 
and the above results imply the conclusion of the
lemma by using Proposition \ref{scaled-norm} and the continuity
criterion for the magnetic pseudodifferential calculus (Theorem 3.1 in
\cite{IMP1}).
\end{proof}

Thus we can write our {\it main contribution term} in \eqref{ter-1-main} as:
\beq \label{ter-1-main-0}
\begin{array}{l}
\hspace{-1cm} \big(\lambda^\epsilon-\epsilon a\big)\,\,\sharp^{\epsilon,\kappa}\,\,
\big(\widetilde { \chi}\,\,\sharp^{\epsilon,\kappa}\,\,
g_{1/\delta}\big)\,\,\,\sharp^{\epsilon,\kappa}\,\,
r^{\epsilon,\kappa}(\epsilon  a)\\ \quad  =
\big(\lambda^\epsilon-\epsilon a\big)\,\,\sharp^{\epsilon,\kappa}\,\,
g_{1/\delta}\,\,\,\sharp^{\epsilon,\kappa}\,\,
r^{\epsilon,\kappa}(\epsilon  a)\,+\,\big(\lambda^\epsilon-\epsilon a\big)\,\sharp^{
\epsilon,\kappa}\,
\mathfrak{x}^{\epsilon,\delta,\kappa}_{(\epsilon,\delta^{-1})}\,\sharp^{\epsilon
,\kappa}r^{\epsilon,\kappa}(\epsilon  a)
\end{array}
\eeq
and notice that
\begin{align}\label{ter-1-main-rest}
\left\|\big(\lambda^\epsilon-\epsilon a\big)\,\sharp^{
\epsilon,\kappa}\,
\mathfrak{x}^{\epsilon,\delta,\kappa}_{(\epsilon,\delta^{-1})}\,\sharp^{\epsilon
,\kappa}r^{\epsilon,\kappa}(\epsilon  a)\right\|_{B_{\epsilon,\kappa}} & \leq C\delta^{
2(\mu-1)}\epsilon^{-1}\left\|\lambda^\epsilon-\epsilon a\right\|_{B_{
\epsilon,\kappa}}\left\|\epsilon r^{\epsilon,\kappa}(\epsilon  a)\right\|_{B_{
\epsilon,\kappa}}\nonumber \\ & \leq C \delta^{\mu-2}.
\end{align}

\begin{proposition}\label{P-main-term}
Given $\epsilon_0>0$ and $\delta_0>0$ satisfying \eqref{delta-0}, there exist $C>0$ such that for any $(\epsilon,\delta,\kappa)\in(0,
\epsilon_0]\times(0,\delta_0]\times[0,1]$ satisfying \eqref{ScalHyp} we have: 
$$
\big(\lambda^\epsilon-\epsilon a\big)\,\,\sharp^{\epsilon,\kappa}\,\,
g_{1/\delta}\,\,\,\sharp^{\epsilon,\kappa}\,\,r^{\epsilon,\kappa}(a)\
=\ g_{1/\delta}\,+\,\delta^{\mu-2}\mathfrak{z}^{\epsilon,\delta,\kappa}\,,
$$
with $\mathfrak{z}^{\epsilon,\delta,
\kappa}\in S^{-\infty}(\Xi)$ and $\|\mathfrak{z}^{\epsilon,\delta,
\kappa}\|_{B_{\epsilon,\kappa}}\leq C\,.$ 
\end{proposition}
\begin{proof}
We repeat the arguments in the proof of Lemma \ref{P-2-ter-1-main}
and write that
\beq\label{4.1}
\big(\lambda^\epsilon-\epsilon a\big)\,\,\sharp^{\epsilon,\kappa}\,\,
g_{1/\delta}\ =\
\big(\lambda^\epsilon-\epsilon a\big)\,\natural^{\epsilon}\,
g_{1/\delta}\,+\,\delta^{2(\mu-1)}\psi^{\epsilon,\delta,\kappa}_{(\epsilon,
\delta^{-1})}\,,
\eeq
where the subset
$
\{\psi^{\epsilon,\delta,\kappa}\,\mid\,(\epsilon,\delta,
\kappa)\in[0,
\epsilon_0]\times(0,\delta_0]\times[0,1],\,\epsilon=\delta^{\mu}\}
\subset S^{-\infty}(\Xi)
$
is bounded in $S^{-\infty}(\Xi)$.

For the first term in \eqref{4.1} we use Proposition \ref{trunc-comp-N} with
$N=2$ and write that
\begin{align}\label{4.2}
\big[\big(\lambda^\epsilon-\epsilon a\big)\,\natural^{\epsilon}\,
g_{1/\delta}\big](x,\xi)&= \big(\lambda^\epsilon-\epsilon a\big)(\xi)\,
g_{1/\delta}(\xi)
+\,\epsilon\big(B_0+\kappa B(\epsilon
x)\big)\underset{j=1,2}{\sum}\big(\partial_{\xi_j}\lambda^\epsilon\big)(\xi)\,
\big(\partial_{\xi^\bot_j} g_{\delta^{-1}} \big)(\xi) \nonumber \\  
&+\epsilon^2\big(B_0+\kappa B(\epsilon
x)\big)^2\underset{|\alpha|=2}{\sum}\mathscr{M}^{\epsilon,\kappa}
_2\big(\partial_{\xi}^\alpha\lambda^\epsilon,\partial_{\xi^\bot}^\alpha
g_{\delta^{-1}}\big)(x,\xi).
\end{align}

For the first term in \eqref{4.2} we use \eqref{4.0} and notice that 
$\delta^4\epsilon^{-1}=\delta^{4-\mu}$ with $4-\mu >0$ by  \eqref{ScalHyp} 
  and thus  
$$
\delta^4\mathfrak{f}^{\epsilon,\delta}_{(\delta^{-1})}\,\,\sharp^{\epsilon,\kappa}\,\,
r^{\epsilon,\kappa}(a)\ =\ (\delta^4\epsilon^{-1})\big[
\mathfrak{f}^{\epsilon,
\delta}_{(\delta^{-1})}\,\,\sharp^{\epsilon,\kappa}\,\,\big(\epsilon
r^{\epsilon,\kappa}(a)\big)\big]\,,
$$
where 
$
\left\{
\mathfrak{f}^{\epsilon,
\delta}_{(\delta^{-1})}\,\mid\,(\epsilon,\delta)\in[0,
\epsilon_0]\times(0,\delta_0],\,\epsilon=\delta^{\mu}
\right\}$ is a bounded set in $ S^{-\infty}(\Xi)^\circ
$.\\
Thus 
$$
\hspace{-1cm} \|\delta^{4}\mathfrak{f}^{\epsilon,
\delta}_{(\delta^{-1})}\,\,\sharp^{\epsilon,\kappa}\,\,
r^{\epsilon,\kappa}(a)\|_{B_{\epsilon,\kappa}}\leq
(\delta^4\epsilon^{-1})\left\|
\mathfrak{f}^{\epsilon,
\delta}_{(\delta^{-1})}\right\|_{B_{\epsilon,\kappa}}\left\|\epsilon
r^{\epsilon,\kappa}(a)\right\|_{B_{\epsilon,\kappa}}\leq C\, \delta\,.
$$

The analysis of the second term in \eqref{4.2} is more difficult and crucially depends on the special
form of the cut-off function $g$ in \eqref{g-delta}. 
We denote by $\nabla_{\xi^\bot_j}:=(-\partial_{\xi_2},\partial_{\xi_1})$ the orthogonal gradient operator. The
support of all derivatives of $g_{1/\delta^{\circ}}$ is contained in the support of
$g_{1/\delta^{\circ}}$, thus: 
$$
\big(\partial_{\xi_j}\lambda^\epsilon\big)(\xi)\,
\big(\partial_{\xi^\bot_j}g_{1/\delta^{\circ}}\big)(\xi)\ =\
\big(\partial_{\xi_j}(\lambda^\epsilon\,g_{1/\delta}
)\big)(\xi)\,
\big(\partial_{\xi^\bot_j}g_{1/\delta^\circ}\big)(\xi)\,.
$$
Moreover, if we denote by $\chi^\prime$ the derivative of $\chi$
having support in the  annulus  $1\leq|t|\leq2\,$, then due to
the choice \eqref{g-delta} we have the following  equality:
\begin{equation}\label{28aug-3}
\begin{array}{l}
\hspace{-1cm} \big(\nabla_{\xi}(\lambda^\epsilon\,g_{1/\delta}
)\big)(\xi)\cdot \big(\nabla_{\xi^\bot}g_{1/\delta^\circ}\big)(\xi) \nonumber \\
\qquad \qquad =4\,m^\epsilon_1m^\epsilon_2\, (\delta^\circ)^{-1}\lambda^\epsilon(\xi)
\chi^\prime
\big(h_{m^\epsilon}
((\delta^\circ)^{-1}\xi)\big)\big(-\xi_1\xi_2+\xi_2\xi_1\big)
\,+\,\delta^3\, \sum_{j=1}^2\mathfrak{f}^{j,\epsilon,\delta}(\delta^{-1}\xi)\nonumber  \\
\qquad \qquad =
\delta^3\, \sum_{j=1}^2\mathfrak{f}^{j,\epsilon,\delta}(\delta^{-1}\xi)\,.
\end{array}
\end{equation}

In conclusion, the second term in \eqref{4.2} is of the form
$\epsilon\delta^3F^{\epsilon,\delta,\kappa}_{(1,\delta^{-1})}$ for a bounded family of symbols in $S^{-\infty}(\Xi)$
$$
\left\{F^{\epsilon,\delta,\kappa}\,\mid\,(\epsilon,
\delta,\kappa)\in[0,\epsilon_0]\times(0,\delta_0]\times[0,1]\,,\,\epsilon= \delta^{
\mu}\right\}
$$
and we have the estimate:
$$
\epsilon\delta^{3}\|F^{\epsilon,\delta,\kappa}_{(1,\delta^{-1})}\, \sharp^{\epsilon,\kappa }\,
r^{\epsilon,\kappa}(a)\|_{B_{\epsilon,\kappa}}\leq
\delta^3\left\|
F^{\epsilon,\delta,\kappa}_{(1,\delta^{-1})}\right\|_{B_{\epsilon,\kappa}}
\left\|\epsilon
r^{\epsilon,\kappa}(a)\right\|_{B_{\epsilon,\kappa}}\leq C\, \delta^{3}\, .
$$
For the term in  the last line of \eqref{4.2} the same procedure as in the previous proof (see
\eqref{4-1}) allows us to conclude that it defines a bounded  magnetic 
pseudodifferential operator that is small in norm of order $\delta^{\mu-2}$.

This allows us to conclude, using also \eqref{4.0}, that there exists a constant $C>0$ such that the following relation is true:
$$
\big(\lambda^\epsilon-\epsilon a\big)\,\sharp^{\epsilon,\kappa}\,
g_{1/\delta}\,\,\,\sharp^{\epsilon,\kappa}\,\,r^{\epsilon,\kappa}(a)\ =\ 
\big[g_{1/\delta}\big(h_{m^\epsilon}-\epsilon a\big)\big]\, \sharp^{\epsilon,
\kappa}\, r^{\epsilon,\kappa}(a)\,+\,\delta^{\mu-2}\mathfrak{x}^{\epsilon,\delta,
\kappa} \,,
$$
with $\mathfrak{x}^{\epsilon,\delta,
\kappa}\in S^{-\infty}(\Xi)$ and $$\|\mathfrak{x}^{\epsilon,\delta,
\kappa}\|_{B_{\epsilon,\kappa}}\leq C\,,$$ for all 
$(\epsilon,\delta,\kappa)\in[0,\epsilon_0]\times(0,\delta_0]\times[0,\kappa_0]$
s. t. 
$\epsilon=\delta^{\mu}\,$.

A similar procedure allows us to transform $g_{1/\delta}\big(h_{m^\epsilon}-\epsilon a\big)$ into ${g_{1/\delta}\sharp^{\epsilon,\kappa}\big(h_{m^\epsilon}-\epsilon a\big)}$ and use the equality $$\big(h_{m^\epsilon}-\epsilon a\big)\sharp^{\epsilon,\kappa}r^{\epsilon,\kappa}(a)=1\,.
$$
This time the formula similar to \eqref{4.2} will contain factors of the form $\xi_j\big(\partial_{\xi^\bot_j} g_{\delta^{-1}} \big)(\xi)$ that are still symbols of class $S^{-\infty}(\Xi)\,$.
Finally we conclude the existence of $C>0$ such that:
\begin{align*}
\big[g_{1/\delta}\big(h_{m^\epsilon}-\epsilon a\big)\big]\sharp^{\epsilon,
\kappa}r^{\epsilon,\kappa}(a)\
& =\
g_{1/\delta}\,\,\,\sharp^{\epsilon,\kappa}\,\,\big(\lambda^\epsilon-\epsilon a\big)
\,\,\sharp^{\epsilon,\kappa}\,\,
r^{\epsilon,\kappa}(a)\,+\,\delta^{\mu-2}
\mathfrak{z}^{\epsilon,\delta,\kappa}\ \\ & 
=\ g_{1/\delta}\,+\,\delta^{\mu-2}\mathfrak{z}^{\epsilon,\delta,\kappa}\,,
\end{align*}
with $$\mathfrak{z}^{\epsilon,\delta,
\kappa}\in S^{-\infty}(\Xi) \mbox{ and } \|\mathfrak{z}^{\epsilon,\delta,
\kappa}\|_{B_{\epsilon,\kappa}}\leq C\,,$$ 
for all
$(\epsilon,\delta,\kappa)\in[0,\epsilon_0]\times(0,\delta_0]\times[0,\kappa_0]$
s.t. 
$\epsilon=\delta^{\mu}$.

\fussy
\end{proof}

\paragraph{The series in  \eqref{ter-1} continued.}
We shall rewrite the expression in the last line of  \eqref{ter-1} as:
\begin{equation}\label{28aug-1}
\begin{array}{l}
\underset{\gamma^*\in\Gamma_*}{\sum}\lek\,\,\sharp^{\epsilon,\kappa}\,\,
\left(1-\widetilde{g}_{1/\delta}\right)\sharp^{
\epsilon,\kappa}r_{\delta,\epsilon,\kappa}(\epsilon a)\\
=
\big(\lambda^\epsilon-\epsilon a\big)\,\,\sharp^{\epsilon,\kappa}\,\,
\left(1-\widetilde{g}_{1/\delta}\right)\sharp^{
\epsilon,\kappa}r_{\delta,\epsilon,\kappa}(\epsilon a)\\
=\
\left(1-\widetilde{g}_{1/\delta}\right)\,-\,\Big[\lambda^\epsilon,\widetilde{g}
_ {1/\delta}\Big]_{\epsilon,\kappa}\sharp^{
\epsilon,\kappa}r_{\delta,\epsilon,\kappa}(\epsilon a)\,-\,
(\delta^\circ)^2\left(1-\widetilde{g}
_{1/\delta}\right)\,\,\sharp^{\epsilon,\kappa}\,\,\widetilde{g}_{1/\delta^\circ}
\,\,\sharp^{\epsilon,\kappa}\,\,r_{\delta,\epsilon,\kappa}(\epsilon a)\nonumber\\
=\ 1-\underset{\gamma^*\in\Gamma_*}{\sum}\tau_{\gamma^*}\big(g_{1/\delta}
\big)  
-\Big[\lambda^\epsilon,\widetilde{g}
_{1/\delta}\Big]_{\epsilon,\kappa}\sharp^{
\epsilon,\kappa}r_{\delta,\epsilon,\kappa}(\epsilon a)-
(\delta^\circ)^2\left(1-\widetilde{g}
_{1/\delta}\right)\,\,\sharp^{\epsilon,\kappa}\,\,\widetilde{g}_{1/\delta^\circ}
\,\,\sharp^{\epsilon,\kappa}\,\,r_{\delta,\epsilon,\kappa}(\epsilon a)\,,
\nonumber
\end{array}
\end{equation}
where 
$$
[f,g]_{\epsilon,\kappa}:=f\,\,\sharp^{\epsilon,\kappa}\,\,g\,-\,g\sharp^{\epsilon,
\kappa}f\,.
$$

We start with the second term on the right hand side of the second line of \eqref{28aug-1}.

\begin{lemma}
There exist some positive constants $C$, $\epsilon_0\,$, $\kappa_0$ and $\delta_0$ such that,  for
$(\epsilon,\kappa, \delta)\in[0,\epsilon_0]\times[0,\kappa_0] \times(0,\delta_0]$
verifying  \eqref{ScalHyp}\,,
$$
\left\|\Big[\lambda^\epsilon,\widetilde{g}
_{1/\delta}\Big]_{\epsilon,\kappa}\sharp^{
\epsilon,\kappa}r_{\delta,\epsilon,\kappa}(\epsilon a)\right\|_{B_{\epsilon,
\kappa}}\leq\,C\, \big(\delta+\delta^{2(\mu-2)}\big)\,.
$$
\end{lemma}
\begin{proof}
We note that  both
symbols $\lambda^\epsilon$ and $\widetilde{g}_{1/\delta}$ are of class ${
S^0_0(\Xi)}^\circ$ and thus we can apply Proposition \ref{P-comm-sv} in order to
obtain the existence of positive $\epsilon_0$ and $\kappa_0$ such that, for $(\epsilon,\kappa) \in [0,\epsilon_0]\times  [0,\kappa_0]\,$,
\begin{align*}
&\hspace{-0.5cm}\Big[\lambda^\epsilon,\widetilde{g}_{1/\delta}\Big]_{\epsilon,\kappa} (x,\xi) \\
&\hspace{0.5cm}  =\,-4i\epsilon B_\kappa(\epsilon
x)\left[\big(\partial_{\xi_1}\lambda^\epsilon\big)(\xi)\big(\partial_{\xi_2}
\widetilde{g}_{1/\delta}\big)(\xi)-\big(\partial_{\xi_2}
\lambda^\epsilon\big)(\xi)\big(\partial_{\xi_1}\widetilde{g}_{1/\delta}
\big)(\xi)\right]\,+\,\epsilon^2\widetilde{\mathscr{R}}_{\epsilon,\kappa}
(\lambda^\epsilon,\widetilde{g}_{1/\delta})_{(\epsilon,1)}\,,\nonumber
\end{align*}
where
$\{\widetilde{\mathscr{R}}_{\epsilon,\kappa}\big(\lambda^\epsilon,\widetilde{g}_
{1/\delta}\big)\}_{\epsilon\in[0,
\epsilon_0], \kappa \in [0,\kappa_0]}$ is a bounded family in 
$S^{-\infty}(\Xi)\,$.\\
For the first two terms we
apply once again  \eqref{28aug-3} and obtain 
$$
\big(\partial_{\xi_1}\lambda^\epsilon\big)(\xi)\big(\partial_{\xi_2}\widetilde{g
}_{1/\delta}\big)-
\big(\partial_{\xi_2}\lambda^\epsilon\big)
\big(\partial_{\xi_1}\widetilde{g}_{1/\delta}\big)\ =\
\delta^3\, \mathfrak{f}^{\epsilon,\delta}_{1/\delta}\,,
$$
where
$\left\{\mathfrak{f}^{\epsilon,\delta}\,\mid\,(\epsilon,
\delta)\in[0,\epsilon_0]\times(0,\delta_0],\,\epsilon=\delta^{\mu}\right\}$ is a bounded subset in $C^\infty_0(\X^*)\,$.\\
For the third term we have 
\begin{align*}
\widetilde{\mathscr{R}}_{\epsilon,\kappa}\big(\lambda^\epsilon,\widetilde{g}_{1/\delta}\big)(x,\xi)&=
-\frac{B^2_\kappa(x)}{\pi^2}\underset{|\alpha|=2}{\sum}
(\alpha!)^{-1}\hspace*{-12pt}
\int\limits_{\X^*\times\X^*}\hspace*{-12pt}
e^{i(\eta\wedge\zeta)}\int_0^1s\,
ds\,\mathfrak{T}_{\epsilon,\kappa,\delta}(X,s\eta,\zeta)
d^2\eta\,d^2\zeta\,
\\
&\quad +\,\kappa \int\limits_{\Xi\times\Xi}e^{-2i\sigma(Y,Z)} \left(\int_0^1\Theta^{\epsilon,\kappa}_\tau(x,y,z)d\tau\right)
\ \mathfrak{R}_{\epsilon,\kappa,\delta}(X,Y,Z)\,dYdZ\,,
\end{align*}
where
\begin{align*}
\mathfrak{T}_{\epsilon,\kappa,\delta}(X,s\eta,\zeta)&:=
\big(\partial^\alpha_\xi\lambda^\epsilon\big)(\xi-s\epsilon^{1/2}b_\kappa(
x)\eta)\big(\partial^\alpha_{\xi^\bot}\widetilde{g}_{1/\delta}
\big)(\xi-\epsilon^{1/2}
b_\kappa(x)\zeta)\,,
\end{align*}
and 
\begin{align*}
\mathfrak{R}_{\epsilon,\kappa,\delta}(X,Y,Z)&:=
\underset{j,k,\ell}{\sum}
\left(R_1\big(\partial_\ell B_{jk}\big)(x,\epsilon
y,\epsilon z)\big(\partial_{\xi_k}\lambda^\epsilon\big)(
\xi-\eta)\big(\partial^2_{\xi_j\xi_\ell}\widetilde{g}_{1/\delta}\big)(\xi-\zeta) \right.\\
&\quad \qquad +\left . R_2\big(\partial_\ell B_{jk}\big)(x,\epsilon
y,\epsilon z)\big(\partial^2_{\xi_k\xi_\ell }\lambda^\epsilon\big)(
\xi-\eta)\big(\partial_{\xi_j}\widetilde{g}_{1/\delta}
\big)(\xi-\zeta)\right)
\\
&
=\underset{j,k,\ell}{\sum}\left(R_1\big(\partial_\ell B_{jk}\big)(x,
\epsilon
\delta^{-1}y,\epsilon z)\big(\partial_{\xi_k}\lambda^\epsilon\big)(
\xi-\eta)\big(\partial^2_{\xi_j\xi_\ell}\widetilde{g}\big)(\delta^{-1}\xi-\zeta)\,
\right.
\\
&\quad\qquad\left.+R_2\big(\partial_\ell B_{jk}\big)(x,\epsilon
\delta^{-1}y,\epsilon z)\big(\partial^2_{\xi_k\xi_\ell}\lambda^\epsilon\big)(
\xi-\eta)\big(\partial_{\xi_j}\widetilde{g}\big)(\delta^{-1}
\xi-\zeta)\right),
\end{align*}
with $R_1(\cdot )$ and $R_2(\cdot )$ defined by \eqref{def-R1} and \eqref{def-R2}.\\
Taking into account   \eqref{ScalHyp} and using Proposition
\ref{comp-epsilon-symb} we obtain: 
$$
\Big[\lambda^\epsilon,\widetilde{g}_{1/\delta}\Big]_{\epsilon,\kappa}\ =\ 
\delta^3\, \mathfrak{f}^{\epsilon,\delta}_{1/\delta}\,+\,
\epsilon^2\, \delta^{-2}\, \widetilde{\mathscr{R}^\prime}_{\epsilon,\kappa,\delta}
(\lambda^\epsilon,\widetilde{g})_{(\epsilon,\delta^{-1})}
$$
where, for some $\epsilon_0>0$ and $\kappa_0>0$ small enough, the following families 
$$
\left\{\widetilde{\mathscr{R}^\prime}_{\epsilon,\kappa,\delta}
\big(\lambda^\epsilon,\widetilde{g}\big)\,\mid\,(\epsilon,\kappa,\delta)\in[0,
\epsilon_0]\times[0,\kappa_0]\times(0,\delta_0],\,\epsilon\delta^\mu=1\right\}\,\subset\,
S^{-\infty}(\Xi)
$$ 
and
$$
\left\{\mathfrak{f}^{\epsilon,\delta}\,\mid\,(\epsilon,
\delta)\in[0,\epsilon_0]\times(0,\delta_0],\,\epsilon=\delta^{\mu}\right\}\subset C^\infty_0(\X^*)
$$
are bounded. 
\end{proof}

For the third term in the right hand side of \eqref{28aug-1} we proceed
similarly using Proposition~\ref{P-disj-supp-sv} and obtain
\begin{lemma} There exist $\epsilon_0 >0$, $\kappa_0 >0$, $\delta_0 >0$ and $C>0$ such that:
$$
\left\|(\delta^\circ)^2\left(1-\widetilde{g}
_{1/\delta}\right)\,\,\sharp^{\epsilon,\kappa}\,\,\widetilde{g}_{1/\delta^\circ}
\,\,\sharp^{\epsilon,\kappa}\,\,r_{\delta,\epsilon,\kappa}(\epsilon a)\right\|_{B_{
\epsilon,\kappa}}\leq\, C \, \delta^{2(\mu-2)}\,,
$$
for all
$(\epsilon,\kappa,\delta)\in[0,\epsilon_0]\times[0,\kappa_0] \times(0,\delta_0]$
verifying  \eqref{ScalHyp}.
\end{lemma}

\paragraph*{The rest of the double series in the first term of \eqref{ter-1}.}

We still have to study the following double series:
\beq
\begin{array}{l}
\underset{\gamma^*\in\Gamma_*}{\sum}\lek\,\,\sharp^{\epsilon,\kappa}\,\,\underset{
\alpha^*\in\Gamma_*\setminus\{\gamma^*\}}{\sum}
\tau_{\alpha^*}\big(g_{1/\delta}\,\,\sharp^{\epsilon,\kappa}\,\,
r^{\epsilon,\kappa}(a)\big)\ 
\\
\qquad 
=\
\big(\lambda^\epsilon\,-\,\epsilon a\big)\,\,\sharp^{\epsilon,\kappa}\,\,\underset{
\gamma^*\in\Gamma_*}{\sum}\tau_{\gamma^*}\widetilde{\chi}\sharp^{\epsilon,\kappa
}\underset{
\alpha^*\in\Gamma_*\setminus\{\gamma^*\}}{\sum}
\tau_{\alpha^*}\big(g_{1/\delta}\,\,\sharp^{\epsilon,\kappa}\,\,
r^{\epsilon,\kappa}(a)\big)\ \\
\qquad 
=\
\big(\lambda^\epsilon\,-\,\epsilon a\big)\,\,\sharp^{\epsilon,\kappa}\,\,
\left[\underset{
\gamma^*\in\Gamma_*}{\sum}\tau_{\gamma^*}\left(\underset{
\alpha^*\in\Gamma_*\setminus\{0\}}{\sum}
\big(\widetilde{\chi}\,\,\sharp^{\epsilon,\kappa}\,\,
\tau_{\alpha^*}(g_{1/\delta})\big)\,\,\sharp^{\epsilon,\kappa}\,\,
\tau_{\alpha^*}(r^{\epsilon,\kappa}(a))\right)\right]\,.
\end{array}
\eeq
For the symbol $\widetilde{\chi}\,\,\sharp^{\epsilon,\kappa}\,\,
\tau_{\alpha^*}(g_{1/\delta})$ we proceed as in the proof of Proposition
\ref{P-2-ter-1-main} noticing that for $\alpha^*\ne0$ the supports of
$\widetilde{\chi}$ and $\tau_{\alpha^*}(g_{1/\delta})$ are disjoint and we can
use Proposition \ref{P-disj-supp-sv} in order to obtain the following statement.

\begin{lemma}
Given $\epsilon_0>0$ and $\delta_0>0$ satisfying  \eqref{delta-0},  
for any $N$, there exists $C_N>0$ such that, for any $\alpha^*\in\Gamma_*\,$, 
$$
<\alpha^*>^N\widetilde{\chi}\,\,\sharp^{\epsilon,\kappa}\,\,
\tau_{\alpha^*}(g_{1/\delta})\,=\,\mathfrak{x}^{
\epsilon,\delta,\kappa}_{(\epsilon,\delta^{-1})}\,,
$$
where the family 
$$
\{\mathfrak{x}^{\epsilon,\delta,\kappa}\,\mid\,(\epsilon,\delta,\kappa)\in[0,
\epsilon_0]\times(0,\delta_0]\times[0,1]\,,\,\epsilon=\delta^{\mu}\}\,,
$$  
is  bounded  in $ S^{-\infty}(\Xi)$
and
$$
\|\mathfrak{x}^{
\epsilon,\delta,\kappa}_{(\epsilon,\delta^{-1})}\|_{B_{\epsilon,\kappa}}\leq
C_N\, \delta^{2(\mu-1)}\,.
$$
\end{lemma}
Thus for any $N\in\mathbb{N}$ and for any
$\alpha^*\ne0$\,,
$$
\left\|\big(\widetilde{\chi}\,\,\sharp^{\epsilon,\kappa}\,\,
\tau_{\alpha^*}(g_{1/\delta})\big)\,\,\sharp^{\epsilon,\kappa}\,\,
\tau_{\alpha^*}(r^{\epsilon,\kappa}(a))\right\|_{B_{\epsilon,\kappa}}\leq
C_N\, <\alpha^*>^{-N}\,\delta^{\mu-2}\,.
$$

We now use once again the formulas in \eqref{rez-eq-1} and
proceed similarly with the proofs of Propositions~\ref{P-25aug-2} and
\ref{P-25aug-3} in order to  obtain  the following final result.
\begin{proposition} 
 There exist $\epsilon_0 >0$, $\kappa_0 >0$, $\delta_0 >0$ and $C>0$ such that:
$$
\left\|\underset{\gamma^*\in\Gamma_*}{\sum}\lek\,\,\sharp^{\epsilon,\kappa}\,\,\underset
{
\alpha^*\in\Gamma_*\setminus\{\gamma^*\}}{\sum}
\tau_{\alpha^*}\big(g_{1/\delta}\,\,\sharp^{\epsilon,\kappa}\,\,
r^{\epsilon,\kappa}(a)\big)\right\|_{B_{\epsilon,\kappa}}\leq C\, \delta^{\mu-2}\,.
$$
for all
$(\epsilon,\kappa,\delta)\in[0,\epsilon_0]\times[0,\kappa_0] \times(0,\delta_0]$
verifying \eqref{ScalHyp}.
\end{proposition}

\appendix
\section*{Appendices}
\section{Study of $\lambda_0(\theta)$}\label{A-lambda0}

Being isolated from the rest of the spectrum, the first Bloch eigenvalue $\lambda_0$ is analytic \cite{Ka, RS-4}. In order to prove the statement in Remark \ref{Rem-lambda-0} (see also \cite{KS}) it is enough to prove the following
\begin{lemma}
There exists $C>0$ such that
\begin{equation}\label{lbtheta} 
\lambda_0 (\theta)-\lambda_0(0) \geq C\,  |\theta|^2 \,,\, \forall \theta \in
E_*.
\end{equation}
\end{lemma}

\begin{proof}
Up to a Perron-Frobenius type argument \cite{RS-4}, the first $L^2$-normalized eigenfunction of the elliptic operator $-\Delta + V$ on the compact manifold $\mathbb T$ is non-degenerate and can be chosen  positive; let us denote it by $u_0$. We also have $0<\min(u_0)\leq u_0(x)\leq \max(u_0)$. For $u\in C^\infty (\mathbb T)$, if we denote by $u= u_0 v$ we can write 
$
(- i \nabla - \theta) u = -i [u_0(\nabla -i\theta)v + v\nabla u_0]\,.
$
This implies
$$
| (- i \nabla - \theta) u |^2  = u_0^2\, | (\nabla-i\theta) v|^2 + |\nabla u_0|^2 \, |v|^2 + u_0 (\nabla |v|^2) \cdot \nabla u_0\,.
$$
Integrating on  $\mathbb{T}$ leads to:
$$
\int_\mathbb T | (- i \nabla - \theta) u |^2 \,dx   = \int_\mathbb T  u_0^2\, | (\nabla-i\theta) v|^2\, dx   -  \int_\mathbb T |v|^2 \, (u_0\Delta u_0)\,dx ,
$$
while using the equation $-\Delta u_0=(\lambda(0)-V)u_0$ we obtain:
$$
\int_\mathbb T | (- i \nabla - \theta) u |^2\,dx  +  \int_\mathbb T (V-\lambda_0(0)) \,  |u|^2\,dx   = \int_\mathbb T  u_0^2\, | (\nabla-i\theta) v|^2\, dx   \,.
$$
If $u=u_\theta$ is an eigenfunction of $ (-i\nabla -\theta)^2 + V$ associated with $\lambda(\theta)$, we get, with $v_\theta =u_\theta/u_0\,$, 
$$
\lambda(\theta) - \lambda (0) = \frac{\int_\mathbb T u_0^2 \, | (\nabla-i\theta) v_\theta|^2\, dx  }{ \int _\mathbb T |v_\theta|^2 \,u_0^2 \,dx}
\geq 
  \frac{\min(u_0)^2}{\max(u_0)^2}\; 
  \inf_{v\in C^\infty(\mathbb{T})} \frac{\int_\mathbb T (|(\nabla-i\theta) v|^2 ) dx}{\int_\mathbb T v^2\, dx}\,.
$$
On the right hand side we recognize the variational  characterization of  the ground state energy of the free Laplacian on the torus, which equals $|\theta|^2$ and it is achieved for a constant $v$. Hence \eqref{lbtheta} holds, with $C =  \frac{\min(u_0)^2}{\max(u_0)^2}\,$. 
\end{proof}

\section{The magnetic Moyal calculus with slowly varying symbols}\label{A-sl-var-symb}
This appendix is devoted to a brief reminder of the main definitions, notation and results concerning the magnetic pseudodifferential calculus \cite{IMP1,CP-1,CP-2} and to prove some special properties in the case of {\it slowly varying symbols and magnetic fields} (see Subsection \ref{SS-sl-var-m-field}).

\subsection{H\"{o}rmander classes of symbols}
\label{A-not}

Let us recall the notation $\X:=\mathbb{R}^2$ and let us denote by $\X^*$ the dual of $\X$ ({\it the
momentum space}) with
$\langle\cdot,\cdot\rangle:\X^*\times\X\rightarrow\mathbb{R}$ denoting the duality map.
Let $\Xi:=\X\times\X^*$ be {\it the phase space} with the
canonical symplectic form
\beq\label{can-sympl-form}
\sigma(X,Y):=\langle\xi,y\rangle-\langle\eta,x\rangle\,,
\eeq
for $X:=(x,\xi)\in\Xi$
and $Y:=(y,\eta)\in\Xi^*$. \\
We consider the spaces $BC(\V)$ of bounded
continuous functions on any finite dimensional real vector space
$\mathcal{V}$ with the
$\|\cdot\|_\infty$ norm.
We shall denote by $C^\infty(\mathcal{V})$ the space of smooth functions on
$\mathcal{V}$ and by $C^\infty_{\text{\sf
pol}}(\mathcal{V})$ (resp.  by $BC^\infty(\V)$)  its subspace of smooth functions
that are polynomially bounded together with all their derivatives, (resp. smooth
and bounded together with all their derivatives), endowed with the usual  locally 
convex topologies. 
We denote by $\tau_v$ the translation with $v\in\mathcal{V}$.
For any $v\in\mathcal{V}$ we write  $<v>:=\sqrt{1+|v|^2}\,$.
$\mathscr{S}(\mathcal{V})$ denotes the space of Schwartz 
functions on $\mathcal{V}$ endowed
with the Fr\'{e}chet topology defined by the following family $\{\nu_{m,n} \}_{(m,n) \in \mathbb N^2}$ of seminorms:
$$
\mathscr{S}(\mathcal{V})\ni \phi \mapsto \nu_{n,m}(\phi):=\underset{v\in\mathcal{V}}{\sup}
<v>^n\underset{|\alpha|\leq
m}{\sum}\left|\big(\partial^\alpha\phi\big)(v)\right|\,.
$$
We will use the following class of H\"{o}rmander type symbols. 
\begin{definition}\label{Def-symb}
For any
$s\in\mathbb{R}$ and any $\rho\in[0,1]$,  we denote by
$$
S^s_\rho(\Xi):=\{F\in
C^\infty(\Xi)\mid\nu^{s,\rho}_{n,m}(F)<+\infty\,,\forall(n,m)\in\mathbb{N}
\times\mathbb{N}\}\,,
$$
where
$$
\nu^{s,\rho}_{n,m}(f):=\underset{(x,\xi)\in\Xi}{\sup}\underset{|\alpha|\leq
n}{\sum}\underset{|\beta|\leq m}{\sum}\left|\langle \xi\rangle
^{-s+\rho m}
\big(\partial^\alpha_x\partial^\beta_\xi f\big)(x,\xi)\right|\,,
$$ 
and
$$S^\infty_\rho(\Xi):=\underset{s\in\mathbb{R}}{\bigcup}S^s_\rho(\Xi) \mbox{
 and } S^{-\infty}(\Xi):=\underset{s\in\mathbb{R}}{\bigcap}S^s_\rho(\Xi)\,.$$
\end{definition}
\begin{definition}
A symbol $F$ in $S^s_\rho(\Xi)$ is called {\it
elliptic}  
if there exist two positive constants $R$ and $C$ such that $$|F(x,\xi)|\geq
C \, \langle \xi\rangle ^s\,,$$ 
for any
$(x,\xi)\in\Xi$ with $|\xi|\geq R\,$.
\end{definition}
\begin{definition}
For
$h\in S^m_1(\Xi)$ the \textit{Weyl quantization} associates the operator
$\mathfrak{Op^w}(h)$ defined, for  $u\in \mathscr{S}(\X)$, by 
\beq\label{Op0} 
\big(\mathfrak{Op^w}(h)u\big)(x)\ :=\, 
(2\pi)^{-2}\int_\X\int_{\X^*}e^{i\langle
\xi,x-y\rangle} \, h \big(\frac{x+y}
{2},\xi\big)u(y)\,d\xi\,dy\,.
\eeq
\end{definition}
This operator is continuous on $ \mathscr{S}(\X)$
 and has a natural extension by duality to
$\mathscr{S}^\prime(\X)$. Moreover it is just the special case $\epsilon=1$ of the operator in \eqref{Opepsilon}   (but in dimension $2$).

\subsection{Magnetic pseudodifferential calculus}
\label{SS-Main-Def}
We consider a vector potential $A$ with components of class
$C^\infty_{\text{\sf pol}}(\X)\,$, such that  the magnetic field  $B = \curl A$
belongs to $ BC^\infty(\X)$.
We  recall from \cite{IMP1} that the functional calculus (see \eqref{OpA})
\beq\label{OpAraapp}
u\mapsto \big(\mathfrak{Op}^A(F)u\big)(x)\ :=\
(2\pi)^{-2}\int_\X\int_{\X^*}e^{i\langle
\xi,x-y\rangle}e^{-i\int_{[x,y]}A}\, F\left(\frac{x+y}
{2},\xi\right )u(y)\,d\xi\,dy\,,
\eeq
generalizes the usual Weyl calculus and extends to the H\"{o}rmander classes of
symbols. Moreover the usual {\it composition of symbols theorem} is still valid
(Theorem 2.2 in \cite{IMP1}). In fact for our special classes of H\"{o}rmander
symbols   ($S^m_{\rho}(\Xi)\equiv S^m_{\rho,0}(\Xi)$) the result
concerning the composition of symbols is a corollary of  Proposition
\ref{comp-epsilon-symb}.

We recall (see  Proposition 3.4 in \cite{MP1}) that two vector
potentials that are
gauge equivalent define two unitarily equivalent functional calculi and that (Proposition 3.5 in \cite{MP1})  the application $\mathfrak{Op}^A$  extends to a linear and
topological isomorphism between $\mathscr{S}^\prime(\Xi)$ and
${\mathcal L}\big(\mathscr{S}(\X);\mathscr{S}^\prime(\X)\big)$ (considered with
the strong topologies). 

In the same spirit as the Calderon-Vaillancourt theorem for classical pseudo-differential operators, Theorem 3.1 in \cite{IMP1} states that any symbol $F\in S^0_0(\Xi)$ defines a
bounded operator $\mathfrak{Op}^A(F)$ in $L^2(\X)$ with an upper bound of  the operator norm
 by some symbol
seminorm of $F$.  We
denote by $\|F\|_B$ the operator norm of $\mathfrak{Op}^A(F)$ in $\mathcal L
(L^2(\X))$ :
\begin{equation}\label{defFb}
\|F\|_B:= ||\mathfrak{Op}^A(F)||_{\mathcal L (L^2(\X))}\,.
\end{equation}
This norm  only depends on the
magnetic field $B$ and not on the choice of the vector potential (different
choices being unitary equivalent).

We also recall from \cite{MP1} that the operator composition of the
operators $\mathfrak{Op}^A(F)$ and $\mathfrak{Op}^A(G)$ induces a {\it twisted
Moyal product}, also called magnetic Moyal product, such that
$$
\mathfrak{Op}^A(F)\, \mathfrak{Op}^A(G)=\mathfrak{Op}^A(F \, \,\sharp^B\,
\,G)\,.
$$
This product depends only on the magnetic field $B$ and is  given by the following oscillating
integral:
\begin{align}\label{II.1.5}
\big(F \, \,\sharp^B\, \,G\big)(X)&:=\pi^{-4}\int_\Xi dY\int_\Xi
dZ\,e^{-2i\sigma(Y,Z)}e^{-i\int_{T(x,y,z)}B}F(X-Y)\, G(X-Z)\\
&=\pi^{-4}\int_\Xi dY\int_\Xi
dZ\,e^{-2i\sigma(X-Y,X-Z)}e^{-i\int_{\widetilde{T}(x,y,z)}B}
F(Y)\, G(Z)\nonumber  \,,
\end{align}
where $T(x,y,z)$ denotes the triangle in $\X$ of vertices $x-y-z\,$,
$x+y-z\,$, $x-y+z$ and  $\widetilde{T}(x,y,z)$ the triangle in $\X$ of vertices
$x-y+z\,$, $y-z+x\,,z-x+y\,$. 

For any symbol $F$ we denote by $F^-_B$ its inverse
with respect to the magnetic Moyal product, if it exists. It is shown in
Subsection~2.1 of \cite{MPR1} that, for  any $m>0$ and  for $a>0$ large enough (depending on $m$)
the symbol
$\mathfrak{s}_m(x,\xi):=<\xi>^m+a$, has an inverse for the magnetic
Moyal product. We shall use the shorthand notation 
$\mathfrak{s}^B_{-m}$ instead of $\big(\mathfrak{s}_m\big)^-_B$ and extend it
to any $m\in\mathbb{R}$ (thus for $m>0$ we have
$\mathfrak{s}^B_{m}\equiv\mathfrak{s}_{m}$).
\\
The following results have been established in \cite{IMP2} (Propositions 6.2 and 6.3):
\begin{proposition}~
\begin{enumerate}
\item 
If $F\in S^0_\rho(\Xi)$ is invertible for
the magnetic Moyal product, then the inverse
$F^-_B$ also belongs to $S^0_\rho(\Xi)\,$.
\item  For $m<0$, if $f\in S^{m}_\rho(\Xi)$ is
such that $1+f$ is invertible  for
the magnetic Moyal product, then $(1+f)^-_B-1\in
S^m_\rho(\Xi)\,$.
\item 
 Let $m>0$ and $\rho\in[0,1]\,$. If $G\in
S^m_\rho(\Xi)$ is invertible  for
the magnetic Moyal product, with $\mathfrak{Op}^A\big(\mathfrak
s_{m}\, \sharp^BG^{-}_B\big)\in\mathcal{L}\big(L^2(\X)\big)\,$, then $G^-_B\in
S^{-m}_\rho(\Xi)\,$.
\end{enumerate}
\end{proposition}
We
recall from \cite{IMP2} that one can associate with $X \in \Xi$, a  linear symbol by:
$$ \mathfrak{l}_X(Y):=\sigma(X,Y)\,,\, \forall \,Y\in \Xi\, , $$
and, for $\epsilon \geq 0$,  an operator $\mathfrak{ad}_X^\epsilon$ on $\mathcal
S'(\Xi)$
\beq\label{ad-op}
\mathfrak{ad}_X^\epsilon[\psi]
:=\mathfrak{l}_X\,\sharp^\epsilon\,\psi-\psi\,\sharp^\epsilon\,\mathfrak{l}_X\,,
\qquad\forall\, \psi\in\mathscr{S}^\prime(\Xi)\,.
\eeq

\begin{proposition}\label{Gamma*per}
Let us consider a lattice $\Gamma_*\subset\X^*$. If $f$ and $g$ are $\Gamma_*$-periodic symbols, then their magnetic Moyal product is also $\Gamma_*$-periodic.
\end{proposition}
\noindent The proof is evident by \eqref{II.1.5}.
\begin{remark}\label{Rem-s-adj}
By Theorem 4.1 in \cite{IMP1},  for any real elliptic symbol
$h\in S^m_1(\Xi)_\Gamma$ (with $m>0$) and   for any $A$ in $C^\infty_{\text{\sf
pol}}(\X,\mathbb R^2)$, the operator
$\mathfrak{Op}^A(h)$ has a closure $H^A$ in $L^2(\X)$ that is self-adjoint on a domain
$\mathcal{H}^m_A$ (a {\it magnetic Sobolev space}) and lower semibounded. Thus
we can define its resolvent $(H^A-\z)^{-1}$ for any $\z\notin\sigma(H^A)$ and
Theorem 6.5 in \cite{IMP2} states that it exists a symbol
$r^B_\z(h)\in S^{-m}_1(\Xi)$ such that 
$$
(H^A-\z)^{-1}\ =\ \mathfrak{Op}^A(r^B_\z(h))\,.
$$
\end{remark}
\begin{remark}
For symbols of class $S^0_\rho(\Xi)$ with $\rho\in[0,1]$, we have seen that the
associated magnetic pseudodifferential operator is bounded in
$\mathcal{H}$ and is self-adjoint if and only if its symbol is real. In that
case we can also define its resolvent and the results in \cite{IMP2}, cited
above, show that it is also defined by a symbol of class $S^0_\rho(\Xi)\,$.
\end{remark}
\noindent Using the notation 
\begin{equation}\label{defLambda}
{\Lambda}^A(x,y)=e^{-i\int_{[x,y]}A}\,,
\end{equation}
applying the Stokes theorem we get the identity (see also \eqref{D-Omega'})
\begin{equation}\label{D-Omega}
\Omega^B(x,y,z)=\exp\left\{-i\int_{<x,y,z>}\hspace*{-10pt}
B\right\}={\Lambda}^A(x,y){\Lambda}^A(y,z){\Lambda}
^A(z,x)
\end{equation}

and there exists $C >0$ such that,  for any magnetic field $B$ of class
$BC^\infty(\X)\,$, 
\beq\label{Est-Omega}
\left|\Omega^B(x,y,z)-1\right|\leq C\, \|B\|_\infty
|(y-x)\wedge(z-x)|\,.
\eeq
The above integral of the $2$-form $B=dA$ is taken on the positively oriented
triangle $<x\,,\,y\,,\,z>\,$.

As we are working in a two-dimensional framework, we  use the following
notation for the {\it vector product} of two vectors $u$ and
$v$ in $\mathbb R^2$\,:
\beq\label{Def-wedge}
u\wedge v\ :=\ u_1v_2\,-\,u_2v_1\,.
\eeq
and the $-\pi/2$ rotation of $v$:
\beq\label{Def-bot}
v^\bot:=\big(v_2,-v_1\big)\,.
\eeq

We can make the connection with the `{\it twisted integral kernels}' formalism in
\cite{Ne02}, where for any integral kernel $K\in\mathscr{S}^\prime(\X\times\X)$
one associates a twisted integral kernel
\beq\label{twist-int-kernel}
K^A(x,y)\ :=\ {\Lambda}^A(x,y)\, K(x,y)\,.
\eeq
For any integral
kernel $K\in\mathscr{S}^\prime(\X\times\X)$ we denote by
$\mathcal{I}\text{\sf nt\,}\,K$ its corresponding linear operator from
$\mathscr{S}(\X)$ into $\mathscr{S}^\prime(\X)$: 
$$\big((\mathcal{I}\text{\sf nt\,}\,K)u\big)(\phi)=\int_{\X \times \X} K(x,y)\phi
(x) u(y)\,dx dy\,,\, \forall \phi \in \mathcal S(\X)\,.
$$
Let us recall that there exists a  linear bijection
$\mathfrak{W}:\mathscr{S}^\prime(\Xi)\rightarrow\mathscr{S}^\prime(\X\times\X)$
defined by 
 \beq\label{iulie3}
\big(\mathfrak{W}F\big)(x,y)\ :=\
(2\pi)^{-2}\int_{\X^*}e^{i<\xi,x-y>}F\big(\frac{x+y}{2},\xi\big)\,d\xi\,,
\eeq
such that
$$ \mathfrak{Op}^w(F)=\mathcal{I}\text{\sf
nt}(\mathfrak{W}F)\,.$$
In the magnetic calculus, we have the equality
\beq\label{magn-quant}
\mathfrak{Op}^A(F)\ =\ \mathcal{I}\text{\sf
nt}({\Lambda}^A\,\mathfrak{W}F)\,.
\eeq

As our main operator $H^0$ is a differential operator, we shall have to use instead of formula \eqref{twist-int-kernel} the following commutation relation (valid for any $z\in\X$)
\beq
\Lambda^A(x,z)^{-1}\big(-i\partial_{x_j}\big)\Lambda^A(x,z)\,=\,A_j(x)+a_j(x,z)\,,
\eeq
with 
\beq\label{comm-B-moment}
a_j(x,z)\,:=\,\underset{k}{\sum}(x_k-z_k)\int_0^1B_{jk}\big(z+s(x-z)\big)s\,ds\,.
\eeq

\subsection{The scaling of symbols}\label{SS-scal-symbols}

For any symbol $F\in S^\infty_\rho(\Xi)$ and for any pair  $(\epsilon,\tau)\in[0,+\infty)\times[0,+\infty)$ we define the scaled symbol:
\beq\label{scaled-symb}
F_{(\epsilon,\tau)}(x,\xi):=F(\epsilon x,\tau\xi)\,.
\eeq

\begin{proposition}\label{scaled-norm}
Suppose given
a symbol $ F\in S^m_\rho(\Xi)$, with either $m<0$ and $\rho\in(0,1]$ or
$m <-2$ and $\rho=0$ and a magnetic field
$B\in BC^\infty(\X)$. Then, there exists $p_0>0$ such that for any $p>p_0$ there exists $C_p>0$ such that
$$
\|F_{(\epsilon,\tau)}\|_{B}\,\leq\,C_p\, \nu^{m,\rho}_{(0,p)}\big(F\big)\,,
$$
for any 
$(\epsilon,\tau)\in[0,+\infty)\times[0,+\infty)\,$. \\
In fact $p_0=(2-m)\rho^{-1}$ for the first
case and  $p_0=2$ for the second case\,.
\end{proposition}
\begin{proof}
We use the Schur-Holmgren criterion for the integral kernel associated with 
the magnetic quantization of the given symbol, introducing the following
notation for the Schur-Holmgren norm:
$$
\|K\|_{SH}:=\max\left \{\underset{x\in\X}{\sup}\int_\X|K(x,y)|\,dy\, ,\quad \underset{y\in\X}{\sup}\int_\X|K(x,y)|\,dx \right \}\,.
$$
Then we have the Schur estimate:
$$
\big\|\mathscr{I}\text{\sf nt\,}K\big\|_{{ \mathcal
L}(\mathcal{H})}\leq\|K\|_{SH}.
$$
If $A$ is a vector potential for the given
magnetic field $B$ we recall from \eqref{magn-quant}:
$$
\mathfrak{Op}^A(F)=\mathscr{I}\text{\sf
nt}\big({\Lambda}^A\mathfrak{W}F\big)\,.
$$
Thus we obtain the { estimate}:
$$
\big\|\mathfrak{Op}^A(F)\big\|_{{ \mathcal L}(\mathcal{H})}=\big\|
\mathscr{I}\text{\sf
nt}\big({\Lambda}^A\mathfrak{W}F\big)\big\|_{{ \mathcal L}(\mathcal{H})}
\leq
\big\|\big({\Lambda}^A\mathfrak{W}F\big)\big\|_{SH}
=\big\|\mathfrak{W}F\big\|_{SH}\,.
$$
Now using \eqref{iulie3} we can write:
\begin{align*}
\big(\mathfrak{W}F_{(\epsilon,\tau)}\big)(x,y)&=(2\pi)^{-2}\int_{\X^*}e^{i<\xi,
x-y>}F_{(\epsilon,\tau)}\big(\frac{x+y}{2},\xi\big)\,d\xi\\
&=(2\pi)^{-2}\int_{\X^*}
e^ {i<\xi,
x-y>}F\big(\epsilon\frac{x+y}{2},\tau\xi\big)\, d\xi\\
&=(2\pi)^{-2}\tau^{-1}\int_{\X^*}
e^ {i<\xi,\tau^{-1}(
x-y)>}F\big(\epsilon\frac{x+y}{2},\xi\big)\, d\xi\,,
\end{align*}
and thus
\begin{equation}\label{F1}
\begin{array}{ll}
&\big\|\mathfrak{W}F_{(\epsilon,\tau)}\big\|_{SH} =\underset{x\in\X}{\sup}
\int_\X|\mathfrak{W}F_{(\epsilon,\tau)}(x,y)|\, dy \\
&=(2\pi)^{-2}\tau^{-1}\underset{x\in\X}{\sup}
\int_\X\left|\int_{\X^*}
e^ {i<\xi,\tau^{-1}(
x-y)>}F\big(\epsilon\frac{x+y}{2},\xi\big)d\xi\right|\,dy \\
&=(2\pi)^{-2}\tau^{-1}\underset{x\in\X}{\sup}
\int_\X\left|\int_{\X^*}
e^ {i<\xi,(2/\tau)v>}F\big(\epsilon(x-v),\xi\big)d\xi\right|\, dv \\
&=(4\pi)^{-2}\underset{x\in\X}{\sup}
\int_\X\left|\int_{\X^*}
e^ {i<\xi,u>}F\big(\epsilon(x-(\tau/2)u),\xi\big)d\xi\right|\, du \nonumber\\
&
=(4\pi)^{-2}\underset{x\in\X}{\sup}
\int_\X<u>^{-p}\left|\int_{\X^*}
e^ {i<\xi,u>}\left[(\bb1-(1/2)\Delta_\xi)^{p/2}F\right]\big(
\epsilon(x-(\tau/2)
u),\xi\big)d\xi\right|\,du\,.
\end{array}
\end{equation}

Now suppose that $m<-2$ so that for any $x\in\X$ 
and any $(\alpha,\beta)\in\mathbb{N}^2\times\mathbb{N}^2\,,$
$\big(\partial_x^\alpha\partial_\xi^\beta F\big)(x,\cdot)$ is in $L^1(\X^*)$
and
$$
\underset{x\in\X}{\sup}
\big\|\big(\partial_x^\alpha\partial_\xi^\beta
F\big)(x,\cdot)\big\|_{L^1(\X^*)}\leq
C_{\alpha,\beta}\, \nu^{m,0}_{|\alpha|,|\beta|}(F)\,.
$$
Then from \eqref{F1} we deduce that
\begin{align*}
\big\|\mathfrak{W}F_{(\epsilon,\tau)}\big\|_{SH}&\leq(4\pi)^{-2}\underset{
z\in\X }{\sup}\big\|(\bb1-(1/2)\Delta_\xi)^{p/2
}F\big(z,\cdot\big)\big\|_{L^1(\X^*)}\int_\X<u>^{-p}du\\
&\leq C_p\, \nu^{m,0}_{0,p}(F),\quad \forall p>2\,.
\end{align*}
If we have $m<0$ but $\rho\in(0,1]$  we
can use Lemma A.4 in \cite{MPR1}
 (based on Propositions 1.3.3 and 1.3.6 in \cite{ABG}) in order to obtain from
\eqref{F1} the following estimate:
\begin{equation*}
\big\|\mathfrak{W}F_{(\epsilon,\tau)}\big\|_{SH}  \,\leq \, (2\pi)^{-1}\underset{
z\in\X
}{\sup}\big\|\mathcal{F}F\big(z,\cdot\big)\big\|_{L^1(\X)}
\, \leq \,  C_p \, \nu^{m,\rho}_{0,p}(F)\,,
\end{equation*}
for some $p>(2-m)\rho^{-1}\,$.
\end{proof}

\begin{proposition}\label{comp-epsilon-symb}
Let $F\in { S^{m}_\rho(\Xi)}$ and $G\in { S^{p}_\rho(\Xi)}$, with
$m\in\mathbb{R}$, $p\in\mathbb{R}$ and $\rho\in[0,1]$
 and  let $L$ some integral kernel in
$
BC^\infty\big(\X;C^\infty_{\text{\sf pol}}(\X\times\X)\big)$.
Let  $D\subset(0,\infty)\times(0,\infty)$ and $K$  some compact interval in $[0,+\infty)$. For any 
 $(\epsilon,\tau)\in D $, $(\mu_1,\cdots,\mu_4)\in K^4$, let us 
consider the  oscillatory integral:
$$
\begin{array}{l}
\mathcal{L}_{(\mu_1,\mu_2,\mu_3,\mu_4)}^{\epsilon,\tau}\big(F,G\big)(X)\,
\\ \qquad  := \int_{\Xi\times\Xi}e^{-2i\sigma(Y,Z)}
L(\epsilon x,y,z))F\big(\epsilon
x-\mu_1y,\tau\xi-\mu_2\eta\big)G\big(\epsilon x-\mu_3
z,\tau\xi-\mu_4\zeta\big)\,dYdZ\, .
\end{array}
$$
Then there exists a symbol $\T:=\T^{(\mu_1,\mu_2,\mu_3,\mu_4)}$ in
$S^{m+p}_\rho(\X)$,  such that:
$$
\mathcal{L}_{(\mu_1,\mu_2,\mu_3,\mu_4)}^{\epsilon,\tau}\big(F,G\big)\,=\,
\T_{(\epsilon,\tau)}\,.
$$
Moreover, for any  $(m,p,n_1,n_2)$ (with $(n_1,n_2)\in\mathbb{N}\times\mathbb{N}$)\,, there exist some indices $(q_1,p_1,q_2,p_2)$, a map 
$\nu:BC^\infty\big(\X;C^\infty_{\text{\sf
pol}}(\X\times\X)\big)\rightarrow\mathbb{R}_+$,
 and some positive constants
$C_{K}(n_1,n_2,q_1,q_2,p_1,p_2)$, such that:
$$
{\nu^{(m+p),\rho}_{(n_1,n_2)}\left(\mathcal{L}_{(\mu_1,\mu_2,\mu_3,\mu_4)}^{\epsilon,\tau}\big(F,G\big)\right)}\,\leq\,C_{K}(n_1,n_2,q_1,q_2,p_1,
p_2)\,\nu\big(L\big)\,
\nu^{m,\rho}_{(q_1,p_1)}\big(F\big)\,\nu^{p,\rho}_{(q_2,p_2)}\big(G\big),
$$
for any 
 $(\epsilon,\tau)\in D $, $(\mu_1,\cdots,\mu_4)\in K^4$.
\end{proposition}
\begin{proof}
If we define:
\begin{equation}\label{(2)}
\T^{(\mu_1,\mu_2,\mu_3,\mu_4)}(X):=\int\limits_{\Xi\times\Xi}e^{-2i\sigma(Y,
Z)}L(x,y,z))\, F\big(
x-\mu_1y,\xi-\mu_2\eta\big)\, G\big(x-\mu_3
z,\xi-\mu_4\zeta\big)\,dYdZ\,,
\end{equation}
as an oscillatory integral
we have 
$$
\mathcal{L}_{(\mu_1,\mu_2,\mu_3,\mu_4)}^{\epsilon,\tau}\big(F,G\big)\,=\,
\T^{(\mu_1,\mu_2,\mu_3,\mu_4)}_{(\epsilon,\tau)}.
$$
Let us estimate the behavior of the expressions
$<\xi>^{-(m+p)+\rho|\beta|}\big(\partial_x^\alpha\partial_\xi^\beta T\big)(x,
\xi)$ appearing in 
the norms on $ S^{m+p}_\rho(\Xi)$ (we take
$|\alpha|=n_1$, $|\beta|=n_2$). A simple calculus shows that we have to estimate
integrals of the form:
$$
\begin{array}{l}\hspace{-0.5cm}
<\xi>^{-(m+p)+\rho|\beta|}\int_{
\Xi\times\Xi}e^{-2i\sigma(Y,Z)}\big(\partial_x^{\alpha_3}
L\big)(x,y,z) \\
\hspace{4cm} \times\ \big(\partial_x^{\alpha_1} \partial_\xi^{\beta_1}F\big)
\big(x-\mu_1 y,\xi-\mu_2\big)\big(\partial_x^{\alpha_2}
\partial_\xi^{\beta_2}\, G\big)\big(x-\mu_3z,\xi-\mu_4\zeta\big)\,dYdZ\,
\\~\\
\qquad\qquad =\int_{
\Xi\times\Xi}e^{-2i\sigma(Y,Z)}\big(\partial_x^{\alpha_3}L\big)
(x,y,z)<\mu_2\eta>^{m-\rho|\beta_1|}<\mu_4\zeta>^{p-\rho|\beta_2|}\
\\\hspace{3cm}
  \times\ <\xi-\mu_2\eta>^{-m+\rho|\beta_1|}\big(\partial_x^{\alpha_1}
\partial_\xi^{\beta_1}F\big)
\big(x-\mu_1
y,\xi-\mu_2\eta\big) 
\\
\hspace{5cm}   \times<\xi-\mu_4\zeta>^{-p+\rho|\beta_2|}\big(\partial_x^{
\alpha_2 }
\partial_\xi^{\beta_2}G\big)\big(x-\mu_3z,\xi-\mu_4\zeta\big)\, dYdZ\,,
\end{array}
$$
where $|\alpha_1+\alpha_2+\alpha_3|= n_1$ and
$|\beta_1+\beta_2|=n_2\,$.\\
By the usual integration by parts using the
exponential $e^{-2i\sigma(Y,Z)}$, we are reduced to integrals of the form:
$$
\begin{array}{l}
\int\limits_{\Xi\times\Xi}<\eta>^{-s_2/2}<\mu_2\eta>^{m-\rho|\beta_1|}
<\zeta>^{-s_3/2}<\mu_4\zeta>^{p-\rho|\beta_2|}
\big(<y><z>\big)^{-(s_1-r(n_1,s_2,s_3))/2} \\
\hspace{1,5cm}
\times\
e^{
-2i\sigma(Y,Z)}\mu_1^{|\theta_1|}\big((1-(\mu_2^2/2)\Delta_\xi)^{s_1/2}
\partial_x^{
\alpha_1+\theta_1}\partial_\xi^{\beta_1}F\big)
\big(x-\mu_1y,\xi-\mu_2\eta\big) \\
\hspace{4cm}
\times\
\mu_3^{\theta_2|}\big((1-(\mu_4^2/2)\Delta_\xi)^{s_1/2}
\partial_x^{\alpha_2+\theta_2}
\partial_\xi^{\beta_2}G\big)\big(x-\mu_2z\,,
\xi-\mu_4\zeta\big)\\
\hspace{7cm}
\times\,\big(\partial_x^{\alpha_3}\partial_y^{\rho_1}\partial_z^{\rho_2}L\big)(x
,y,z)\,dYdZ\,,
\end{array}
$$
where $s_2>d+[m-\rho|\beta_1|]_+\,$, $s_3>d+[p-\rho|\beta_2|]_+\,$,
$|\theta_1+\rho_1|\leq s_3\,,$ and $|\theta_2+\rho_2|\leq s_2\,$.\\

Then  by the
hypothesis on  $L\,$, for
$|\alpha_3|\leq n_1$ there exists  $r(n_1,s_2,s_3)\in\mathbb{N}$ such that
$$ 
\nu(L)\ :=\ \underset{(x,y,z)\in\X^3}{\sup}\,\underset{|\alpha_3|\leq n_1}{\max}
\underset{|\rho_1|\leq s_3}{\max}\underset{|\rho_2|\leq
s_2}{\max}(<y><z>)^{r(n_1,s_2,s_3)}\left|\big(\partial_x^{\alpha_3}\partial_y^{
\rho_1}\partial_z^{\rho_2}L\big)(x,y,z)\right| \,<\,+\infty\,.
$$
With this last choice, we take $s_1>d+r(n_1,s_2,s_3)$
and the proof is finished. 
\end{proof}
\begin{remark}
The following relations hold true:
$$\hspace{-2.2cm}
\partial_{\xi_j}\mathcal{L}_{(\mu_1,\mu_2,\mu_3,\mu_4)}^{\epsilon,\tau}\big(F,G\big)=\tau\left(\mathcal{L}_{(\mu_1,\mu_2,\mu_3,\mu_4)}^{\epsilon,\tau}\big(\partial_{\xi_j}F,G\big)+\mathcal{L}_{(\mu_1,\mu_2,\mu_3,\mu_4)}^{\epsilon,\tau}\big(F,\partial_{\xi_j}G\big)\right)\,,
$$
\begin{multline*}
\partial_{x_j}\mathcal{L}_{(\mu_1,\mu_2,\mu_3,\mu_4)}^{\epsilon,\tau}\big(F,G\big)
\\
=\epsilon\left(\mathcal{L}_{(\mu_1,\mu_2,\mu_3,\mu_4)}^{\epsilon,\tau}\big(\partial_{x_j}F,G\big)+\mathcal{L}_{(\mu_1,\mu_2,\mu_3,\mu_4)}^{\epsilon,\tau}\big(F,\partial_{x_j}G\big)+\mathcal{L}_{(\mu_1,\mu_2,\mu_3,\mu_4)}^{\epsilon,\tau,(j)}\big(F,G\big)\right)\,,
\end{multline*}
where 
\begin{align*}
\hspace{-1cm} \mathcal{L}_{(\mu_1,\mu_2,\mu_3,\mu_4)}^{\epsilon,\tau,(j)}\big(F,G\big)(X)&\numberthis \label{1}
\\ &\hspace{-2cm}: =\int_{\Xi\times\Xi}e^{-2i\sigma(Y,Z)}
(\partial_{x_j}L)(\epsilon x,y,z))F\big(\epsilon
x-\mu_1y,\tau\xi-\mu_2\eta\big)G\big(\epsilon x-\mu_3
z,\tau\xi-\mu_4\zeta\big)\, dYdZ\, .
\end{align*}
\end{remark}

\begin{corollary}\label{corA3}
For any bounded subsets $$\mathcal{B}_m\subset { S^{m}_\rho(\Xi)},
\;\mathcal{B}_p\subset { S^{p}_\rho(\Xi)},\; 
{\mathcal{B}_L\subset BC^\infty\big(\X;C^\infty_{\text{\sf
pol}}(\X\times\X)\big)}$$ and for any compact set  $K\subset[0,+\infty)$, the set
$$
\big\{\mathcal{L}_{(\mu_1,\mu_2,\mu_3,\mu_4)}^{1,1}\big(F,
G\big)\,\mid\,F\in\mathcal{B}_m,\,G\in\mathcal{B}_p,\,L\in\mathcal{B}_L,\,
\mu_j\in K,j=1,2,3,4
\big\}
$$ 
is bounded in $S^{m+p}_\rho(\X)$.
\end{corollary}
The following statement (a simplified version
of Proposition 8.1 in \cite{IMP2}) is a simple corollary of Proposition
\ref{comp-epsilon-symb}.
\begin{proposition}\label{comp-symb-gen}
 Let $\phi\in S^m_\rho(\Xi)$, $\psi\in S^p_\rho(\Xi)$ and
$\theta\in BC^\infty(\mathcal{X};C^\infty_{\text{\sf pol}}(\mathcal{X}^2))$. Then
$$
X \mapsto \mathfrak{L}(\theta;\phi,\psi)(X):=\int_\Xi dY\int_\Xi dZ\,
e^{ -2i\sigma(Y,Z)}\, \theta(x,y,z)\,\phi(X-Y)\,\psi(X-Z)
$$
defines a symbol of class $S^{m+p}_\rho(\Xi)$ and the
map
$$
S^m_\rho(\Xi)\times
S^p_\rho(\Xi)\ni(\phi,\psi)\mapsto\mathfrak{L}(\theta;\phi,\psi)\in
S^{m+p}_\rho(\Xi)
$$
is continuous. If $\phi$ or $\psi$ belongs to
$\mathscr{S}(\Xi)$ then $\mathfrak{L}(\theta;\phi,\psi)$ also belongs to
$\mathscr{S}(\Xi)$ and  the map
$(\phi,\psi)\mapsto\mathfrak{L}(\theta;\phi,\psi)$ considered in the
corresponding spaces is jointly continuous
with respect to the associated Fr\'{e}chet topologies.
\end{proposition}
\begin{proof}
While the first conclusion is a particular case of Proposition~\ref{comp-epsilon-symb}\,, the second one follows  as in the
proof of Proposition \ref{comp-epsilon-symb}\,, noticing that for any $p>0$ we
have the inequality $$<x>^p\leq C_p<x-y>^p<y>^p\,.$$
\end{proof}

\subsection{Weak magnetic fields}\label{SS-wmf}

We are interested in {\it weak} magnetic
fields,
controlled by a small parameter $\epsilon\in[0,\epsilon_0]$ for some $\epsilon_0 >0$, that we intend to treat as a small  perturbation of the situation without magnetic field. In this subsection
we work under the following hypothesis, which is more general than the situation
considered in \eqref{Bek}, noticing that we can write
 $B_{\epsilon,\kappa}=\epsilon B^0_\epsilon$ for
$B^0_\epsilon(x):=B_0+\kappa B(\epsilon x)$.

\begin{hypothesis}\label{Hyp-V-magn}
The family of magnetic fields
$\{B_\epsilon\}_{\epsilon\in[0,\epsilon_0]}$ has  the form
$
B_\epsilon:=\epsilon\,B^0_\epsilon,
$
with $B^0_\epsilon\in BC^\infty\big(\X\big)$ uniformly with respect to
$\epsilon\in[0,\epsilon_0]$.
\end{hypothesis}

To simplify
the notation, when dealing with weak
magnetic fields, the indexes (or the exponents) $A_\epsilon$ or $B_\epsilon$
shall be replaced by $\epsilon$ and we shall
use the notation $\|\cdot\|_{\epsilon}$ instead of $\|\cdot\|_{B_\epsilon}$.

Let us first consider the difference between the magnetic and the usual Moyal
products, for a weak magnetic field.

\begin{proposition}\label{L.II.1.1}
For $\epsilon\in[0,\epsilon_0]$ there exists a continuous
application
$r_\epsilon:{S^m_\rho(\Xi)}\times
{S^{m^\prime}_\rho(\Xi)}\rightarrow
{S^{m+m^\prime-2\rho}_\rho(\Xi)}$ such that:
\begin{equation}\label{II.1.3}
a\,  \,\sharp^{\epsilon}\, \, b\,=\,a\,  \sharp^{0} \, b\,+\,\epsilon \, 
r_\epsilon(a,b)\,,\qquad \forall(a,b)\in
S^m_\rho(\Xi)\times S^{m^\prime}_\rho(\Xi)\,.
\end{equation}
\end{proposition}
\begin{proof}
Let us introduce the notation:
$$
F_\epsilon(x,y,z):= \epsilon\, (y\wedge z) \, F^\circ_\epsilon(x,y,z)\,,
$$
with 
$$
F^\circ_\epsilon(x,y,z):=\left(\int_{-1/2}^{1/2}ds\int_{-1/2}^s\,dt \, B^0_{\epsilon}
\big(x+sy+tz\big)\right)\,,
$$
and notice that $F^\circ_\epsilon$ belongs
to $BC^\infty\big(\mathbb{R}^6\big)$ uniformly 
with respect to
$\epsilon\in[0,\epsilon_0]\,$. 

Then, using the Taylor expansion at first order for the exponential $\exp\{-4iF_\epsilon\}$, we can write:
\begin{align}\label{II.1.6}
&\frac{1}{\epsilon}(a \,\sharp^{\epsilon}\, b\,-\,
a \, \sharp^{0}\, b) (X)\\
&\quad =\,-\frac{4i}{(2\pi)^{4}\epsilon}\int_{\Xi\times\Xi}e^
{-2i\sigma(Y,Z)}\left(\int_0^1e^{-4itF_{\epsilon}(x,y,z)}
dt\right)F_\epsilon(x,y,z) \, a(X-Y)\,b(X-Z)\,dYdZ\,.\nonumber
\end{align}
Integrating by parts in the variables $\eta\,$, $\zeta\,$, we obtain:
\begin{align}\label{II.1.6.a}
\big[r_\epsilon(a,b)\big](X) \, 
=\,-\frac{4i}{\pi^{4}}\int_{\Xi\times\Xi}\,dY
dZ\, e^
{-2i\sigma(Y,Z)}\left(\int_0^1e^{-4itF_{\epsilon}(x,y,z)}
dt\right) \times \qquad  \\
\quad\quad\qquad  \times F^\circ_{\epsilon}(x,y,z)\, \Big( \partial_\xi a (X-Y)
\wedge
\partial_ \xi b (X-Z) \Big)\,.\nonumber 
\end{align}
The proof can be completed by using Proposition
\ref{comp-symb-gen} with $$\theta(x,y,z)=\left(\int_0^1e^{-4itF_{\epsilon}(x,y,z)}
dt\right)F^\circ_{\epsilon}(x,y,z)\,.$$
\end{proof}

\begin{remark}
Using the $N$'th order Taylor expansion of the
exponential and similar arguments, we obtain that, for any
$N\in\mathbb{N}^*$\,,
\beq\label{iulie5}
a\,  \,\sharp^{\epsilon}\, \, b=a\,  \sharp^{0} \, b+\underset{1\leq k\leq
N-1}{\sum}\epsilon^kc^{(k)}_\epsilon(a,b)+\epsilon^N\rho^{(N)}_\epsilon(a,b)\,,
\eeq
with $  c^{(k)}_\epsilon(a,b)\in S^{m+m^\prime-2k\rho}_1(\Xi)$ and
$\rho^{(N)}_\epsilon(a,b)\in S^{m+m^\prime-2N\rho}_1(\Xi)$ 
uniformly for $\epsilon\in[0,\epsilon_0]\,$.
\end{remark}

With Remark
\ref{Rem-s-adj} in mind, let us consider, for an
elliptic real symbol $h\in S^m_1(\Xi)$ with $m>0\,$,  the self-adjoint
extension of $\mathfrak{Op}^\epsilon(h)$ denoted by
$H^\epsilon$,  whose  domain is given by the magnetic
Sobolev space of order $m>0\,$.

\begin{proposition}\label{rez-dev}
For  $\z\in \rho(H^\epsilon)$,  let $r^\epsilon_\z(h)\in {S^{-m}_1(\Xi)}$
denote
the symbol of \break  $(H^\epsilon-\z)^{-1}$. For any compact subset $K$ of $\mathbb{C}\setminus\sigma(H)$, there exists
$\epsilon_0>0$ such that:
\begin{enumerate}
\item $K\subset\mathbb{C}\setminus\sigma(H^\epsilon)$\,, for $\epsilon\in[0,\epsilon_0]$ .
\item The following expansion is convergent in $\mathcal L (\mathcal H)$ uniformly with respect to 
$(\epsilon,\z)\in[0,\epsilon_0]\times K$:
$$
r^\epsilon_\z(h)\ =\
\underset{n\in\mathbb{N}}{\sum}\epsilon^nr_n(h;\epsilon,\z)\,,\quad
r_0(h;\epsilon,\z)=r^0_\z(h)\, ,\quad r_n(h;\epsilon,\z)\in S^{-(m+2n)}_1(\Xi)\,.
$$
\item The map $K\ni \z\mapsto
r^\epsilon_\z(h)\in
{S^{-m}_1(\Xi)}$ is a ${S^{-m}_1(\Xi)}$-valued  analytic function, uniformly in
$\epsilon\in[0,\epsilon_0]\,$.
\end{enumerate}
\end{proposition} 

\begin{proof}
The first point follows from the spectral stability (see Corollary 1.2 in \cite{IP14} or Theorem
1.4 in \cite{AMP} or Theorem 3.1 in \cite{CP-1} for a more precise result).

For the last two statements we start from the analyticity in norm of the
application
$$\mathbb{C}\setminus\sigma(H^\epsilon)\ni\z\mapsto
(H^\epsilon-\z)^{-1}= \mathfrak{Op}^\epsilon\big(r^\epsilon_\z(h)\big)\in
\mathcal L (\mathcal H).$$
In order to obtain a control for the topology of
${S^{-m}_1(\Xi)}$ we recall Theorem~5.2
in \cite{IMP2}.  Using \eqref{ad-op} one shows that for any
$\epsilon\in[0,\epsilon_0]$ and any $\z\notin\sigma(H^\epsilon)$ 
\beq\label{ad-rez}
\mathfrak{ad}_X^\epsilon[r^\epsilon_\z(h)]
=-r^\epsilon_\z(h)\,\sharp^\epsilon\,\mathfrak{ad}
_X^\epsilon[h]\,\sharp^\epsilon\, r^\epsilon_\z(h)\,.
\eeq
Using the resolvent equation:
$$
r^\epsilon_\z(h)=r^\epsilon_i(h)+(i-\z)r^\epsilon_i(h)r^\epsilon_\z(h)\,,
$$
 and Propositions 3.6 and 3.7 from \cite{IMP2} we easily prove that 
for any pair of natural numbers $(p,q)\in\mathbb{N}\times\mathbb{N}$ and
any families of points $\{u_1,\ldots,u_p\}\subset\X$ and
$\{\mu_1,\ldots,\mu_q\}\subset\X^*$, the applications:
\beq\label{z-cont}
K\ni\z\mapsto\mathfrak{s}^\epsilon_{m+q}\,\sharp^\epsilon\,\big(\mathfrak
{ad}_{u_1}^\epsilon\cdots\mathfrak{ad}_{u_p}^\epsilon\mathfrak{ad}_{\mu_1}
^\epsilon\cdots\mathfrak{ad}_{\mu_q}^\epsilon[r^\epsilon_\z(h)]\big)\in
{S^0_0(\Xi)}
\eeq
are well defined, bounded and uniformly continuous for the norm
$\|\cdot\|_{\epsilon}$ for any $\epsilon\in[0,\epsilon_0]$.

The second point follows by noticing that Proposition \ref{L.II.1.1}
implies the equality
\begin{equation}\label{II.1.9}
1\,=\,(h-\z)\, \sharp^0\, r^0_\z(h)\,=\,(h-\z) \,\sharp^{\epsilon}\, 
r^0_\z(h) - \epsilon\,  r_\epsilon\big(h,r^0_\z(h)\big)\,,
\end{equation}
where the family
$\{r_\epsilon\big(h,r^0_\z(h)\big)\}_{\epsilon\in[0,\epsilon_0]}$ is a
bounded
subset in $S^{-m}_1(\Xi)\,$.  

We conclude that for some $\epsilon_0>0\,$,
$1+\epsilon r_\epsilon(h,r^0_\z(h))$ defines an invertible magnetic
operator for any $\epsilon\in[0,\epsilon_0]$
and its inverse has a symbol $s^\epsilon(\z)$ given as the limit of the
following norm convergent series:
\beq\label{rez-dev-corr}
s^\epsilon(\z):=\underset{n\in\mathbb{N}}{\sum}\big(-\epsilon
r_\epsilon(h,r^0_\z(h))\big)^{\,\sharp^\epsilon\, n}\ \in\
S^0_1(\Xi)\,.
\eeq
This clearly gives us the expansion in point (2) of the theorem with
$$
r_n(h;\epsilon,\z):=(-1)^n\, r^0_\z(h)\,\sharp^\epsilon\,\big(r_\epsilon(h,
r^0_\z(h))\big)^{\,\sharp^\epsilon\, n}\in S^{-(m+2n)}_1(\Xi)\,.
$$

In order to control the uniform continuity  with respect to $\epsilon\in[0,\epsilon_0]$
 of the application in \eqref{z-cont} let us notice that
$$
r^\epsilon_\z(h)-r^\epsilon_{\z'}
(h)=(\z'-\z)\, r^\epsilon_\z(h)\,\sharp^\epsilon\, r^\epsilon_{\z'}(h)\,,
$$
and that for any $\z\in K$ the family of symbols
$\{r^\epsilon_\z(h)\}_{\epsilon\in[0,\epsilon_0]}$ is a bounded set in
$S^{-m}_1(\Xi)$ due to the uniform convergence of the series
expansion
obtained at point (2).
\end{proof}

Associated with the series expansion of the symbol $r^\epsilon_\z(h)$ given in
Proposition \ref{rez-dev}\,, we shall also use the notation 
\beq\label{rez-dev-1}\widetilde{r^\epsilon}_{\z,n}(h):=\underset{n+1\leq
k}{\sum}\epsilon^kr_k(h;\epsilon,\z)\in  S^{-(m+2n+2)}_1(\Xi)\,.
\eeq

\begin{remark}\label{rem-est-rest}
Having in mind \eqref{II.1.3} we
conclude that there exist $\epsilon_0 >0$ and $C>0$ such that, for $\epsilon \in [0, \epsilon_0]\,$,  the remainder $\widetilde{r^\epsilon}_{\z,n}$ has the
following properties:
\begin{enumerate}
 \item
$\widetilde{r^\epsilon}_{\z,n}\,=\,\epsilon^{n+1}\, \widetilde{\widetilde{
r^\epsilon}}_{\z,n}\,$, where $\widetilde{\widetilde{
r^\epsilon}}_{\z,n}\in {S^{-m}_1(\Xi)}_\Gamma$ uniformly for 
$\epsilon\in[0,\epsilon_0)\,$.
\item
$\|h\,\sharp^\epsilon\,\widetilde{r^\epsilon}_{\z,n}\|_{\epsilon}\leq
C\, \epsilon^{n+1}\,.
$
\end{enumerate}
\end{remark}

\subsection{The slowly varying magnetic fields}
\label{SS-sl-var-m-field}

\subsubsection{Some properties of the pseudodifferential calculus in a constant magnetic field.}\label{SS-const-m-field}

In this case, for symbols independent of $x$, some interesting particularities have to be pointed out. One can also mention in this context the pseudo-calculus developed for other purposes in 
 \cite{BGH}  which is applied in the magnetic context in  Helffer-Sj\"ostrand \cite{HS4}.
 
Considering a constant magnetic field $B_0\,$, we notice that 
$\omega^{B^0}(x,y,z)=\exp\{-2iB^0(y\wedge z)\}$ and thus for $\phi$ and
$\psi$  in $\mathcal S$, we can write
\begin{align*}
\big(\phi  \,\sharp^{B_0} \,  \psi\big)(x,\xi)&= \numberthis \label{17Sept} \\
&\hspace{-2cm}=\pi^{-4}\int_{\X*}\int_{\X*}\left(\int_\X\int_\X
e^{-2i<\eta,z>}e^{2i<\zeta,y>}e^{-2iB^0(y\wedge z)}\,d^2y\,
d^2z\right)\phi(\xi-\eta)\psi(\xi-\zeta)\,d^2\eta\,d^2\zeta\,.
\end{align*}
We may compute the Fourier transform of $e^{-2iB^0(y\wedge z)}$ in $\mathcal S'$ using Theorem 7.6.1 in \cite{Ho1} and get that, in the sense of tempered distributions
$$
\int_\X\int_\X
e^{-2i<\eta,z>}e^{2i<\zeta,y>}e^{-2iB^0(y\wedge z)}\,d^2y\,d^2z=\left(\frac{\pi}{2B^0}\right)^2e^{
(i/B^0)(\eta \wedge \zeta)}.
$$
Thus we conclude that
\beq \label{mag-comp}
\begin{array}{ll}
\big(\phi\,\sharp^{B_0}
\,\psi\big)(x,\xi)& =\left(\frac{1}{2B^0\pi}\right)^2\int_{\X^*
} \int_ {\X^*}
e^{(i/B^0)(\eta\wedge \zeta)}\phi(\xi-\eta)\psi(\xi-\zeta)\,d^2\eta\,
d^2\zeta\\
&
=\left(\frac{1}{2\pi}\right)^2\int_{\X^*}\int_{\X^*}
e^{i(\eta\wedge \zeta))}\phi(\xi-\sqrt{B^0}\eta)\psi(\xi-\sqrt{B^0}\zeta)\,
d^2\eta\,d^2\zeta\,.
\end{array}
\eeq
Using now a Taylor expansion of some order $N\in\mathbb{N}^*$ in
\eqref{17Sept}, we obtain
\beq\label{f-B^01}
\begin{array}{ll}
\big(\phi\,\sharp^{B_0} \,\psi\big)(x,\xi)& =\phi(\xi)\psi(\xi)
\\ &  \qquad \qquad  +\underset{1\leq
p\leq
N-1}{\sum}(-2i)^p(B^0)^{p}\underset{|\alpha|=p}{\sum}(\alpha!)^{-1}
\big(\partial^\alpha_\xi\phi\big)(\xi)\big(\partial^\alpha_{\xi^\bot}
\psi\big)(\xi)\\
& \qquad  \qquad\qquad  \qquad \qquad +\, (2\pi)^{-2}(-2iB^0)^N r_N( \phi,\psi, B_0)(\xi) \,,
\end{array}
\eeq 
with
\begin{align*}
r_N(\phi,\psi,B_0) (\xi)&\numberthis\label{reste1}
\\ &\hspace{-2cm}\quad :=
\underset{|\alpha|=N}{\sum}
(\alpha!)^{-1}\hspace*{-10pt}\int\limits_{\X^*\times\X^*}
\hspace*{-10pt}e^{i\eta\wedge \zeta} \left(\int\limits_0^1(1-s)^{N-1}\,
\big(\partial^\alpha_\xi\phi\big)(\xi-\sqrt
{sB^0}\eta) \big(\partial^\alpha_{\xi^\bot}\psi\big)(\xi-\sqrt{sB^0}\zeta)\,
ds \right)\, d^2\eta\, d^2\zeta \,.
\end{align*}
In writing the remainder $r_N$ we have used \eqref{mag-comp} with $B^0$ replaced by $sB^0$
(with $s\in[0,1]$)\,.

{From  these formulas and the results in \cite{IMP2} concerning the
functional calculus for
a magnetic pseudodifferential self-adjoint operator in $L^2(\X)$, we directly deduce the following statement.
\begin{proposition}\label{const-X}
For a constant magnetic field,  the subspace 
generated by tempered distributions which are  constant with respect to  $\X$  is stable for the magnetic Moyal product. \\
Moreover,  if a self-adjoint magnetic pseudodifferential
operator $H$ in $L^2(\X)$ has a  constant symbol with respect to $\X$,
then all the operators $f(H)$  obtained by functional calculus  have the same property.
\end{proposition}
Motivated by the above results we shall introduce the following classes of
symbols which are  constant along the directions in $\X$:
\begin{equation}\label{S-circ}
{S^{m}_\rho(\Xi)}^\circ\,:=\,{S^{m}_\rho(\Xi)}\cap\left(\mathbb C \otimes
\mathscr{S}^\prime(\X^*)\right)\,.
\end{equation}

\subsubsection{The class of slowly varying  symbols.}
\label{A-cl-sl-var}
The interest in working with slowly varying magnetic fields comes from
the fact that their derivatives of order $k\in\mathbb{N}$ produce a factor
$\epsilon^k$. In order to systematically keep track of this property it will be
useful to consider the following
class of $\epsilon$-indexed families of symbols, replacing  the classes
${S^{m}_\rho(\Xi)}^\circ$  introduced in \eqref{S-circ}. 
Let us recall the definition of {\it a scaled symbol}
in \eqref{scaled-symb}; here we shall
consider the situation $\tau=1$ and $\epsilon\searrow0$.

\begin{definition}\label{S-epsilon}
For any $(m,\rho)\in\mathbb{R}\times[0,1]$ and for some $\epsilon_0>0$, we  denote by  ${S^{m}_\rho(\Xi)}^\bullet$ the  families of symbols $\{F^\epsilon\}_{\epsilon\in[0,\epsilon_0]}$
satisfying the following properties:
\begin{enumerate}
 \item $F^\epsilon\in
{S^{m}_\rho(\Xi)}\,, \forall\epsilon\in[0,\epsilon_0]$\,;
\item $\exists\underset{\epsilon\searrow0}{\lim}F^\epsilon := F^0\in
{S^{m}_\rho(\Xi)}^\circ$ in the topology of ${S^{m}_\rho(\Xi)}$\,;
\item $\forall(\alpha,\beta)\in\mathbb{N}^2\times\mathbb{N}^2,\ \exists
C_{\alpha\beta}>0$ such that
\beq\label{est-epsilon-symb}
\underset{\epsilon\in(0,\epsilon_0]}
{\sup}\,\epsilon^{-|\alpha|}
\left\|\partial^\alpha_x\partial^\beta_\xi
F^\epsilon\right\|_\infty\leq C_{\alpha\beta}\,.
\eeq
\end{enumerate}
${S^{m}_\rho(\Xi)}^\bullet$ is endowed with the topology defined by the
seminorms indexed by $(p,q)\in \mathbb N^2$\,,
$$ 
F^\bullet \mapsto \tilde{\nu}^{m,\rho}_{p,q}(F^\bullet)\,:=\,\underset{\epsilon\in[0,
\epsilon_0]}{\sup}\,\epsilon^{-p}\underset{|\alpha|=p}{\sum}\,\underset{|\beta|=q
}{\sum}\underset{(x,\xi)\in\Xi}{\sup}<\xi>^{-(m-q\rho)}
\left|\big(\partial_x^\alpha\partial_\xi^\beta F^\epsilon\big)(x,\xi)\right|\,.
$$
\end{definition}

\begin{remark}\label{rem-sl-var-class}
Defining $\widetilde{F}^\epsilon(x,\xi):=F^\epsilon(\epsilon^{-1}x,\xi)$, it is
easy to see that $\{F^\epsilon\}_{\epsilon\in[0,\epsilon_0]}\subset
{S^{m}_\rho(\Xi)}$ belongs to ${S^{m}_\rho(\Xi)}^\bullet$ if and only if it is
of the form
$
F^\epsilon(x,\xi) = \widetilde{F}^\epsilon(\epsilon
x,\xi)=\widetilde{F}^\epsilon_{(\epsilon,1)}(x,\xi)
$
for some bounded family
$\{\widetilde{F}^\epsilon\}_{\epsilon\in[0,\epsilon_0]}\subset
{S^{m}_\rho(\Xi)}$
verifying the condition (for the  topology on
${S^{m}_\rho(\Xi)}^\circ$)
$$
\exists
\underset{\epsilon\searrow0}{\lim}\widetilde{F}^\epsilon(0,\cdot) := F^0\in
{S^{m}_\rho(\Xi)}^\circ.
$$
\end{remark}
Let us notice that
for a slowly varying magnetic field of the form \eqref{Bek} we have:
\beq\label{sv-magnetic-phase}
\omega^{\kappa\epsilon B_{\epsilon}}(x,y,z)\,=\,\exp\left\{
-4i\kappa\epsilon(y\wedge z)\int_{-1/2}^{1/2}ds\int_{-1/2}^sdt\,B
\big(\epsilon
x+\epsilon(2sy+2tz)\big)\right\}\,.
\eeq
{\sloppy
\begin{proposition}\label{epsilon-comp-symb}  Let $B_{\epsilon,\kappa}(x)$  be a magnetic field
of the form \eqref{Bek}. If $f^\bullet\in {S^{m}_\rho(\Xi)}^\bullet$ and $g^\bullet\in
{S^{p}_\rho(\Xi)}^\bullet$, then
$\{f^\epsilon \, \sharp^{B_{\epsilon,\kappa}} g^\epsilon\}_
{\epsilon\in[0,\epsilon_0]}$
belongs to $S^{m+p}_\rho(\Xi)^\bullet$ uniformly with respect to
$\kappa\in[0,1]\,$.
\end{proposition}}

\fussy
\begin{proof}~
Using Remark \ref{rem-sl-var-class}, \eqref{Bek} and \eqref{sv-magnetic-phase},  we have, $\forall\epsilon\in[0,\epsilon_0]$\,, 
$$
f^\epsilon \,  \sharp^{B_{\epsilon,\kappa}} g^\epsilon\ =\
\widetilde{f}^\epsilon_{(\epsilon,1)}\,\sharp^{\kappa\epsilon B_{\epsilon}}\,
\widetilde{g}^\epsilon_{(\epsilon,1)}\,,
$$
and 
\begin{align*}
&\hspace{-1cm}\Big(\widetilde{f}^\epsilon_{(\epsilon,1)}\,\sharp^{\kappa\epsilon
B_{\epsilon}}\,
\widetilde{g}^\epsilon_{(\epsilon,1)}\Big)(x,\xi)\,\\ 
& =\,\pi^{-4}\int\limits_{\Xi\times\Xi}e^{-2i\sigma(Y,Z)}
\omega^{\kappa\epsilon
B_\epsilon(x,y,z)}\widetilde{f}^\epsilon\big(\epsilon(x-y),\xi-\eta\big)
\widetilde{g}^\epsilon\big(\epsilon(x-z),\xi-\zeta\big)\,dYdZ\,\\
& =\,\pi^{-4}\int\limits_{\Xi\times\Xi}e^{-2i\sigma(Y,Z)}
\exp\left\{
-4i\epsilon(y\wedge z)\left(B_{0}+\kappa\int\limits_{-1/2}^{1/2}
ds\int\limits_{-1/2}^sdt
\,B\big(\epsilon
x+\epsilon(2sy+2tz)\big)\right)\right\}
\\  &\hspace{6cm}
\times\,\widetilde{f}^\epsilon\big(\epsilon x-\epsilon y,\xi-\eta\big)\, 
\widetilde{g}^\epsilon\big(\epsilon x-\epsilon z,\xi-\zeta\big)\,dYdZ \,.
\end{align*}
Using Proposition \ref{comp-epsilon-symb}  with
$$
L^{\epsilon,\kappa}(x,y,z)\ :=\ \exp\left\{
-4i\epsilon(y\wedge z)\left(B_{0}+\kappa\int\limits_{-1/2}^{1/2}
ds\int\limits_{-1/2}^sdt
\,B\big(
x+\epsilon(2sy+2tz)\big)\right)\right\}\,,
$$ 
we obtain:
$$
f^\epsilon \sharp^{B_{\epsilon,\kappa}} g^\epsilon\ =\
\widetilde{f}^\epsilon_{(\epsilon,1)}\,\sharp^{\kappa\epsilon B_{\epsilon}}\,
\widetilde{g}^\epsilon_{(\epsilon,1)}\ =\
\Big(\mathfrak{L}^{\epsilon,\kappa}_{(\epsilon,1,\epsilon,1)}\big(\widetilde{f}
^\epsilon,\widetilde{g}^\epsilon\big)\Big)_{(\epsilon,1)}\,,
$$
where 
$\big\{\mathfrak{L}^{\epsilon,\kappa}_{(\epsilon,1,\epsilon,1)}\big(\widetilde{f
}^\epsilon,\widetilde{g}^\epsilon\big)\big\}_{(\epsilon,\kappa)\in[0,\epsilon_0]
\times[0,1]}$ is a bounded subset of $S^{m+p}_\rho(\Xi)\,$.

Using Remark \ref{rem-sl-var-class} we only have  to compute 
$
\underset{\epsilon\searrow0}{\lim}\, \mathfrak{L}^{\epsilon,\kappa}_{(\epsilon,1,
\epsilon,1)}\big(\widetilde{f}^\epsilon,\widetilde{g}^\epsilon\big)(x,\xi)
\,$. By hypothesis we know that
$\exists\underset{\epsilon\searrow0}{\lim}\, \widetilde{f}^\epsilon(0,
\xi)=\widetilde{f}_0(\xi)$ and
$\exists\underset{\epsilon\searrow0}{\lim}\, \widetilde{g}^\epsilon(0,
\xi)=\widetilde{g}_0(\xi)$ with respect to the topology of each
associated symbol class. It is also easy to notice that,  in  $BC^\infty\big(\X;C^\infty_{\text{\sf
pol}}(\X\times\X)\big)\,$,
$$
\exists\underset{\epsilon\searrow0}{\lim}\, L^{\epsilon,\kappa}(0,y,z)=1\,.
$$
Thus taking into account Proposition
\ref{comp-epsilon-symb} we obtain:
$$
\exists\underset{\epsilon\searrow0}{\lim}\, \mathfrak{L}^{\epsilon,\kappa}_{
(\epsilon,1,\epsilon,1)}\big(\widetilde{f}^\epsilon,\widetilde{g}
^\epsilon\big)(x,\xi)\,=\,\widetilde{f}_0(\xi)\, \widetilde{g}_0(\xi)\,.
$$
\end{proof}

The following proposition will be useful in working with inverses of slowly varying symbols.
\begin{proposition}\label{inv-cont}
For given $\epsilon_0 >0\,$, $(m,\rho)\in\mathbb{R}_+\times[0,1]$,  let $\{F_\epsilon\}_{\epsilon\in[0,\epsilon_0]}$ be a family in $ {S^{m}_\rho(\Xi)}$ which   is continuous at $0\,$. Let us consider a family
of magnetic fields $B_\epsilon=\epsilon B^0_\epsilon $ satisfying Hypothesis
\ref{Hyp-V-magn} and suppose that each symbol $F_\epsilon$, for
$\epsilon\in[0,\epsilon_0]$ has an inverse
$F^-_\epsilon:= \big(F_\epsilon\big)^-_{B_\epsilon}$ with respect to the
magnetic Moyal product  $\sharp^{\epsilon}$. If moreover
$\mathfrak{s}_m\,\sharp^\epsilon\, F^-_\epsilon$ defines a bounded operator in
$\mathcal{H}$ and  the family $\{F^-_\epsilon\}_{\epsilon\in[0,\epsilon_0]}$
is bounded in $S^{-m}_\rho(\Xi)\,$, then  $F^-_\epsilon$ is continuous at $0\,$. 
\end{proposition}
\begin{proof}
By Propositions 6.1 and 6.2 in \cite{IMP2} each $F^-_\epsilon$
belongs to $S^{-m}_\rho(\Xi)\,$. Due to the continuity at $ \epsilon=0$ of the family $\{F_\epsilon\}_{\epsilon\in[0,\epsilon_0]}$ we conclude that for $\epsilon>0$ small enough, the inverse $\big(F_0\big)^-_{B_\epsilon}$ still exists.
For any seminorm $\nu:S^{-m}_\rho(\Xi)\rightarrow\mathbb{R}_+$
defining its  topology, we can write:
\beq\label{norm-diff-inv}
\nu\big[\big(F_\epsilon\big)^-_{B_\epsilon}-\big(F_0\big)^-_0\big]\leq
\nu\big[\big(F_\epsilon\big)^-_{B_\epsilon}-\big(F_0\big)^-_{B_\epsilon}\big]\
+\nu\big[\big(F_0)^-_{B_\epsilon}-\big(F_0\big)^-_0\big]\,.
\eeq
For the first seminorm a simple usual computation shows that for any bounded
smooth magnetic field $B$
\beq\label{diff-inverses}
\big(F_\epsilon\big)^-_{B}-\big(F_0\big)^-_{B}\,=\,-\big(F_\epsilon\big)^-_{B}
\,\sharp^B\,\big[F_\epsilon-F_0\big] \,\sharp^B\,\big(F_0\big)^-_{B}\,.
\eeq
Theorem 5.2
in \cite{IMP2} (Beals like criterion)  states that the 
topology on 
space ${S^{-m}_1(\Xi)}$ (for any $m\in\mathbb{R}$) may be also defined by the
following equivalent family of seminorms:
\beq\label{equiv-seminorms}
 S^{-m}_\rho(\Xi)\ni\psi\mapsto\left\|\mathfrak{s}^\epsilon_{m+q\rho}\,
\sharp^\epsilon\, \big(\mathfrak
{ad}_{u_1}^\epsilon\cdots\mathfrak{ad}_{u_p}^\epsilon\mathfrak{ad}_{\mu_1}
^\epsilon\cdots\mathfrak{ad}_{\mu_q}^\epsilon[\psi]\big)\right\|_{\epsilon}\,,
\eeq
indexed by a pair of natural numbers $(p,q)\in\mathbb{N}\times\mathbb{N}$ and
by two families of points $\{u_1,\ldots,u_p\}\subset\X$ and
$\{\mu_1,\ldots,\mu_q\}\subset\X^*$. A simple computation using
\eqref{diff-inverses} shows that for any magnetic field $B$
\begin{equation*}
\begin{array}{ll}
\mathfrak{ad}_X^B\left[\big(F_\epsilon\big)^-_{B}-\big(F_0\big)^-_{B}\right]
\ & =\,  -\big(F_\epsilon\big)^-_{B}
\,\sharp^B\,\mathfrak{ad}_X^B\big[F_\epsilon-F_0\big]
\,\sharp^B\,\big(F_0\big)^-_{B}\,\\ 
& 
\quad -\, \mathfrak{ad}_X^B\left[\big(F_\epsilon\big)^-_{B}\right]
\,\sharp^B\,\big[F_\epsilon-F_0\big] \,\sharp^B\,\big(F_0\big)^-_{B}\,\\ 
& \quad -\,
\big(F_\epsilon\big)^-_{B} \,\sharp^B\,\big[F_\epsilon-F_0\big] \,\sharp^B\,
\mathfrak{ad}_X^B\left[\big(F_0\big)^-_{B}\right]\  \\
& 
=\, -\big(F_\epsilon\big)^-_{B}
\,\sharp^B\,\mathfrak{ad}_X^B\big[F_\epsilon-F_0\big]
\,\sharp^B\,\big(F_0\big)^-_{B}\,\\ & 
\quad  + \big(F_\epsilon\big)^-_{B} \,\sharp^B\,\mathfrak{ad}_X^B\left[F_\epsilon\right]
\,\sharp^B\,
\big(F_\epsilon\big)^-_{B} \,\sharp^B\,\big[F_\epsilon-F_0\big]
\,\sharp^B\,\big(F_0\big)^-_{B}\,\\
& \quad +\,
\big(F_\epsilon\big)^-_{B} \,\sharp^B\,\big[F_\epsilon-F_0\big] \,\sharp^B\,
\big(F_0\big)^-_{B} \,\sharp^B\,\mathfrak{ad}_X^B
\left[F_0\right] \,\sharp^B\,\big(F_0\big)^-_{B}\,.
\end{array}
\end{equation*}
Iterating the above computation and using Theorem 5.2 in \cite{IMP2} once again,
we prove that
$$
\underset{\epsilon\searrow 0}{\lim}\,
\nu\big[\big(F_\epsilon\big)^-_{B_\epsilon}-\big(F_0\big)^-_{B_\epsilon}\big]\,
=\,0\,.
$$

For the second seminorm in \eqref{norm-diff-inv} we use  Proposition
\ref{L.II.1.1} above with $a=F_0$ and $b=\big(F_0\big)^-_0$ in order to obtain
$$
F_0\, \sharp^{B_\epsilon}\big(F_0\big)^-_0\,=\,1\,+\,\epsilon\, r_\epsilon\big(
F_0,\big(F_0\big)^-_0\big)\,,
$$
and finally
$$
\underset{\epsilon\searrow 0}{\lim}\,
\nu\big[\big(F_0)^-_{B_\epsilon}-\big(F_0\big)^-_0\big]\,=\,0\,.
$$
\end{proof}

{\sloppy
\begin{proposition}\label{epsilon-comp-symb-inv}  Let $B_{\epsilon,\kappa}(x)$  be a magnetic field
of the form \eqref{Bek}. If  ${f^\bullet\in {S^{m}_\rho(\Xi)}^\bullet}$  and  if 
the inverse $(f^\epsilon)^-\equiv(f^\epsilon)^-_{B_{\epsilon,\kappa}}\in
S^{-m}_\rho(\Xi)$ exists for every $\epsilon\in[0,\epsilon_0]$, then
$\{(f^\epsilon)^-\}_{\epsilon\in[0,\epsilon_0]}\in
S^{-m}_\rho(\Xi)^\bullet$.
\end{proposition}}

\fussy
\begin{proof}
The first condition in Definition \ref{S-epsilon} is verified
by  hypothesis and the second follows from Proposition \ref{inv-cont}\,.
 In order to prove the third condition in Definition \ref{S-epsilon} we
recall from Subsection~3.3 in \cite{IMP2} the subspace $\mathfrak{A}(\Xi)$ of symbols in 
 $BC^\infty\big(\X;L^1(\X^*)\big)$ having rapid decay in $\xi\in\X^*$
together with all their derivatives with respect to $x\in\X$\,,  and
the fact that using the operators \eqref{ad-op} we can write
\begin{equation*}
\begin{array}{ll}
\partial_{x_j}f^\epsilon\, & =\,i\mathfrak{ad}^{B_{\epsilon,\kappa}}_{
e_j}[
f^\epsilon] +i\delta^{B_{
\epsilon,\kappa}}f\,\\ &  =\,i\mathfrak{ad}^{B_{\epsilon,\kappa}}_{e_j}[
f^\epsilon]
+i\underset{
1\leq|\alpha|\leq5}{\sum}c^{B_{\epsilon,\kappa}}_{j,\alpha}
\star\big(\partial_\xi^\alpha f^\epsilon\big) \\
& =\,i\mathfrak{ad}^{B_{\epsilon,\kappa}}_{e_j}[
f^\epsilon]
+i\underset{
1\leq|\alpha|\leq5}{\sum}c^{B_{\epsilon,\kappa}}_{j,\alpha}
\star\big((\mathfrak{ad}^{B_{\epsilon,\kappa}}_{e^*})^\alpha f^\epsilon\big)\,,
\end{array}
\end{equation*}
where $c^{B_{\epsilon,\kappa}}_{j,\alpha}\in\mathfrak{A}(\Xi)$ for
any $(j,\alpha)\in\{1,2\}\times\mathbb{N}^2$ and, for $\phi$ and $\psi$ in $BC^\infty\big(\X;L^1(\X^*)\big)$\,, 
$$
\big(\phi\star \psi\big)(x,\xi)\,:=\int_{\X^*}\phi(x,\xi-\eta) \, \psi(x,\eta)\,d\eta\,.
$$

Then, the explicit description of the coefficients
$c^{B_{\epsilon,\kappa}}_{j,\alpha}\in\mathfrak{A}(\Xi)$ given in Propositions
3.6 and 3.7 in \cite{IMP2} easily implies that 
$$
\partial^\beta_x c^{
B_{\epsilon,\kappa}}_{j,\gamma}\ =\  c^{\partial^\beta
B_{\epsilon,\kappa}}_{j,\gamma}\ =\ \epsilon^{|\beta|+1}\,
\widetilde{c}^{
\epsilon,\kappa}_{j,\gamma,\beta}\,,
$$
where $\{\widetilde{c}^{
\epsilon,\kappa}_{j,\gamma,\beta}\}_{(\epsilon,\kappa)\in[0,\epsilon_0]\times[0,
1]}$ is  a bounded set in $\mathfrak{A}(\Xi)\,$. 

By Theorem 5.2
in \cite{IMP2} we know that the topology on any 
space ${S^{m}_\rho(\Xi)}$ may be also defined by the
 family of seminorms:
\beq\label{equiv-seminormsa}
{S^{m}_\rho(\Xi)}\ni\psi\mapsto\left\|\mathfrak{s}^{B_{\epsilon,\kappa}}_{
-(m-q\rho)}
\,\,\sharp^{\epsilon,\kappa}\,\,\big(\mathfrak
{ad}^{B_{\epsilon,\kappa}}_{u_1}\cdots\mathfrak{ad}^{B_{\epsilon,\kappa}}_{u_p
} \mathfrak{ad}^{B_{\epsilon,\kappa}}_{ \mu_1}
\cdots\mathfrak{ad}^{B_{\epsilon,\kappa}}_{\mu_q}[\psi]\big)\right\|_{B_{
\epsilon,\kappa}}\,,
\eeq
which is indexed by  $(p,q)\in\mathbb{N}\times\mathbb{N}\,$,  $\{u_1,\ldots,u_p\}\subset\X$ and
$\{\mu_1,\ldots,\mu_q\}\subset\X^*$. \\
Moreover, we can repeat in this situation
the argument in the proof of Proposition 6.1 in \cite{IMP2} and note that, with
$$\mathfrak{D}^{\epsilon,\kappa}_{\{j_1,\ldots,j_r\}}:=\mathfrak{ad}^{\epsilon,
\kappa}_{X_{j_1}}\dots\mathfrak{ad}^{\epsilon,\kappa}_{X_{j_r}}\,,
$$
and for  coefficients $C_{J_1,\dots,J_r}$ taking only the values $\pm 1$, we have:
$$
\begin{array}{l}\hspace{-1cm}
\mathfrak{s}^{B_{\epsilon,\kappa}}_{-(m-q\rho)}
\,\,\sharp^{\epsilon,\kappa}\,\,\big(\mathfrak
{ad}^{B_{\epsilon,\kappa}}_{u_1}\cdots\mathfrak{ad}^{B_{\epsilon,\kappa}}_{u_p}
\mathfrak{ad}^{B_{\epsilon,\kappa}}_{ \mu_1}
\cdots\mathfrak{ad}^{B_{\epsilon,\kappa}}_{\mu_q}\big[(f^\epsilon)^-\big]\big) \\
\qquad :=\,
\mathfrak{s}^{B_{\epsilon,\kappa}}_{-(m-q\rho)}\,\,\sharp^{\epsilon,\kappa}\,\,\mathfrak
{D}^{\epsilon,\kappa }_{\{1,\ldots,p+q\}}\big[(f^\epsilon)^-\big]\,\\
 \qquad 
\, =\,\underset{1\leq r\leq p+q}{\sum}\,\underset{J_1\sqcup\ldots\sqcup
J_r}{\sum}C_{J_1,\ldots,J_r}\,\mathfrak{s}^{B_{\epsilon,\kappa}}_{-(m-q\rho)}\,\,\sharp^{\epsilon,\kappa}\,\,
(f^\epsilon)^-\,\,\sharp^{\epsilon,\kappa}\,\,\mathfrak{D}^{\epsilon,\kappa}_{J_1}\big[
f^\epsilon\big]\,\,\sharp^{\epsilon,\kappa}\,\,
(f^\epsilon)^-\,\,\sharp^{\epsilon,\kappa}\,\,
\ldots\, \\
\hspace{7cm} \ldots \,\sharp^{\epsilon,\kappa}\,\,(f^\epsilon)^-\,\,\sharp^{\epsilon,\kappa}\,\,\mathfrak{D}
^{\epsilon,\kappa}_{J_r}\big[f^\epsilon\big]\,,
\end{array} 
$$
the sum being over all partitions $\{1,\ldots,p+q\}=\sqcup_{i=1}^pJ_i$ where,
for example, the partition $(J_1,J_2)$ is considered different
from $(J_2,J_1)\,$.\\
These remarks allow us to replace the condition \eqref{est-epsilon-symb} in
Definition \ref{S-epsilon} with the following condition:
$$
\underset{\epsilon\in(0,\epsilon_0]}
{\sup}\,\epsilon^{-q}
\left\|\mathfrak{s}^{B_{\epsilon,\kappa}}_{-(m-q\rho)}
\,\,\sharp^{\epsilon,\kappa}\,\,\big(\mathfrak
{ad}^{B_{\epsilon,\kappa}}_{u_1}\cdots\mathfrak{ad}^{B_{\epsilon,\kappa}}_{u_p
} \mathfrak{ad}^{B_{\epsilon,\kappa}}_{ \mu_1}
\cdots\mathfrak{ad}^{B_{\epsilon,\kappa}}_{\mu_q}[\psi]\big)\right\|_{B_{
\epsilon,\kappa}}< + \infty\,,
$$
for any pair of natural numbers $(p,q)\in\mathbb{N}\times\mathbb{N}$ and
any two families of points $\{u_1,\ldots,u_p\}\subset\X$ and
$\{\mu_1,\ldots,\mu_q\}\subset\X^*\,$.

Combining these remarks and the identity \eqref{ad-rez} achieves the proof.
\end{proof}

\paragraph{A 'localized' approximant for the magnetic Moyal product with slowly varying magnetic field.}

Starting from \eqref{sv-magnetic-phase} and using the Taylor formula with integral remainder one gets:
\begin{equation*}
\begin{array}{l}\hspace{-1cm}
 \int_{-1/2}^{1/2}ds\int_{-1/2}^sdt\,B
\big(\epsilon
x+\epsilon(2sy+2tz)\big)\,
\\ \quad \quad  \quad =\,\frac 12  B(\epsilon x) \,
+ 2  \epsilon    \underset{1\leq
\ell \leq2}{\sum}\left[y_\ell R_1\big(\partial_\ell B\big)(\epsilon x,\epsilon
y,\epsilon z)\,+\, 
z_\ell R_2\big(\partial_\ell B\big)(\epsilon x,\epsilon y,\epsilon z)\right]\,,
\end{array}
\end{equation*}
with
\beq\label{def-R1}
R_1\big(F)(x,y,z)\,:=\,\int_{-1/2}^{1/2}s\, ds\int_{-1/2}^sdt\int_0^1du\,F
\big(x+u(2sy+2tz)\big)\,,
\eeq
\beq\label{def-R2}
R_2\big(F)(x,y,z)\,:=\,\int_{-1/2}^{1/2}ds\int_{-1/2}^s t\,dt\int_0^1du\,F
\big(x+u(2sy+2tz)\big)\,.
\eeq
If we introduce 
\beq\label{def-Psi} 
\Psi^\epsilon(x,y,z)\,:=
-\,8 (y\wedge z)\underset{1\leq
\ell \leq2}{\sum}\left[y_\ell R_1\big(\partial_\ell B\big)(x,\epsilon
y,\epsilon z)\,+\,
z_\ell R_2\big(\partial_\ell B\big)(x,\epsilon y,\epsilon z)\right],
\eeq
we notice that  $\Psi^\epsilon$ is real and  we can write:
\begin{align*}
\omega^{\kappa\epsilon B_{\epsilon}}(x,y,z)
=\exp\left\{-2i\kappa\epsilon B(\epsilon
x)(y\wedge z)\right\}\left(1\,+\,i\, \kappa\epsilon^2\Psi^\epsilon
(\epsilon x,y,z)\int_0^1\exp\big\{i\, \tau\kappa\epsilon^2\Psi^\epsilon
(\epsilon x,y,z)\big\}d\tau\right)\,.
\end{align*}

\begin{remark}\label{R-def-Theta}
Having in mind Proposition
\ref{comp-epsilon-symb}, we notice that the function
$$
\Theta^{\epsilon,\kappa}_\tau(x,y,z):=\exp\big\{ i\, 
\tau\kappa\epsilon^2\Psi^\epsilon
(x,y,z)\big\}
$$
\sloppy
has modulus one and belongs to the space 
$BC^\infty\big(\mathbb{R}^2;C^\infty_{\text{\sf
pol}}(\mathbb{R}^2\times\mathbb{R}^2)\big)$ uniformly for 
$(\kappa,\epsilon)\in[0,1]\times[0,\epsilon_0]$ and the functions
$R_j\big(\partial_\ell B_{jk}\big)(x,y,z)$ (for $j=1,2$) are also of class
${BC^\infty\big(\mathbb{R}^2;C^\infty_{\text{\sf
pol}}(\mathbb{R}^2\times\mathbb{R}^2)\big)}$ uniformly in
$\epsilon\in[0,\epsilon_0]\,$.
\end{remark}
\fussy

It is clear that the space of symbols
$\underset{m\in\mathbb{R}}{\bigcup}{S^{m}_\rho(\Xi)}^\circ$ is no longer closed
for the magnetic Moyal composition for a magnetic field that is no longer constant. In order to treat this difficulty we 
consider the formula
\eqref{mag-comp} that is a well defined composition in
$\underset{m\in\mathbb{R}}{\bigcup}{S^{m}_\rho(\Xi)}^\circ$ and extend it to the
following frozen magnetic product of $\phi$ and
$\psi$ in $\mathscr{S}(\Xi)$ (obtained by fixing the value of the magnetic field under the integrals at the point $\in \X$ where the product is computed):
\begin{align*}
\numberthis\label{trunc-mag-comp}\hspace{-1cm}
\big(\phi\, \natural^\epsilon\, \psi\big)(x,\xi)&
\\
&\hspace{-1cm}:=\left(\frac{1}{2\pi}
\right)^2\int_{\X^*}\int_{\X^*}
e^{-i(\eta\wedge \zeta))}\phi(x,\xi-\epsilon^{1/2}b_\kappa(x)\,
\eta)\psi(x,\xi-\epsilon^{1/2}b_\kappa(x)\,
\zeta)\,d^2\eta\,d^2\zeta\,.
\end{align*}

\begin{remark}
The proof of Proposition \ref{trunc-comp-N} remains true also for $N=0$ and for
symbols depending also on the configuration variable $x\in\X$ and shows that the
truncated magnetic product $\natural^\epsilon$ is a continuous map 
${S^{m}_\rho(\Xi)}\times {S^{p}_\rho(\Xi)}\rightarrow S^{m+p}_\rho(\Xi)$,
equicontinuous for $\epsilon\in[0,\epsilon_0]$.
\end{remark}
Using now the Taylor expansions up to some order $N\in\mathbb{N}^*$ we obtain
the following formulas similar to \eqref{f-B^01}:
\begin{align}\label{f-B^01-sv}
\big(\phi\,\natural^{\epsilon}\psi\big)(x,\xi) & =\phi(x,\xi)\psi(x,\xi)\nonumber 
\\
& \quad  +\underset{1\leq p\leq N-1}{\sum}(-2i\epsilon)^pB_\kappa(
x)^{p}\underset{|\alpha|=p}{\sum}(\alpha!)^{-1}
\big(\partial^\alpha_\xi\phi\big)(x,\xi) \,
(\partial^\alpha_{\xi}\psi\big)(x,\xi)\,\nonumber \\ 
& \quad +(2\pi)^{-2}(-2i\epsilon)^NB_\kappa(x)^N  r_N(\phi,\psi, B_\kappa (x),\epsilon) (x,\xi)\,,
\end{align}
where
\begin{align}\label{reste2}
r_N (\phi,\psi, B,\epsilon)(x,\xi) &:=\underset{|\alpha|=N}{\sum}(\alpha!)^{-1}
\,\int_{\X^*}\int_{\X^*}
e^{i(\eta\wedge \zeta)}\int_0^1s^{N-1}ds\,\times \nonumber \\ 
& \qquad \quad \quad  \times \big(\partial^\alpha_\xi\phi\,\big)(x,\xi-s\epsilon^{1/2}B^\frac 12 \eta)\,\big(\partial^\alpha_{\xi}\psi\big)(x,\xi- s \epsilon^{
1/2 }B^\frac 12   \zeta) \,d^2\eta\,d^2\zeta\,.
\end{align}

\paragraph{Some basic estimates.}

\begin{proposition}\label{P-trunc-mag-comp}
For a  magnetic field $B_{\epsilon,\kappa}(x)$ given in \eqref{Bek},
there exists  $\epsilon_0>0$ and $M>0$ such that for any
${(\epsilon,\kappa,\tau)\in[0,\epsilon_0]\times[0,1]}\times[1,\infty)$ with 
$\epsilon\tau^2\in[0,M]$ and for any $\phi\in S^m_\rho(\Xi)$ and $\psi\in
S^p_\rho(\Xi)\,$, 
\begin{align*}
\phi_{(\epsilon,\tau)}\, \sharp^{B_{\epsilon,\kappa}} \, \psi_{(\epsilon,\tau)}\,
& =\,\big(\phi
\,\natural^{\epsilon \tau^2}\,\psi\big)_{(\epsilon,\tau)}\,\\
& \quad +\,\epsilon\tau\,\frac{i}{
2}\underset{1\leq 
\ell \leq2}{\sum}\left[\Big(\big(\partial_{x_\ell }\phi\big)
\,\natural^{\epsilon\tau^2}\,\big(\partial_{
\xi_\ell }\psi\big)\Big)_{(\epsilon,\tau)}-\Big(\big(\partial_{\xi_\ell }\phi\big)
\,\natural^{\epsilon\tau^2}\,\big(\partial_{
x_\ell }\psi\big)\Big)_{(\epsilon,\tau)}\right]\,\\ & \quad  +
\,(\epsilon\tau)^2\mathscr{R}_{\epsilon,\kappa,\tau}(\phi,\psi)_{ (\epsilon,
\tau)}\,,
\end{align*}
where  the family $\left\{\mathscr{R}_{\epsilon,\kappa,\tau}(\phi,\psi)\right\}$
is in a bounded subset of $S^{m+p-3\rho}_\rho(\Xi)$. \\
Moreover, the application:
$
\mathscr{R}_{\epsilon,\kappa,\tau}:\ S^m_\rho(\Xi)\times
S^p_\rho(\Xi)\,\rightarrow\,S^{m+p-3\rho}_\rho(\Xi)
$
 is continuous.
\end{proposition}
\begin{proof}
\sloppy
With $B_{\epsilon,\kappa}$ and $\epsilon$ as in the statement of the proposition,
using the notation in \eqref{def-Psi} and  Remark \ref{R-def-Theta}, we
compute
\begin{equation}\label{T-2}
\begin{array}{l}
\hspace{-2cm}  \left(\phi_{(\epsilon,\tau)} \,\sharp^{B_{\epsilon,\kappa}} \, \psi_{(\epsilon,\tau)}\right)(X) \\
 =
\int\limits_{\Xi\times\Xi}\,dYdZ\,e^{-2i\sigma(Y,Z)}e^{ i \,
\epsilon\Phi_\kappa(\epsilon x,y,z)}
\phi_{(\epsilon,\tau)}(X-Y)\psi_{(\epsilon,\tau)}(X-Z)\times \nonumber\\
\hspace{+5cm}  \times\left(1\,+ \,i\,\kappa\epsilon^2\Psi^{\epsilon}(\epsilon
x,y,z)\int_0^1\Theta^{\epsilon,\kappa}_t(\epsilon x,y,z)dt\right).
\end{array}
\end{equation}
 We notice that $\phi_{(\epsilon,\tau)}(X-Y)\psi_{(\epsilon,\tau)}(X-Z)
=\phi\big(\epsilon(x-y),
\tau(\xi-\eta)\big)\psi\big(\epsilon(x-z),\tau(\xi-\zeta)\big)$ and use a
Taylor expansion up to second order with respect to the variables $\epsilon y$
and $\epsilon z$ in order to obtain:
\begin{equation*}
\begin{array}{l}
\phi_{(\epsilon,\tau)}(X-Y)\psi_{(\epsilon,\tau)}(X-Z)\\
\qquad =
\phi\big(\epsilon x,
\tau(\xi-\eta)\big)\psi\big(\epsilon
x,\tau(\xi-\zeta)\big) \\ 
\qquad\quad  -\,\epsilon\int_0^1dt\, \phi\big(\epsilon x,
\tau(\xi-\eta)\big)\Big(z\cdot\big(\partial_x\psi\big)\big(\epsilon x-t\epsilon
z,
\tau(\xi-\zeta)\big)\Big) \\
\qquad\quad   -\,\epsilon\int_0^1dt\, \Big(y\cdot\big(\partial_x\phi\big)\big(\epsilon
x-t\epsilon y,
\tau(\xi-\eta)\big)\Big)\psi\big(\epsilon x,
\tau(\xi-\zeta)\big)
\\
\qquad\quad   +\,\epsilon^2\int\limits_0^1dt\int\limits_0^1ds
\Big(y\cdot\big(\partial_x\phi\big)\big(\epsilon
x-t\epsilon
y,\tau\xi-\tau\eta\big)\Big)\Big(z\cdot\big(\partial_x\psi\big)\big(\epsilon
x-s\epsilon z,\tau\xi-\tau\zeta\big)\Big)\\ 
\qquad  =
\phi\big(\epsilon x,\tau\xi-\tau\eta\big)\psi\big(\epsilon
x,\tau\xi-\tau\zeta\big)\, \\
 \qquad \quad  -
\epsilon\left[\phi\big(\epsilon x,
\tau\xi-\tau\eta\big)\Big(z\cdot\big(\partial_x\psi\big)\big(\epsilon x,
\tau\xi-\tau\zeta\big)\Big)+\Big(y\cdot\big(\partial_x\phi\big)\big(\epsilon x,
\tau\xi-\tau\eta\big)\Big)\psi\big(\epsilon x,
\tau\xi-\tau\zeta\big)\right]\\ 
\qquad \quad +\,\epsilon^2\left\{\;\Big(y\cdot\big(\partial_x\phi\big)\,\big(\epsilon x,
\tau\xi-\tau\eta\big)\Big)\Big(z\cdot\big(\partial_x\psi\big)\big(\epsilon x,
\tau\xi-\tau\zeta\big)\Big) \right. \\  
\qquad \qquad  \qquad +\int\limits_0^1dt\Big[
\phi\big(\epsilon x,
\tau\xi-\tau\eta\big)\left\langle
z,\big(D^2_x\psi\big)\big(\epsilon(x-tz),\tau\xi-\tau\zeta\big)\,z\right\rangle
 \left. \right. \\
  \qquad \qquad \qquad \qquad\qquad \quad +\left\langle
y,\big(D^2_x\phi\big)\big(\epsilon(x-ty),\tau\xi-\tau\eta\big)\,y\right\rangle
)\psi\big(\epsilon x,\tau\xi-\tau\zeta\big)\Big]  \\ 
\qquad \qquad \qquad 
\left.-\epsilon\int\limits_0^1dt\left[
\Big(y\cdot\big(\partial_x\phi\big)\big(\epsilon x,
\tau\xi-\tau\eta\big)\Big)\left\langle
z,\big(D^2_x\psi\big)\big(\epsilon(x-tz),\tau\xi-\tau\zeta\big)\,z\right\rangle
\right.\right.\\ 
\qquad \qquad \qquad \qquad\qquad \quad 
\left.\left.+\left\langle
y,\big(D^2_x\phi\big)\big(\epsilon(x-ty),\tau\xi-\tau\eta\big)\,y\right\rangle
\Big(z\cdot\big(\partial_x\psi\big)\big(\epsilon x,
\tau\xi-\tau\zeta\big)\Big)\right]\right.\\ 
\qquad \qquad  \qquad  +
\left. \epsilon^2\int\limits_0^1dt\int\limits_0^1ds\left\langle
y,\big(D^2_x\phi\big)\big(\epsilon(x-ty),\tau\xi-\tau\eta\big)\,y\right\rangle
\left\langle
z,\big(D^2_x\psi\big)\big(\epsilon(x-sz),\tau\xi-\tau\zeta\big)\,z\right\rangle
\right\}.
\end{array}
\end{equation*}
We make now the usual integration by parts using the exponential factor
$e^{-2i\sigma(Y,Z)}$ and use formula \eqref{mag-comp}. We note that a
change of variables allows to write
\begin{align*}
&\hspace{-2cm} \int\limits_{\Xi\times\Xi}\,dYdZ\,e^{-2i\sigma(Y,Z)}e^{ i \,
\epsilon\Phi_\kappa(\epsilon x,y,z)}\phi\big(\epsilon
x,\tau\xi-\tau\eta\big)\psi\big(\epsilon
x,\tau\xi-\tau\zeta\big)\\
&\qquad =\ \int\limits_{\Xi\times\Xi}\,dYdZ\,e^{-2i\sigma(Y,Z)}e^{ i \, 
\epsilon\tau^2\Phi_\kappa(\epsilon x,y,z)}\phi\big(\epsilon
x,\tau\xi-\eta\big)\psi\big(\epsilon
x,\tau\xi-\zeta\big)\\
&\qquad =\
 \big(\phi\,\,\natural^{\epsilon\tau^2}\,\psi\big)_{(\epsilon,\tau)} (x,\xi) \,,
\end{align*}
 and the terms of order 0 and 1 in $\epsilon$ give us the first three terms
of the formula stated in the proposition. The remaining terms of order
2 in $\epsilon$ may be also integrated by parts using the exponential factor
$e^{-2i\sigma(Y,Z)}$ in order to obtain:
\begin{align*}
\int\limits_{\Xi\times\Xi}\,dYdZ\,&e^{-2i\sigma(Y,Z)}e^{i\,
\epsilon\Phi_\kappa(\epsilon x,y,z)}
\phi_{(\epsilon,\tau)}(X-Y)\psi_{(\epsilon,\tau)}(X-Z)\\
&=\big(\phi
\,\natural^{\epsilon \tau^2}\,\psi\big)_{(\epsilon,\tau)}\,\\
& \quad +\,\epsilon\tau\,\frac{i}{
2}\underset{1\leq 
\ell \leq2}{\sum}\left[\Big(\big(\partial_{x_\ell }\phi\big)
\,\natural^{\epsilon\tau^2}\,\big(\partial_{
\xi_\ell }\psi\big)\Big)_{(\epsilon,\tau)}-\Big(\big(\partial_{\xi_\ell
}\phi\big)
\,\natural^{\epsilon\tau^2}\,\big(\partial_{
x_\ell }\psi\big)\Big)_{(\epsilon,\tau)}\right]\,\\ 
&\quad +
\,(\epsilon\tau)^2\int\limits_{\Xi\times\Xi}\,dYdZ\,e^{-2i\sigma(Y,Z)}e^{i\,
\epsilon\Phi_\kappa(\epsilon
x,y,z)}\, \mathcal{R}^{(1)}_{\epsilon,\kappa,\tau}(X,Y, Z,\phi,\psi)\,,
\end{align*}
with
\begin{equation*}
\begin{array}{l}\hspace{-0.5cm}
\mathcal{R}^{(1)}_{\epsilon,\kappa,\tau}(X,Y, Z,\phi,\psi) \\
\quad  :=  \frac{1}{4}\underset{j,k}{\sum}
\big(\partial_{x_j}\partial_{\xi_k}\phi\big)(x,\xi-\eta)\big(\partial_{
x_k}\partial_{\xi_j}\psi\big)(x,\xi-\zeta) \\  
\qquad -\,\frac{1}{4}\int\limits_0^1dt \Big({\underset{k,\ell}{\sum}
\big(\partial^2_{\xi_k\xi_\ell }\phi\big)(x,\xi-\eta)\big(\partial^2
_{x_kx_\ell }\psi\big)(x-t\epsilon\tau z,\xi-\zeta) } \\ 
\hspace{4cm}  -\, {\underset{j,k}{\sum}
\big(\partial^2_{x_jx_k}\phi\big)(x-t\epsilon\tau y,
\xi-\eta) \big(
\partial^2_{\xi_j\xi_k}\psi\big)(x,\xi-\zeta)}\Big)\,\\
\qquad 
\left.+\,\frac{i\epsilon\tau}{8}\int\limits_0^1dt\Big(\underset{j,k,\ell }{\sum}
\big(\partial_{x_j}\partial^2_{\xi_k\xi_\ell
}\phi\big)(x,\xi-\eta)\big(\partial^2
_{x_kx_\ell }\partial_{\xi_j}\psi\big)(x-t\epsilon\tau z,\xi-\zeta)\right. \\ 
\hspace{4cm}  -\ \underset{j,k,\ell }{\sum}
\big(\partial^2_{x_jx_k}\partial_{\xi_\ell}\phi\big)(x-t\epsilon\tau y,
\xi-\eta)\big(\partial_{x_\ell }
\partial^2_{\xi_j\xi_k}\psi\big)(x,\xi-\zeta)\Big)\\  
\qquad
+\,\frac{\epsilon^2\tau^2}{16}\int\limits_0^1dt\int\limits_0^1ds\underset{j,k,
\ell ,
m}
{\sum}\big(\partial^2_{x_jx_k}\partial^2_{\xi_\ell
\xi_m}\phi\big)(x-t\epsilon\tau
y,
\xi-\eta)\big(\partial^2_{x_\ell x_m}
\partial^2_{\xi_j\xi_k}\psi\big)(x-s\epsilon\tau z,\xi-\zeta) \\ 
\end{array}
\end{equation*}
Thus the remainder in the statement of the proposition is explicitely given by:
\beq\label{rest-quadratic}
\mathscr{R}_{\epsilon,\kappa,\tau}(\phi,\psi) (X) 
=\ \int\limits_{\Xi\times\Xi} \, dYdZ\, e^{-2i\sigma(Y,Z)}e^{ i \, 
\epsilon\tau^2\Phi_\kappa(x,y,z)}\,\mathcal{\widehat
R_{\epsilon,\kappa,\tau}}(X,Y,Z,\phi,\psi)\,,  
\eeq
with 
\begin{equation*}
\begin{array}{l}
\mathcal {\widehat R_{\epsilon,\kappa,\tau}}(X,Y, Z,\phi,\psi)=
\mathcal{R}^{(1)}_{\epsilon,\kappa,\tau}(X,Y, Z,\phi,\psi)
-\,\kappa\left(\int_0^1\Theta^{\epsilon,\kappa}_\tau(x,\tau y,\tau
z)dt \right)\ \times\\  
\qquad \quad \times\,\left(\underset{\ell }{\sum}R_1\big(\partial_\ell B\big)(x,\epsilon
y,\epsilon z)\big[\big(\nabla_{\xi}\phi\big)(x-\epsilon\tau y,
\xi-\eta)\big]\wedge\big[\big(\nabla_\xi\partial_{\xi_\ell}
\psi\big)(x-\epsilon\tau z,\xi-\zeta)\big]  \right.\\
\qquad \qquad \quad  \left.+\,\underset{\ell }{\sum}R_2\big(\partial_\ell B\big)(x,\epsilon\tau
y,\epsilon\tau z)\big[\big(\nabla_\xi\partial_{\xi_\ell }\phi\big)(x-\epsilon\tau y,
\xi-\eta)\big]\wedge\big[\big(\nabla_{\xi}\psi\big)(x-\epsilon\tau
z,\xi-\zeta)\big]\right)\,,
\end{array}
\end{equation*}
with $R_1(\cdot )$ and $R_2(\cdot )$ defined by \eqref{def-R1} and
\eqref{def-R2}.
All these terms are of the form considered in Proposition
\ref{comp-epsilon-symb}.
\end{proof}

\begin{remark}\label{N-dev}
Starting above with a Taylor
expansion of order $N\in\mathbb{N}$, one can prove  that for any
$N\in\mathbb{N}^*\,$,  there exist $\epsilon_0 >0$  and constants $C_j>0$ for $1\leq j\leq
N-1\,$, such that, for any $\phi$, $\psi$, 
\begin{align*}
\phi_{(\epsilon,1)}\, \sharp^{B_{\epsilon,\kappa}} \psi_{(\epsilon,1)}\,
& =\,\big(\phi\, \natural^\epsilon\psi\big)_{(\epsilon,1)}\, \\ & 
\quad +\,\underset{1\leq j\leq
N-1}{\sum}C_j (i \epsilon)^j\underset{|\gamma|=j}{\sum}\left[\Big(\big(\partial_
{x}^\gamma\phi\big)
\, \natural^\epsilon\big(\partial_{
\xi}^\gamma\psi\big)\Big)_{(\epsilon,1)}+(-1)^j\Big(\big(\partial_{\xi}
^\gamma\phi\big)\, \natural^\epsilon
\big(\partial_{x}^\gamma\psi\big)\Big)_{(\epsilon,1)}\right]\, \\ &\quad  +\,\epsilon^N  \mathscr{R}^N_{\epsilon,\kappa}(\phi,\psi)\, , 
\end{align*}
where $\mathscr{R}^N_{\epsilon,\kappa}(\phi,\psi)$ is bounded 
uniformly in $\mathscr{S}(\Xi)$ for $(\epsilon,\kappa)\in[0,\epsilon_0]\times[0,1]\,$.
\end{remark}
\begin{corollary} \label{abovecor}
If $B_{\epsilon,\kappa}(x)$ satisfies \eqref{Bek}, there exists $\epsilon_0 >0$ such that, for any $\epsilon\in[0,\epsilon_0]$, for any  $F^\bullet\in
S^m_\rho(\Xi)^\bullet$ and $G^\bullet\in
S^p_\rho(\Xi)^\bullet$, we have, with  
 the notation  from Remark \ref{rem-sl-var-class}\,, 
\begin{align*}
&F^\epsilon\,  \sharp^{B_{\epsilon,\kappa}} G^\epsilon \\
&=\,\big(\widetilde{F}^\epsilon\,
\natural^\epsilon\widetilde{G}^\epsilon\big)_{(\epsilon,1)}\,+\,\frac{
i \epsilon}{2}\underset{1\leq
\ell \leq2}{\sum}\left(\big(\partial_{x_\ell }\widetilde{F}^\epsilon\,
\natural^\epsilon \, \partial_{
\xi_\ell
}\widetilde{G}^\epsilon\big)_{(\epsilon,1)}-\big(\partial_{\xi_\ell}\widetilde
{F}^\epsilon\,\natural^\epsilon\, \partial_{
x_\ell }\widetilde{G}^\epsilon\big)_{(\epsilon,1)}\right)\,+
\,\epsilon^2\mathscr{R}_{\epsilon,\kappa}(\widetilde{F}^\epsilon,\widetilde{G}
^\epsilon)_{(\epsilon,1)}\,,
\end{align*}
where the family 
$\{\mathscr{R}_{\epsilon,\kappa}(\widetilde{F}^\epsilon,\widetilde{G}
^\epsilon)\}_{\epsilon\in[0,\epsilon_0], \kappa \in [0,1]}$ is a bounded set in  $S^{m+p-3\rho}_\rho(\Xi)\,$.
\end{corollary}

\begin{proposition}\label{P-comm-sv}
 If $B_{\epsilon,\kappa}(x)$ satisfies \eqref{Bek},  there exists $\epsilon_0 >0$ such that
 for any $\epsilon\in[0,\epsilon_0]$,  any $F\in S^m_\rho(\Xi)^\circ$ and $G\in
S^p_\rho(\Xi)^\circ$, if we define
$$ \big[F,G\big]_{B_{\epsilon,\kappa}}:= F \, 
\sharp^{B_{\epsilon,\kappa}} G\,-\,G
 \, \sharp^{B_{\epsilon,\kappa}} F\,,$$
we have the expansion
\begin{align*}
\big[F,G\big]_{B_{\epsilon,\kappa}} (x,\xi) & 
= -4i\epsilon B_\kappa(\epsilon
x)\left[\big(\partial_{\xi_1}F\big)(\xi)\big(\partial_{\xi_2}
G\big)(\xi)-\big(\partial_{\xi_2}F\big)(\xi)\big(\partial_{\xi_1}
G\big)(\xi)\right]\\
&\qquad +\,\epsilon^2\widetilde{\mathscr{R}}_{\epsilon,\kappa}(F,G){ _
{(\epsilon,1)}} (x,\xi)\,,
\end{align*}
where the family 
$\{\widetilde{\mathscr{R}}_{\epsilon,\kappa}(F,G)\}_{\epsilon\in[0,
\epsilon_0], \kappa \in [0,1]}$ is a bounded set in  $S^{m+p-3\rho}_\rho(\Xi)$. \\
More explicitly we have 
\begin{equation*}
\begin{array}{l}
\widetilde{\mathscr{R}}_{\epsilon,\kappa}(F,G) (x,\xi) \\
\quad =
-\frac{1}{\pi^2}B^2_\kappa(x)\underset{|\alpha|=2}{\sum}
(\alpha!)^{-1}
\, \times \\ 
\qquad \qquad  \times \int_{\X^*}\int_{\X^*}
e^{i(\eta\cdot\zeta^\bot)}\int_0^1s\,
ds\big(\partial^\alpha_\xi F\big)(\xi-s\epsilon^{1/2}b_\kappa(
x)\eta)\big(\partial^\alpha_{\xi^\bot}G\big)(\xi-\epsilon^{1/2}
b_\kappa(x)\zeta)\,
d^2\eta\,d^2\zeta\,\\
 \qquad -\,\kappa \int\limits_{\Xi\times\Xi}\, dYdZ\,  e^{-2i\sigma(Y,Z)} \left(\int_0^1\Theta^{\epsilon,\kappa}_\tau (x,y,z) \, d\tau\right)\times \\
\qquad  \qquad \qquad  \times 
\left(\underset{\ell}{\sum}R_1\big(\partial_\ell B\big)(x,\epsilon
y,\epsilon z)\big[\big(\nabla_{\xi}F\big)(
\xi-\eta)\big]\wedge\big[\big(\nabla_\xi\partial_{\xi_\ell}
G\big)(\xi-\zeta)\big]\, \right.  \\ 
 \qquad \qquad  \qquad \qquad \qquad  \left.+\,\underset{\ell }{\sum}R_2\big(\partial_\ell
B\big)(x,\epsilon
y,\epsilon z)\big[\big(\nabla_\xi\partial_{\xi_\ell}F\big)(
\xi-\eta)\big]\wedge\big[\big(\nabla_{\xi}G\big)(\xi-\zeta)\big]\right)\,,
\end{array}
\end{equation*}
with $R_1(\cdot )$ and $R_2(\cdot )$ defined by \eqref{def-R1} and \eqref{def-R2}.
\end{proposition}

\begin{proof}
Using Corollary \ref{abovecor}, and taking into account that both symbols do not
depend on the configuration space variable $x\in\X$, we obtain:
$$
F\, \sharp^{B_{\epsilon,\kappa}} \, G\ =\ F\, \natural^\epsilon\,
G\,+\,\epsilon^2 \, \widetilde{\mathscr{R}}_{\epsilon,\kappa}(F,G)_{ (\epsilon,1)}\,.
$$
Using one of the two formulas for $r_N$ in  \eqref{f-B^01-sv} with $N=2$ we
obtain the desired result.
\end{proof}

\begin{proposition}\label{P-disj-supp-sv}
If  $B_{\epsilon,\kappa}(x)$ satisfies \eqref{Bek}, there exists
$\epsilon_0 >0$ such that  for any $\epsilon\in[0,\epsilon_0]$ and
any  $F\in S^m_\rho(\Xi)^\circ$ and $G\in
S^p_\rho(\Xi)^\circ$ with disjoint supports, 
$$
F \, \sharp^{B_{\epsilon,\kappa}}\, G\
=\ \epsilon^2 \, \widetilde{\mathscr{R}}_{\epsilon,\kappa}(F,G)_{(\epsilon,1)}\,,
$$
where
$\{\widetilde{\mathscr{R}}_{\epsilon,\kappa}(F,G)\}_{\epsilon\in[0,
\epsilon_0]}$ belongs to $S^{m+p-3\rho}_\rho(\Xi)$
uniformly with respect to $(\epsilon,\kappa)\in[0,\epsilon_0]\times[0,1]$. Explicitly:
\begin{equation*}
\begin{array}{l}
\widetilde{\mathscr{R}}_{\epsilon,\kappa}(F,G) (x,\xi) \\
\quad = 
-\frac{1}{\pi^2}B^2_\kappa(x)\underset{|\alpha|=2}{\sum}
(\alpha!)^{-1}\, \times \\
\qquad \qquad \times \int_{\X^*}\int_{\X^*}
e^{i(\eta\cdot\zeta^\bot)}\int_0^1s\,
ds\big(\partial^\alpha_\xi F\big)(\xi-s\epsilon^{1/2}b_\kappa(
x)\eta) \big(\partial^\alpha_{\xi^\bot}G\big)(\xi-\epsilon^{1/2}
b_\kappa(x)\zeta)\,
d^2\eta\,d^2\zeta\,\\
\quad \quad   -\,\kappa \int\limits_{\Xi\times\Xi} \,dYdZ\, e^{-2i\sigma(Y,Z)} \left(\int_0^1\Theta^{\epsilon,\kappa}_\tau(x,y,z)d\tau\right) \times \\
\qquad \qquad \qquad \qquad \times 
\left(\underset{\ell }{\sum}R_1\big(\partial_\ell B\big)(x,\epsilon
y,\epsilon z) \big[\big(\nabla_{\xi}F\big)(
\xi-\eta)\big]\wedge\big[\big(\nabla_\xi\partial_{\xi_\ell}
G\big)(\xi-\zeta)\big]\, \right. \\
\qquad\qquad\qquad \qquad\qquad \quad  \left. +\,\underset{\ell}{\sum}R_2\big(\partial_\ell 
B\big)(x,\epsilon
y,\epsilon z)\big[\big(\nabla_\xi\partial_{\xi_\ell }F\big)(
\xi-\eta)\big]\wedge\big[\big(\nabla_{\xi}G\big)(\xi-\zeta)\big]\right)\,,
\end{array}
\end{equation*}
with $R_1(\cdot)$ and $R_2(\cdot )$ defined by \eqref{def-R1} and \eqref{def-R2}.
\end{proposition}
The proof is quite similar to the one of  Proposition \ref{P-comm-sv}.

\begin{proposition}\label{trunc-comp-N}
Given a
magnetic field of the form \eqref{Bek}, for any $N\in\mathbb{N}^*$ and any
$\epsilon\in[0,\epsilon_0]$ there exists
a family of continuous bilinear maps
$$
\mathscr{M}^{\epsilon,\kappa}_N:S^m_\rho(\Xi)^\circ\times
S^p_\rho(\Xi)^\circ\rightarrow S^{m+p}_\rho(\Xi)
$$
for any $m\in\mathbb{R}$, $p\in\mathbb{R}$ and $\rho\in[0,1]$, that are
uniformly bounded for $\epsilon\in[0,\epsilon_0]$ and such that 
the frozen magnetic product
\eqref{trunc-mag-comp} satisfies the following relation:
\begin{align*}
\big(\phi\,\natural^\epsilon\,\psi\big)(x,\xi) &
=\phi(\xi)\psi(\xi)+\hspace{-0.3cm}\underset{1\leq p\leq
N-1}{\sum}\hspace{-0.3cm}(-2i\epsilon)^pB^\kappa(
x)\big)^{p}\underset{|\alpha|=p}{\sum}(\alpha!)^{-1}
\big(\partial^\alpha_\xi\phi\big)(\xi)\big(\partial^\alpha_{\xi^\bot}
\psi\big)(\xi) \\ & \quad 
+\,\epsilon^NB^\kappa(
x)\big)^N\underset{|\alpha|=N}{\sum}(\alpha!)^{-1}\mathscr{M}^{\epsilon,\kappa}
_N\big(\partial^\alpha_\xi\phi,\partial^\alpha_{\xi^\bot}\psi\big)\big)\,.
\end{align*}
\end{proposition}
\begin{proof}
From  formula \eqref{trunc-mag-comp}, after a Taylor expansion up to order
$N\in\mathbb{N}^*$ and the usual integration by parts argument (using the
exponential factor $\exp\{-2i\sigma(Y,Z)\}$) we obtain:
\begin{align*}
\big(\phi\,\natural^\epsilon\,\psi\big)(x,\xi)& 
= \phi(\xi)\psi(\xi)+\hspace{-0.3cm}\underset{1\leq p\leq
N-1}{\sum}\hspace{-0.3cm}(-2i\epsilon)^pB^\kappa(x)^{p}\underset{|\alpha|=p}{\sum}
(\alpha!)^{-1}
\big(\partial^\alpha_\xi\phi\big)(\xi)\big(\partial^\alpha_{\xi^\bot}
\psi\big)(\xi) \\ & \,
\quad +\left(\frac{1}{2\pi}\right)^2(-2i\epsilon)^NB^\kappa(
x)^N\underset{|\alpha|=N}{\sum}(\alpha!)^{-1}\int_0^1s^{N-1}\,ds
\int_{\X^*}\int_{\X^*}e^{i(\eta\cdot\zeta^\bot)}\\ &\quad \qquad \qquad
\times\big(\partial^\alpha_\xi\phi\big)(\xi-s\epsilon^{1/2}b_\kappa(
x)\eta)\big(\partial^\alpha_{\xi^\bot}\psi\big)(\xi-\epsilon^{1/2}
b_\kappa(\epsilon
x)\zeta)\,d^2\eta\,d^2\zeta\,.
\end{align*}
Thus, if we define
\begin{align}\label{def-M}
\mathscr{M}^{\epsilon,\kappa}_N\big(f,g\big)(x,\xi)& :=\frac{(-2i)^N}{(2\pi)^2}
\int_0^1s^{N-1}\,ds\ \times \nonumber \\ &
\hspace{2cm} \times\ \int_{\X^*}\int_{\X^*}e^{i(\eta\cdot\zeta^\bot)}
f(\xi-s\epsilon^{1/2}b_\kappa(x)\eta)\ 
g(\xi-\epsilon^{1/2}b_\kappa(x)\zeta)\,d^2\eta\,d^2\zeta\,,
\end{align}
we notice that
$
\big(\partial_x^\alpha\partial_\xi^\beta\mathscr{M}^{\epsilon,\kappa}_N(f,
g)\big)(x,\xi)
$
is a finite sum of terms of the form
\begin{equation*}
\begin{array}{l}
\hspace{-0.5cm} F^B_{\gamma_1}(x)F^B_{\gamma_2}(x)\int_0^1s^{N-1+|\gamma_1|}\,ds
\int_{\X^*}\int_{\X^*}e^{i(\eta\cdot\zeta^\bot)}\eta^{\gamma_1}\zeta^{\gamma_2}
\, d^2\eta\, d^2\zeta\, \\ 
\hspace{4cm}
\times\,
\big(\partial_\xi^{\gamma_1+\beta_1}f\big)(\xi-s\epsilon^{1/2}b_\kappa(
x)\eta)\ 
\big(\partial_\xi^{\gamma_2+\beta_2}g\big)(\xi-\epsilon^{1/2}b_\kappa(
x)\zeta)\\ 
\quad =
(-i)^{|\gamma+\gamma_2|}\epsilon^{|\gamma_1+\gamma_2|/2}
\widetilde{F}^B_{\gamma_1,\gamma_2}(x)\int_0^1s^{N-1+|\gamma_1|}\,ds
\int_{\X^*}\int_{\X^*}e^{i(\eta\cdot\zeta^\bot)}\, d^2\eta\, d^2\zeta\\  
\hspace{4cm}
\times\,
\big(\partial_{\xi^\bot}^{\gamma_2}\partial_\xi^{\gamma_1+\beta_1}
f\big)(\xi-s\epsilon^ { 1/2 } b_\kappa(
x)\eta)\ 
\big(\partial_{\xi^\bot}^{\gamma_1}\partial_\xi^{\gamma_2+\beta_2}
g\big)(\xi-\epsilon^{1/2}b_\kappa(
x)\zeta)\,,
\end{array}
\end{equation*}
with $|\gamma_1+\gamma_2|\leq|\alpha|$ and $\beta_1+\beta_2=\beta$ and the
functions $F^B_\gamma$ and $\widetilde{F}^B_{\gamma_1,\gamma_2}$ of class
$BC^\infty(\X)\,$.\\

Then we have the following equality:

\begin{equation*}
\begin{array}{l}
\hspace{-1.5cm}\int_{\X^*}\int_{\X^*}e^{i(\eta\cdot\zeta^\bot)}
\big(\partial^\alpha_\xi\phi\big)(\xi-s\epsilon^{1/2}b_\kappa(
x)\eta)\big(\partial^\beta_{\xi}\psi\big)(\xi-\epsilon^{1/2}
b_\kappa(x)\zeta)\,d^2\eta\,d^2\zeta \\ \quad 
=\int_{\X^*}<\eta>^{-N_1}
\underset{|a|\leq
N_2}{\sum}\mathcal{P}_a(\eta)\Big[\partial_\eta^{a}\big(\partial^\alpha_\xi
\phi\big)(\xi-s\epsilon^{1/2}b_\kappa(x)\eta)\Big]\,d^2\eta \times \\ 
\qquad\qquad \qquad 
\times\,\int_{\X^*}
e^{i(\eta\cdot\zeta^\bot)}<\zeta>^{-N_2}\big(1+\Delta_\zeta\big)^{N_1/2}
\Big[\big(\partial^\beta_{\xi}\psi\big)(\xi-\epsilon^{1/2}
b_\kappa(x)\zeta)\Big]\,d^2\zeta\,,
\end{array}
\end{equation*}
 where the  $\mathcal{P}_a(\cdot)$ are bounded functions on $\X^*\,$.
We notice that:
\beq\label{st-phase-1}
\begin{array}{l} 
\hspace{-2cm} \int_{\X^*}\int_{\X^*}
e^{i(\eta\cdot\zeta^\bot)}\phi(\xi-\tau\eta)\psi(\xi-\tau\zeta)\,
d^2\eta\,d^2\zeta \\ \qquad 
=\int_{\X^*}\big(1+\Delta_\eta\big)^{N_2/2}\big(<\eta>^{-N_1}
\phi(\xi-\tau\eta)\big)\,d^2\eta \\ \qquad \qquad \qquad 
\times\,\int_{\X^*}
e^{i(\eta\cdot\zeta^\bot)}<\zeta>^{-N_2}\big(\big(1+\Delta_\zeta\big)^{N_1/2}
\psi\big)(\xi-\tau\zeta)\,d^2\zeta\\ 
\qquad 
=\int_{\X^*}<\eta>^{-N_1}\underset{|a|\leq
N_2}{\sum}\mathcal{P}_a(\eta)\big(\partial_\eta^{a}\phi\big)(\xi-\tau
\eta)\,d^2\eta \\ \qquad \qquad \qquad 
\times\,\int_{\X^*}
e^{i(\eta\cdot\zeta^\bot)}<\zeta>^{-N_2}\big(\big(1+\Delta_\zeta\big)^{N_1/2}
\psi\big)(\xi-\tau\zeta)\,d^2\zeta\,,
\end{array}
\eeq
 where the  $\mathcal{P}_a(\cdot )$ are bounded functions on $\X^*$ and $N_1>2\,$,
$N_2>2\,$.

Finally these arguments allow to prove estimates of the following type:
$$
\nu^{m+p,\rho}_{n,m}\big(\mathscr{M}^{\epsilon,\kappa}_N(f,g)\big)\ \leq\
C\,  \nu^{m,\rho}_{0,q}(f)\,\nu^{p,\rho}_{0,q}(g)\,,
$$
with $q>2+n+m$ and some constant $C>0$ that may depend on $B$ and
on the three seminorms but not on $\epsilon\in[0,\epsilon_0]\,$.
\end{proof}

\subsection{Control of some $\Gamma_*$-indexed series of symbols}\label{AppB}
\label{series-conv}
Consider a
symbol $ \varphi\in S^{-\infty}(\Xi)$ and let us study the convergence of the
series
$$
\Phi:=\underset{\gamma^*\in\Gamma_*}{\sum}\tau_{\gamma^*}\big(\varphi\big)\,.
$$
For any $N\in\mathbb{N}$ let us  introduce
\begin{equation*}
\Gamma_*^N:=\left\{\gamma^*\in\Gamma_*\,\mid\,|\gamma^*|\leq N\right\}\,\mbox{ and }\, 
\Phi_N:=\underset{\gamma^*\in\Gamma_*^N}{\sum}\tau_{\gamma^*}\big(\varphi\big)\, .
\end{equation*}
\begin{lemma}\label{periodic-sums}
With the above notation, for any symbol $ \varphi\in S^{-\infty}(\Xi)$ the
sequence $ \{\Phi_N\}_{N\in\mathbb{N}}$ in $ S^{-\infty}(\Xi)$ is
weakly convergent in $\mathscr{S}^\prime(\Xi)$ and the limit $\Phi$ is
a $\Gamma_*$-periodic $C^\infty$-function on $\Xi\,$. The sequence also converges
for
the norms $\nu^{-m,0}_{p,q}$ with any $m>2$ and $(p,q)\in\mathbb{N}^2$
and for any $(p,q)\in\mathbb{N}^2$ there exists $C>0$ such that,   for all $\varphi \in  S^{-\infty}(\Xi)\,,$
$$\nu^{0,0}_{p,q}(\Phi_N)\leq C\, \nu^{0,0}_{p,q}(\varphi)\,, \,\forall
N\in\mathbb{N}\,.$$
\end{lemma}
\begin{proof}
We clearly have, for any $m\in\mathbb{R_+}\,$,
$$
\underset{(x,\xi)\in\Xi}{\sup}
<\xi>^m\left|\big(\partial_x^a\partial_\xi^b\varphi\big)(x,\xi)\right|\leq\nu^{
-m,0}_{|a|,|b|}(\varphi)\,,
$$ 
so that, for any $N_2\geq N_1$ in $\mathbb{N}\,$,  
$$ 
\begin{array}{ll}
|\big(\partial_x^a\partial_\xi^b\Phi_{N_1}\big)(x,
\xi)-\big(\partial_x^a\partial_\xi^b\Phi_{N_2}\big)(x,\xi)|\, & \leq\,\underset{
\gamma^*\in\Gamma_*^{N_2}
\setminus\Gamma_*^{N_1}}{\sum}\left|\tau_{\gamma^*}\big[
\big(\partial_x^a\partial_\xi^b\varphi\big)\big](x,\xi)\right| 
\\ & 
\leq\,\nu^{-m,0}_{|a|,|b|}(\varphi)\Big(\underset{\xi\in\X^*}{\sup}\underset{\gamma^*\in\Gamma_*^{N_2}
\setminus\Gamma_*^{N_1}}{\sum}<\xi+\gamma^*>^
{-m}\Big)\\ & \leq\,C\nu^{-m,0}_{|a|,|b|}(\varphi)<N_1>^{-s},
\end{array}
$$
for any $m>2+s$ with $s>0\,$. \\
Thus the weak convergence in
$\mathscr{S}^\prime(\V)$ follows easily and also  the other conclusions of the
lemma by some standard arguments.
\end{proof}

\begin{proposition}\label{P.sum}
With the notation  from Lemma \ref{periodic-sums}, suppose that we have a family
of symbols $ \{\varphi^\epsilon\}_{\epsilon\in[0,\epsilon_0]}\in
S^{-\infty}(\Xi)^\bullet$ and a magnetic field given by \eqref{Bek}. Then $ \{\Phi^\epsilon\}_{\epsilon\in[0,\epsilon_0]}\in
S^{0}_0(\Xi)^\bullet$ and the
sequence
$\{\mathfrak{Op}^{\epsilon,\kappa}(\Phi^\epsilon_N)\}_{N\in\mathbb{N}}$ converges
strongly to
$\mathfrak{Op}^{\epsilon,\kappa}(\Phi^\epsilon)$ in ${ \mathcal L}\big(L^2(\X)\big)$ uniformly for $(\epsilon,\kappa)\in[0,\epsilon_0]\times[0,\kappa_0]\,$.
Moreover, the application
$$
S^{-\infty}(\Xi)\ni\varphi^\epsilon\,\mapsto\,\Phi^\epsilon\in\big({
S^0_0(\Xi)},\|\cdot\|_{B_{\epsilon,\kappa}}\big)
$$
is continuous uniformly for $(\epsilon,\kappa)\in[0,\epsilon_0]\times[0,\kappa_0]\,$.
\end{proposition}
\begin{proof}
We write 
$$
\mathfrak{Op}^{\epsilon,\kappa}(\Phi_N)\,=\,\underset{|\gamma^*|\leq
N}{\sum}\mathfrak{Op}^{\epsilon,\kappa}\big(\tau_{\gamma^*}[\varphi^\epsilon]\big)\,,
$$
and we introduce the following simpler notation :
$$
X_{\gamma^*}:=\mathfrak{Op}^{\epsilon,\kappa}\big(\tau_{\gamma^*}[\varphi^\epsilon]\big)\,,
\quad
\widetilde{X}_\infty:=\mathfrak{Op}^{\epsilon,\kappa}(\Phi),\quad
\widetilde{X}_N:=\mathfrak{Op}^{\epsilon,\kappa}(\Phi_N)\,,
$$
so that $\widetilde{X}_N=\underset{\gamma\in\Gamma_*^N}{\sum}X_{\gamma^*}$ 
and we use the Cotlar-Stein Lemma. For this we need to verify some
estimates. Let us first consider the products (using also Remark
\ref{rem-sl-var-class})
\beq\label{CS-est}
\|X_{\beta^*}^*X_{\gamma^*}\|_{{ \mathcal
L}(\mathcal{H})}=\left\|\overline{\varphi^\epsilon}
 \,\sharp^{B_{\epsilon,
\kappa}}\,\big(\tau_{(\gamma^*-\beta^*)}\varphi^\epsilon\big)\right\|_{B_{\epsilon,
\kappa}}=\left\|\overline{\widetilde{\varphi}^\epsilon_{(\epsilon,1)}}
 \,\sharp^{B_{\epsilon,
\kappa}}\,\big(\tau_{(\gamma^*-\beta^*)}\widetilde{\varphi}^\epsilon_{
(\epsilon,1)}
\big)\right\|_{B_{\epsilon,\kappa}}\,,
\eeq
and use Proposition \ref{P-trunc-mag-comp} and \eqref{trunc-mag-comp} in order
to obtain that
$$
\begin{array}{l}
\overline{\widetilde{\varphi}^\epsilon}_{(\epsilon,1)}
\sharp^{B_{\epsilon,\kappa}}
\big(\tau_{(\gamma^*-\beta^*)}\widetilde{\varphi}^\epsilon
\big)_{(\epsilon,1)}  \\
 \quad 
=\,\overline{\widetilde{\varphi}^\epsilon}_{(\epsilon,1)}
\,\natural^\epsilon\,\big(\tau_{(\gamma^*-\beta^*)}\widetilde{\varphi}^\epsilon
\big)_{(\epsilon,1)} \\ \qquad 
+\,\epsilon\,\frac{i}{2}\underset{1\leq
l\leq2}{\sum}\left(\big(\partial_{x_l}
\overline{\widetilde{\varphi}^\epsilon}\big)_{(\epsilon,1)}
\,\natural^\epsilon\,\big(\tau_{(\gamma^*-\beta^*)}\big(\partial_{
\xi_l}\widetilde{\varphi}^\epsilon
\big)\big)_{(\epsilon,1)}-\big(\partial_{\xi_l}
\overline{\widetilde{\varphi}^\epsilon}\big)_{(\epsilon,1)}
\,\natural^\epsilon\,\big(\tau_{(\gamma^*-\beta^*)}\big(\partial_{
x_l}\widetilde{\varphi}^\epsilon
\big)\big)_{(\epsilon,1)}\right) \\ \qquad 
+\,\epsilon^2\mathscr{R}_{\epsilon,\kappa}\big(\overline{\widetilde{\varphi}
^\epsilon},\big(\tau_{(\gamma^*-\beta^*)}\widetilde{\varphi}^\epsilon
\big)\big)_{(\epsilon,1)}\,.
\end{array}
$$
For $\phi$ and $\psi$ in $ S^{-\infty}(\Xi)$ we can repeat the arguments in the
proof of Proposition \ref{trunc-comp-N} and obtain:
\begin{align}\label{25aug-1}
\big(\phi_{(\epsilon,1)}\,\natural^\epsilon\,\big(\tau_{\alpha^*}\psi\big)_{
(\epsilon,1)}\big)(x,\xi) & =
\,\left(\frac{1}{2\pi}
\right)^2\int_{\X^*}\int_{\X^*}
e^{-i(\eta^\bot\cdot\zeta))}\phi(\epsilon x,\xi-\epsilon^{1/2}b_\kappa(\epsilon
x)\,
\eta) \times \nonumber \\ &  \qquad\qquad  \qquad \qquad   \times  \psi(\epsilon x,\xi+\alpha^*-\epsilon^{1/2}b_\kappa(\epsilon x)\,
\zeta)\,d^2\eta\,d^2\zeta \nonumber \\ & 
=\int_{\X^*}<\eta>^{-N_1}
\underset{|a|\leq
N_2}{\sum}\mathcal{P}_a(\eta)\partial_\eta^{a}\phi\Big[
(\epsilon x,\xi-s\epsilon^{1/2}b_\kappa(\epsilon
x)\eta)\Big]\,d^2\eta\,  \times \nonumber \\ & 
\qquad \qquad \times\,\int_{\X^*}
e^{i(\eta\cdot\zeta^\bot)}<\zeta>^{-N_2}\, \times \nonumber \\ &\qquad \quad \qquad \times \big(1+\Delta_\zeta\big)^{N_1/2}
\Big[\psi(\epsilon x,\xi+\alpha^*-\epsilon^{1/2}
b_\kappa(\epsilon x)\zeta)\Big]\,d^2\zeta\,,
\end{align}
where the  $\mathcal{P}_a(\cdot)$ are bounded functions on $\X^*$ and $N_1=N_2>2$.

Thus we obtain, for any $\alpha^*\in\Gamma_*$ and for any $N\in\mathbb{N}$, the estimate:
\begin{align*}
&\hspace{-0.5cm} <\alpha^*>^N\,\nu^{0,0}_{0,0}\big(\phi_{(\epsilon,1)}\,\natural^\epsilon\,\big(\tau_
{\alpha^*}\psi\big)_
{(\epsilon,1)}\big)\\
&\qquad =\
\underset{(x,\xi)\in\Xi}{\sup}<\alpha^*>^N\left|\big(\phi_{(\epsilon,1)}
\,\natural^\epsilon\,\big(\tau_{\alpha^*}\psi\big)_{(\epsilon,1)}\big)(x,
\xi)\right|\ \\
&
\qquad 
\leq 
 C_N \, \left|\int_{\X^*}<\eta>^{-N_1}\left<\xi-s\epsilon^{1/2}b_\kappa(\epsilon
x)\eta\right>^N
\underset{|a| \leq
N_2}{\sum}\mathcal{P}_a(\eta)\partial_\eta^{a}\Big[\phi
(\epsilon x,\xi-s\epsilon^{1/2}b_\kappa(\epsilon
x)\eta)\Big]\,d^2\eta\,\times\right.
\\ &\hspace{3cm}
\left.\times\,\int_{\X^*}
e^{i(\eta\cdot\zeta^\bot)}<\zeta>^{-N_2}\left<\xi+\alpha^*-s\epsilon^{1/2}
b_\kappa(\epsilon
x)\eta\right>^N \times  \right.\\ &\hspace{6cm}\left.  \times  
\big(1+\Delta_\zeta\big)^{N_1/2}
\Big[\psi(\epsilon x,\xi+\alpha^*-\epsilon^{1/2}
b_\kappa(\epsilon x)\zeta)\Big]\,d^2\zeta\right| \\
&
\qquad \leq\ C_N \, \nu^{-N,0}_{0,0}\big(\phi\big)\,\nu^{-N,0}_{0,0}\big(\psi\big)\,.
\end{align*}
Noticing that
$\epsilon^{-|a|}\partial_x^a\partial_\xi^b\big(\phi_{(\epsilon,1)}
\,\natural^\epsilon\,\big(
\tau_{\alpha^*}\psi\big)_{(\epsilon,1)}\big)$ is an oscillating integral of the
same type as the expression \eqref{25aug-1} with $\phi$ and $\psi$ replaced by
some derivatives of them, we conclude that we have similar estimates for all the
seminorms defining the Fr\'{e}chet topology. Choosing now $N>2$ large enough,  we
verify the hypothesis of the Cotlar-Stein Lemma and obtain the conclusion of the
proposition.
\end{proof}

\end{document}